\documentclass[12pt,a4paper,reqno]{amsart}
\usepackage{preemble}
\usepackage{tikz,enumerate,mathrsfs}
\usepackage[toc,page]{appendix}

\usepackage[margin=2cm]{geometry}

\usepackage{setspace}

\usepackage{fancyhdr}
\usepackage[dvipsnames]{xcolor}
\usepackage{tcolorbox}

\newtcolorbox{commentbox}{
  colback=yellow!10,
  colframe=red!50!black,
  boxrule=0.5pt,
  arc=2pt,
  left=4pt,right=4pt,top=2pt,bottom=2pt
}

\definecolor{MutedBlue}{RGB}{75, 105, 135}

\usepackage[backend=biber,style=numeric,sorting=nyt,giveninits=true,maxbibnames=99]{biblatex}
\usepackage[
    colorlinks=true,
    linkcolor=MutedBlue,
    citecolor=MutedBlue,
    urlcolor=MutedBlue,
    filecolor=trueblue,
    linktoc=all,
    pdfstartview=FitH,
    pdfpagemode=UseNone
]{hyperref}

\usepackage{cleveref}

\renewcommand{\Qm}{\mathcal{QM}}
\newcommand{\qm}{\textnormal{Qm}}
\newcommand{\Fix}{\textnormal{Fix}}
\newcommand{\Per}{\textnormal{Per}}
\newcommand{\ov}{\overline}
\newcommand{\oo}{\infty}
\newcommand{\Dil}{\textnormal{Dil}}
\newcommand{\scrD}{\mathscr{D}}
\newcommand{\Qq}{\mathfrak{Q}}

\newcommand{\MC}{\textnormal{MC}}

\newcommand{\dd}{\operatorname{d}}
\newcommand{\Minv}{\mathcal M(\sigma)}

\author{Stephen Cantrell}
\address{Mathematical Institute, University of St Andrews, North Haugh, St Andrews, KY16 9SS, Scotland}
\email{sjc33@st-andrews.ac.uk}

\author{Pablo D. Carrasco}
\address{ICEx-UFMG, Avda. Presidente Antonio Carlos 6627, Belo Horizonte-MG, BR31270-90}
\email{pdcarrasco@gmail.com}

\author{Emilio Corso}
\address{Penn State, 203 McAllister Building, University Park, State College, PA16802}
\email{corsoemilio2@gmail.com}

\author{Federico Rodr\'{i}guez-Hertz}
\address{Penn State, 227 McAllister Building, University Park, State College, PA16802}
\email{hertz@math.psu.edu}

\date{\today. 
}

\title[The moduli space of quasimorphisms]{The moduli space of quasimorphisms on subshifts and its pressure geometry}

\begin{document}
\begin{abstract} 
We introduce a metric geometry on the space of positive, linearly growing quasimorphisms on a mixing subshift of finite type. Two such quasimorphisms are identified when their homogenizations agree up to multiplicative scaling. The resulting moduli space carries a natural symmetrised Thurston-type metric, defined from the extremal ratios of their periodic-orbit lengths. We prove that this is a genuine metric and determine several global features of the resulting space: it is unbounded, separable, incomplete, and has geodesics that are canonically generated by thermodynamic pressure functions. We then describe the endpoints of these geodesics and develop a boundary theory for our moduli space.

We apply this  theory to metric structures on right-angled Artin and Coxeter groups. In particular we construct isometric embeddings of geometric moduli spaces into our quasimorphism moduli spaces. As a consequence, Manhattan curves for metrics on these groups are $C^1$, with equality cases characterised by rough similarity of the underlying metrics. This provides insights into Manhattan curves in settings where neither classical thermodynamic formalism for H\"older potentials nor Patterson–Sullivan theory apply. The same construction recovers the corresponding results for hyperbolic groups and places these examples within a common dynamical-geometric framework.
\end{abstract}

\maketitle

\tableofcontents

\section{Introduction}
\label{sec:Introduction}
A recurring theme in geometry is the construction of moduli spaces whose points encode geometric structures. Such spaces are themselves often equipped with a metric that measures the `distortion' required to move from one structure to another. Classical examples include the Teichm\"uller space of a surface and the Culler-Vogtmann Outer space, both of which can be equipped with Thurston-type metrics; for the latter moduli space we refer to the original article~\cite{Culler-Vogtmann}, whereas for the former we recommend the comprehensive introduction supplied in Farb-Margalit's monograph~\cite{Farb-Margalit}). In these settings, the metric on the moduli space is a central object of study: it reveals dynamical information, rigidity phenomena, and allows one to study the large-scale geometry of the space.

The purpose of this paper is to develop a geometric moduli space for symbolic dynamical systems. We construct a moduli space of length functions on subshifts of finite type, represented by almost additive potentials, and equip it with a natural Thurston-type metric. The resulting space admits canonical geodesics arising from thermodynamic pressure functions and is sufficiently flexible that it can be used to encode many geometric moduli spaces. For example, we will prove that many well-known geometric moduli spaces, including Teichm\"uller space and Outer space, embed isometrically into our dynamical moduli space (see Remark \ref{rem.out}), and thereby we provide a dynamical framework in which we can compare different geometries.

The fundamental objects in our construction are equivalence classes of quasimorphisms on subshifts of finite type.
We now briefly introduce the basic definitions so that we can present the construction of our symbolic moduli space; see \textsection\ref{subsec:quasimorphisms} for a more extensive discussion of definitions and results, and \textsection\ref{sec:prelim} for the required preliminaries.\\

\subsection{The symbolic moduli space}
\label{subsec:modulispace}
 Let $k\geq 1$ be an integer, $(\Sigma_A, \sigma)$ a mixing one-sided subshift of finite type on the alphabet $\{1, \ldots, k\}$ with transition matrix $A$. We write $\W$ for the collection of all finite non-empty words allowed by $A$: $\W = \bigcup_{n \ge 1} \W_n$ where
\[
  \W_n =  \{ (x_1,\ldots,x_n)\in \{1,\dots,k\}^n : A_{x_i,x_{i+1}} =1 \text{ for } i =1, \ldots, n-1\}
\]
for all $n\geq 1$. Given a word $w \in \W$ we write $|w|$ for its length, that is, $|w| = n$ if $w\in\W_n$. 

\begin{definition}[Quasimorphism]
\label{def:quasimorphism}
A function $\p \colon \W \to \R$ is called a quasimorphism if there exists $D>0$ such that
\begin{align}\label{eq:defquasimorphism}
|\p(w_1w_2) - \p(w_1) - \p(w_2)| \le D
\end{align}
for all pairs of words $w_1, w_2 \in \W$ with $w_1w_2 \in \W$. The least $D$ realising the property is denoted by $D_{\p}$, and is called the \emph{defect} of $\p$.
\end{definition}
 Since $\Sigma_A$ is fixed in what follows, we suppress it from the notation and write $\qm$ for the set of quasimorphisms on $\Sigma_A$, which is naturally endowed with the structure of an $\R$-vector space. 

Let $\qm^+ \subset \qm$ denote the collection of quasimorphisms that are bounded below, and define $\qm^{++}$ to be the subset of $\qm^+$ consisting of those quasimorphisms $\p\colon \W\to \R$ with linear growth\footnote{It is immediate, by subadditivity and finiteness of the alphabet $\{1,\dots,k\}$, that any quasimorphism $\psi$ has \emph{at most} linear growth: there exist $\tau',C'>0$ such that $\psi(w)\leq \tau'|w|+C'$ for all $w\in \W$.}: there exist $ C,\tau>0$ such that 
\begin{align}\label{eq:defqm++}
\p(w) \ge \tau |w| - C
\end{align}
for all $w\in \W$. Equivalently, $\p$ belongs to $\qm^{++}$ if and only if $\{ w \in \W : \p(w) \leq T\}$ is a finite set for every $T>0$; see Proposition \ref{prop.growthrate}.
 
A word $w\in \W$ is periodic if the product $ww$ belongs to the language $\W$ of the subshift. Given a periodic word $w$ we use $\{w\}$ to denote the corresponding \emph{periodic orbit}, by which we mean, abusing terminology, the multiset that contains all cyclic permutations of $w$ with multiplicity. If $w \in \W_n$ we say that $\{w\}$ has length $n$. The set of periodic orbits of length $n$ is denoted by $\Per_n$, and $\Per=\bigsqcup_{n\geq 1} \Per_n$ is the collection of all periodic orbits.

The homogeneization of a quasimorphism $\p$ is defined as the function
\begin{align}\label{eq:defhomogenization}
\ov\p : \Per \to \R\;, \{w\} \mapsto\lim_{n\to\infty} \frac{\p(w^n)}{n}
\end{align}
where the limit exists by almost subadditivity (coming from the quasimorphism property). We use the homogenizations to define the dilation, intuitively the largest stretch factor, between two quasimorphisms. Given $\phi, \p \in \qm^{++}$ let 
\begin{align}\label{eq:defDil}
\Dil(\phi,\p) = \sup_{w \in \Per} \frac{\ov{\phi}\{w\}}{\ov{\p}\{w\}}\;
\end{align}
which is well defined since $\ov{\p}$ is strictly positive for $\p \in \qm^{++}$.

\begin{definition}[Roughly similar quasimorphisms]\label{def:eqqm}
We say that two quasimorphisms $\phi, \p$ are roughly similar, and write $\phi \sim \p$, if there exists $C,\tau >0$ such that
\begin{align}\label{eq:defeqqm}
|\p(w) - \tau \phi(w)| \leq C
\end{align}  
for all $w \in \W$.
\end{definition}

This clearly defines an equivalence relation on $\qm$, for which $\qm^+$ and $\qm^{++}$ are saturated sets. This equivalence relation is natural from the dynamical perspective: it characterises when $\p$ and  $\phi$ have the same periodic orbit data, see Corollary \ref{cor.od}. We are now ready to define our moduli space.

\begin{definition*}[Moduli space of quasimorphisms]
Let $(\Sigma_A,\sigma)$ be a mixing subshift of finite type. Define the moduli space of quasimorphisms $\Qm^{++}$ to be the space of rough similarity equivalence classes of quasimorphisms in $\qm^{++}$. We equip $\Qm^{++}$ with the function $\Qq : \Qm^{++} \times \Qm^{++} \to \R$ defined by
\[
\Qq([\phi], [\psi]) = \log \left( \Dil(\phi,\p) \, \Dil(\p, \phi) \right)\;.
\]
Here $[\phi], [\p]$ denote the equivalence classes of  $\phi, \p \in \qm^{++}$.
\end{definition*}

\begin{remark}
1) The definition of this space was inspired by the moduli space of hyperbolic metrics studied by Reyes in~\cite{Oregon-Reyes} and by the first author with Reyes in~\cite{cantrell.reyes.man}. \\
2) If $\phi_i\sim\p_i, i=1,2$, then $\overline{\phi_i}=\tau_i\overline{\p_i}$, and $\Dil(\phi_1,\phi_2)=\frac{\tau_1}{\tau_2}\Dil(\p_1,\p_2)$, and
$\Dil(\phi_2,\phi_1)=\frac{\tau_2}{\tau_1}\Dil(\p_2,\p_1)$. It follows that $\Qq$ is well defined on rough classes.\\
3) It is also possible to equip $\Qm^{++}$ with the asymmetric Thurston metric: $\widetilde{\Qq}([\phi], [\psi]) = \log \left( \Dil(h(\phi) \phi, h(\p)\p\right)$ where $h(\phi), h(\p)$ are the growth rates of $\phi,\p$, see \eqref{eq:hquasi}. Most of the results of this paper carry over to the asymmetric metric (with the appropriate modifications if necessary).
\end{remark}

One of our main results can be loosely stated as follows; see Theorem~\ref{thm.spaceprops} for a more accurate formulation.

\begin{theorem*} 
The function $\Qq$ is a metric on $\Qm^{++}$. Furthermore, the metric space $(\Qm^{++}, \Qq)$ is contractible, unbounded, separable and is not complete.
\end{theorem*}

\begin{remark}
1) To prove this theorem we develop thermodynamic results for quasimorphisms; see \textsection\ref{subsec:quasimorphisms}. \\
2) We show that the non-completeness occurs due to the defect diverging. If we make restrictions on the defect, we recover a complete subspace inside $\Qm^{++}$ (see Theorem \ref{thm.def}).\\
3) We also show that the subspace of equivalence classes in $\Qm^{++}$ that are naturally associated to locally constant (or H\"older) functions is dense in $\Qm^{++}$.
\end{remark}

After proving this result we turn to understanding geodesics in the space of quasimorphisms $(\Qm^{++},\Qq)$. For a single quasimorphism $\p \in \qm$ there is a natural notion of a pressure curve $t \mapsto \Pr(t\p)$ associated to $\p$.  Given a pair $\phi,\p \in \qm^{++}$; for each $t \in \R$ the implicit solution in $s$ to the equation $\Pr(t\phi - s\p) = 0$ determines the pressure function $\Pr(\phi,\p)(t)$ of the pair $\phi, \p$ (cf.~Definition~\ref{def:pressurefunction}). We show the following:

\begin{theorem*}[{cf.~Theorem~\ref{thm.geo}}]
Each pair of points $[\phi], [\p] \in \Qm^{++}$ can be connected by the bi-infinite $\Qq$ geodesic $t\mapsto [\Pr(\p,\phi)(t) \phi - t\p]$. 
\end{theorem*}
\noindent We shall make use of this theorem in order to define and study a natural boundary for the space $\Qm^{++}$.

One of the principal motivations for introducing this abstract moduli space (and studying its geometry) is that it contains concrete geometric examples. In the second part of the paper we consider metric structures on right-angled Artin groups, right-angled Coxeter groups and hyperbolic groups (see \textsection\ref{sec.geoapp}). These groups admit natural symbolic codings, and their metric structures give rise to almost additive potentials on the associated subshifts. The geometric moduli spaces associated to these groups then admit natural embeddings into the quasimorphism moduli space. 

More precisely, for the classes of metrics considered below, the embedding is compatible with the respective equivalence relations. In particular, the metric $\Qq$ restricts to the symmetrised Thurston metric on the geometric moduli space. The  coarse geometry then naturally sits within the pressure geometry of the ambient symbolic moduli space. This point of view also clarifies the relationship with Manhattan curves. For a pair of geometric metrics in our considered geometric moduli spaces, their  Manhattan curve is in fact the pressure function associated to the two corresponding quasimorphisms. That is, under the embedding constructed in this paper, the Manhattan curve becomes precisely a pressure curve in the ambient moduli space. Thus regularity, convexity, and rigidity properties of Manhattan curves can be understood as manifestations of the general pressure geometry developed here.

Finally, the same framework yields counting and large-deviation results. For suitable observables and metrics, the distribution of geometric lengths of group elements is governed by the pressure function of the associated symbolic system. The corresponding rate functions are obtained from the Fenchel-Legendre transform of the pressure, giving a unified explanation for the large-deviation phenomena that arise in the geometric group-theoretic setting. See Theorem \ref{thm.ldp}.

We end this section by noting that there are several conceptual consequences of this framework. First, it provides a common language for thermodynamic formalism and geometric moduli spaces: asymptotic length functions on a symbolic system can be treated as points of a metric space, rather than merely as observables on a fixed dynamical system. Second, it shows that pressure has an intrinsic geometric meaning, since its zero sets generate geodesics in the moduli space. Third, the construction gives a natural setting in which geometric structures on non-hyperbolic groups can be studied using thermodynamic methods.

Before we state our results precisely, we present the aforementioned applications to metrics on right-angled Artin/Coxeter groups.

\subsection{An application: coarse geometries for RAAGs and RACGs}\label{sec.geoapp}

\subsubsection{Manhattan curves}
Suppose that $\G$ is a group and that $d_1, d_2$ are metrics on $\G$. There is a natural curve, called the Manhattan curve, that interpolates between these metrics. It is obtained as the boundary of the convex set
\[
\mathcal{C}_M(d_1,d_2) = \left\{ (a,b) \in \R^2 : \sum_{x \in \G} e^{-ad_1(o,x) - b d_2(o,x)} < \infty \right\} \subset \R^2
\]
(which could be empty). These curves were introduced by Burger \cite{burger} for pairs of metrics coming from actions of a surface group on rank $1$ symmetric spaces. Since their conception, they have proved a useful tool for establishing counting, rigidity and statistical results: we refer, among others, to the works of Kao~\cite{Kao}, Tanaka and the first author~\cite{cantrell.tanaka.1}, \cite{SharpManhattan}, Bianchi and He~\cite{Bianchi-He}, \cite{bray2022counting}. It is of particular interest to establish good regularity properties of Manhattan curves as one obtains multiple corollaries from such results. There are two methods that have been employed to establish regularity of Manhattan curves. The first is thermodynamic formalism for potentials with sufficiently good (i.e., H\"older) regularity~\cite{SharpManhattan,Kao,cantrell.tanaka.2} and the second is Patterson-Sullivan theory~\cite{cantrell.tanaka.1,Reyes-Wang}. We are able to use new thermodynamic results to establish  $C^1$ regularity and rigidity properties of Manhattan curves in settings where neither thermodynamic formalism for H\"older potentials nor Patterson-Sullivan theory are immediately applicable. For now we present examples for right-angled Artin groups and Coxeter groups. See \textsection\ref{sec.apps} for further examples.

Suppose that $\Gamma$ is a right-angled Artin group or right-angled Coxeter group that is not virtually cyclic and does not split as a direct product. Let $d$ denote the word metric on $\G$ associated to the standard vertex generating set.

Given a possibly asymmetric, $\G$ invariant metric  $d_\ast : \G \times \G \to \R$, let $(\cdot|\cdot)_{d_\ast}$ denote (the possibly asymmetric Gromov product)
\begin{align}\label{eq:defGromovproduct}
(x|y)_{d_\ast} = \frac{1}{2} \left( d_\ast(x,o) + d_\ast(o,y) - d_\ast(x,y) \right)
\end{align}
where $x,y \in \G$ and $o \in \G$ is the identity. The Gromov product for an arbitrary function $\p : \G \times \G \to \R$ is defined analogously. 
\begin{definition}
Let $\mathcal{H}_\G^d$ denote the collection of (possibly asymmetric) functions $\p : \G \times \G \to \R$ that are $\G$-invariant and satisfy the following: there exists $c_1, c_2 >0$ such that
\[
|(x|y)_{\p}|\le c_1(x|y)_d + c_2
\]
for all $x,y \in \G$.  We write $\mathcal{D}_\G^d \subset \mathcal{H}_\G^d$ for the collection of (possibly asymmetric) metrics in $\mathcal{H}_\G^d$ that are quasi-isometric to $d$.
\end{definition}
\begin{example}
The space $\mathcal{H}_\G^d$ contains many functions of interest. For example it contains:
\begin{enumerate}
\item all quasimorphisms $\phi : \G \to \R$, i.e. set $\p(x,y) = \phi(x^{-1}y)$;
\item the function $\#_v : \G \to \R$ that counts the number of occurrences of a fixed vertex $v^{\pm1}$ in the shortest word representation of a group element, i.e. $\p(x,y) = \#_v(x^{-1}y)$; and,
\item the action of $\G$ on the contact graph, i.e. if $\mathcal{C} X$ is the contact graph with metric $d_\mathcal{C}$ set $\p(x,y) = d_\mathcal{C}(x \cdot v_0, y \cdot v_0)$ where $v_0 \in  \mathcal{C} X$ is a fixed base point vertex (see \textsection\ref{sec.con}).
\end{enumerate}
\end{example}

\begin{remark}\label{rem.dg}
When $\G$ is a hyperbolic group and $d$ is a word metric on $\G$, the collection of symmetric metrics in $\mathcal{D}_\G^d$ is precisely the space $\Dc_\G \cup \partial_M\Dc_\G$: the union of hyperbolic, $\G$-invariant metrics quasi isometric to a word metric (defined by Furman in \cite{furman}) with the Manhattan boundary $\partial_M\Dc_\G$ (defined by Reyes and the first author in~\cite{cantrell.reyes.man}).  Furman initially proved interesting rigidity results for metrics in $\Dc_\G$ \cite{furman}. The $C^1$ regularity and rigidity properties for Manhattan curves were later established by the first author in joint work with Tanaka~\cite{cantrell.tanaka.1}, when quasi-isometric to a word metric, and with Reyes~\cite{cantrell.reyes.man} for general metrics in this space. In \textsection\ref{sec.apps} we show how to use our new theory to recover the results from~\cite{furman}, \cite{cantrell.tanaka.1} and \cite{cantrell.reyes.man}.
\end{remark}

We say that two functions $\p_1, \p_2 \in \mathcal{H}_\G^d$ are roughly similar if there exists $\tau, C>0$ such that
\[
|\p_1(x,y) - \tau \p_2(x,y)|< C
\]
for all $x,y \in \G$. Let $\mathscr{D}_\G^d$ and $\mathscr{H}_\G^d$ denote the rough similarity classes of $\mathcal{D}_\G^d$ and $\mathcal{H}_\G^d$ respectively. We can think of $\mathscr{D}_\G^d$ as the collection of genuinely different metrics that fellow travel $d$. The collection $\mathscr{D}_\G^d$ (hence $\mathscr{H}_\G^d$) is uncountable, see Lemma \ref{lem.uncount}.
﻿
Given two functions $\p_1,\p_2 \in \mathcal{H}_\G^d$ we can define the Manhattan curve as above. We can then extend the main results of \cite{cantrell.tanaka.1} beyond the hyperbolic setting to $\mathcal{H}_\G^d$.

\begin{theorem}\label{thm.rags1}
Suppose that two metrics $\p_1, \p_2 \in \mathcal{H}_\G^d$ have non-empty Manhattan curve $\theta_{\p_1/\p_2}$. Then the Manhattan curve $\theta_{\p_2/\p_1}$ is $C^1$ and is a straight line if and only if $\p_1$ and $\p_2$ are roughly similar.
\end{theorem}

We also have the following rigidity result. In the statement, $h(d_1), h(d_2)$ correspond to the exponential growth rates of $d_1, d_2$ respectively:
\[
h(d_j) = \limsup_{T\to\infty} \frac{1}{T} \log \#\{ x \in \G: d_j(o,x) < T\}
\]
for $j=1,2$.

\begin{theorem}\label{thm.rags2}
Suppose that $d_1,d_2 \in \Dc_\G^d$. Then the following limit exists
\[
\tau(d_2/d_1) := \lim_{T\to\infty} \frac{1}{\#\{ x \in \G: d_1(o,x) < T\}}  \sum_{d_1(o,x) < T} \frac{d_2(o,x)}{T},
\]
 satisfies $\tau(d_2/d_1) \ge h(d_1)/h(d_2)$ and is equal to the derivative of $-\theta_{d_2/d_1}$ at $0$. Furthermore, the following are equivalent:
\begin{enumerate}
\item $\tau(d_2/d_1) = h(d_1)/h(d_2)$;
\item $d_1$ and $d_2$ are roughly similar; and,
\item $\theta_{d_2/d_1}$ is a straight line.
\end{enumerate}
\end{theorem}

\begin{remark}
1) If $d_1 \in \Dc_\G^d$ and $\p_2 \in \mathcal{H}_\G^d$ then the above limit exists. Furthermore, in this case, if $\p_2 \in \mathcal{H}_\G^d$ is positive and is not bounded, then $\tau(\p_2/d_1)$ is strictly positive.\\
2) The above two results are inspired by the main results in \cite{cantrell.tanaka.1} for hyperbolic groups.
\end{remark}
To the authors' knowledge, these are the first results establishing regularity and rigidity of Manhattan curves in a setting for which neither Patterson-Sullivan theory nor (classical) thermodynamic formalism is available to be applied.\\

Our methods can also be used to study functions in $\mathcal{H}_\G^d$ allowing us to prove precise statistical results for quasimorphisms on $\G$, see Theorem \ref{thm.ldpraags}. 

Again, to the authors' knowledge, this is the first example of the counting large deviation theory/principle for quasimorphisms being established outside of the hyperbolic setting.

\subsubsection{Space of metric structures}

We can also extend the main results from \cite{cantrell.reyes.man} to RAAGs and RACGs. That is, we can describe the geometry of the moduli space of metric structures $\scrD_\G^d$ equipped with the symmetrised Thurston distance.
Let $\D : \scrD_\G^d \times \scrD_\G^d \to \R$ be defined by
\[
\D([d_1],[d_2]) = \log (\Dil(d_1, d_2) \, \Dil(d_2,d_1))
\]
where $\Dil(d_1,d_2), \Dil(d_2,d_1)$ denote the dilations of $d_1$ and $d_2$, see \textsection\ref{sec.rags}.
Then we prove the following.

\begin{theorem}\label{thm.ragsspace}
The function $\D$ determines a metric on $\scrD_\G^d$. Furthermore, given $[d_1], [d_2] \in \scrD_\G^d$, for any $t\in\R$ there exists a metric $d_t \in \Dc_\G^d$ and $C_t >0$ such that
\[
|d_t(x,y) - (\theta_{d_2/d_1}(t) d_1(x,y) + td_2(x,y))| < C_t \ \ \text{ for all $x,y \in \G$}
\]
and  the curve $t \mapsto [d_t]$ determines a bi-infinite geodesic in $(\scrD_\G^d, \D)$. In particular, $(\scrD_\G^d, \D)$ is unbounded.
\end{theorem}
The geodesics constructed in this theorem provide us with a natural notion of a boundary for $\scrD_\G^d$: given $[d_1], [d_2] \in \scrD_\G^d$ as in the previous theorem, the limit
\[
\p_{\pm \infty}(x,y) := \pm \lim_{t \to \pm \infty} \frac{1}{t} (\theta_{d_2/d_1}(t) d_1(x,y) + td_2(x,y))
\]
for $x,y \in \G$ determines a function in $\Hc_\G^d$. We show that $\p_{\pm \infty}$ are within uniformly bounded distance of metrics $d_{\pm \infty} \in \Hc_\G^d$ (see Proposition \ref{prop.inftymetric}). We think of these metrics as the metrics at infinity of the Manhattan curve joining $d_1, d_2$.

\begin{definition}
Define $\partial_M \scrD_\G^d \subset \mathscr{H}_\G^d$ to be the collection of rough similarity classes containing a metric of the form $d_{\pm \infty}$ as constructed above.
\end{definition}

We then describe these endpoint metrics up to rough similarity.

\begin{theorem}\label{thm.raagboundary}
The class $\partial_M \scrD_\G^d$ consists precisely of equivalence classes of metrics that contain a metric $d_\ast \in \mathcal{H}_\G^d$ that satisfy the following properties:
\begin{enumerate}
\item $d_\ast$ is $\G$-invariant;
\item there exists $c_1, c_2 >0$ such that 
\[
(x|y)_{d_\ast}\le c_1(x|y)_d + c_2 
\]
for all $x,y \in \G$; and,
\item $d_\ast$ is not quasi-isometric to $d$.
\end{enumerate}
\end{theorem}
\noindent In \textsection\ref{sec.con} we show that the action of $\G$ on the contact graph induces a point in $\partial_M \scrD_\G^d$.\\

These results show that the metric properties of $(\scrD_\G, \D)$ studied  in~\cite{cantrell.reyes.man} (see Remark \ref{rem.dg}) hold for non-hyperbolic groups. In fact we will establish the following embedding theorem that illustrates how to relate dynamical and geometrical moduli spaces. In the following result, $\Pr$ denotes the pressure function associated to a pair of quasimorphisms, see \textsection\ref{sec.rigid}.

\begin{theorem}
Let $(\mathscr{D}_\G^d, \D)$ be the space of metrics on $\G$ as above. Then there exist a mixing subshift of finite type $\Sigma_A$ and function $F : \Dc_\G^d \to \qm^{++}(\Sigma)$ that descends, via quotienting by the equivalence relations, to a well-defined, isometric embedding $\ov{F} : (\mathscr{D}_\G^d, \D) \to (\Qm^{++}(\Sigma), \Qq)$, i.e. $\ov{F}$ is injective and preserves distances. Furthermore $F$ satisfies 
\[
\theta_{d_2/d_1}(t) = \Pr(F(d_2), F(d_1))(-t)
\]
for each $t \in \R$  (where $\theta_{d_2/d_1}$ is the Manhattan curve)
and so
\[
\ov{F}[td_2 + \theta_{d_2/d_1}(t)d_1] = [-t F(d_2) + \Pr(F(d_2), F(d_1))(-t) F(d_1)]
\]
for each $t \in \R$.
\end{theorem}
This result forms a concrete link between the geometric and dynamical moduli spaces which we highlight in the following table.
\begin{table}[ht]
\centering
\caption{A dictionary between dynamical and geometric moduli spaces.}
\label{tab:symbolic-group-dictionary}
\begin{tabular}{ll}
\toprule
\textbf{Dynamical } & \textbf{Geometric} \\
\midrule
$\qm$ & $\Hc_\G^d$ \\[1mm]
$\Qm^{++}$ &  $\mathscr{D}_\G^d$ \\[1mm]
Periodic orbit & Conjugacy class in $\G$ \\[1mm]
Homogenization $\ov{\p}$ & Stable translation length $\ell_{d_\ast}$  \\[1mm]
Pressure curve & Manhattan curve \\[1mm]
$\Qq$ & $\Delta$ \\[1mm]
Pressure geodesics & Manhattan geodesics \\[1mm]
\bottomrule
\end{tabular}
\end{table}

\section{Statements of results}
\subsection{Quasimorphisms and their pressure}
\label{subsec:quasimorphisms}

We now precisely state our results on the thermodynamic formalism of quasimorphisms. For more details on the definitions and the notation used throughout this introduction, as well as for historical comments about the various notions we make use of, we refer to \textsection\ref{sec:prelim}.

Let $(\Sigma_A,\sigma)$ be a mixing one-sided subshift of finite type, as in the foregoing section; we retain terminology and notation from the latter. In particular, recall from \textsection\ref{subsec:modulispace} that $\qm^+$ denotes the collection of quasimorphisms on $\W$ that are bounded below, and that $\qm^{++}$ is the subcollection of those which have linear growth; both $\qm^+$ and $\qm^{++}$ are convex cones in $\qm$. The topological pressure of a quasimorphism is defined as the real number
\[P(\p)=\lim\limits_{n\to\infty}\frac{1}{n}\log \sum_{w\in \W_n}e^{\p(w)}\;,\]
where the limit exists and is finite by a standard argument involving both subadditivity and superadditivity.

\begin{definition}[Pressure function]
\label{def:pressurefunction}
Let $\phi \in \qm$ and $\p \in \qm^{++}$. The pressure function $\Pr(\phi,\p)\colon \R\to \R$ is defined by setting $\Pr(\phi,\p)(t)$ to be the unique $s\in \R$ satisfying
\begin{equation}
\label{eq:implicit}
\Pr(t\phi - s\p) = 0\;,
\end{equation}
for every $t\in \mathbb{R}$.
\end{definition}
\noindent In \textsection\ref{subsec:pressurefunction} we show that $P(\phi,\p)$ is $C^1$. Lemma~\ref{lem.expgrow} gives an equivalent, explicit description of $P(\phi,\p)$: for all sufficiently large $R>0$, depending on $\phi$ and $\p$, 
\[
\Pr(\phi,\p)(t) = \lim_{n\to\infty} \frac{1}{n} \log \sum_{w\in \W,\;|\p(w) - n| < R} e^{t\phi(w)\;.}
\]

In \textsection\ref{subsec:pressurefunction}, we will prove that the pressure function is convex and, similarly to~\cite[Section 4.1]{Barreira_2009}, that it is of class $\mathrm{C}^1$;  we further develop the theory by establishing the following rigidity and counting results.

\subsection{Rigidity and counting for quasimorphisms}\label{sec.rigid}
Throughout, the notation $|A|$ stands for the cardinality of a finite set $A$. Also, for $\psi \in \qm^{++}$ and $T>0$, we set 
\[\W_{\psi}(T)=\{w \in \W: \psi(w)\leq T\}\;;\]
it is easy to see that the limit 
\begin{align}\label{eq:hquasi}
h(\p) \coloneq  \lim_{T \to\infty} \frac{1}{T} \log |\W_\p(T)|
\end{align}
exists and is strictly positive.

We first establish rigidity and counting results for quasimorphisms that will be crucial in our study of $\Qm^{++}$.
Our first theorem links the asymptotic average of pairs of quasimorphisms with their growth rates and pressure curve.

\begin{theorem}\label{thm.rigidity2}
For $\phi\in \qm$ and $\p\in \qm^{++}$, the limit
\[
\tau(\phi/\p) := \lim_{T\to\infty} \frac{1}{|\W_\p(T)|} \sum_{w \in \W_\p(T)} \frac{\phi(w)}{T} 
\]
exists and is equal to the derivative of $\Pr(\phi,\p)$ at $0$. If $\phi \in \qm^{++}$, then the limit satisfies $\tau(\phi/\p) \ge h(\p)/h(\phi)$. Furthermore, in the latter case, the following are equivalent:
\begin{enumerate}
\item $\tau(\phi/\p) = h(\p)/h(\phi)$;
\item $\Pr(\p,\phi)$ is affine;
\item $\Pr(\p,\phi)$ is not strictly convex on $\mathbb{R}$; and,
\item $\phi \sim \p$.
\end{enumerate}
\end{theorem}
\noindent As the notation indicates, $\tau(\phi/\psi)$ depends only on the quotient function $\phi/\psi$; this will become apparent in Theorem~\ref{thm.drift}.  
Theorem~\ref{thm.rigidity2} can be seen as a symbolic analogue of~\cite[Theorem 1.2]{cantrell.tanaka.1}.

We will then establish this result from a more general result that has a large deviation theorem.
\begin{theorem}\label{thm.rigidity}
Let $(\Sigma_A,\sigma)$ be a mixing subshift of finite type, $\phi \in \qm$, $\p \in \qm^{++}$.  Assume that $\phi$ is unbounded
Then the following are equivalent:
\begin{enumerate}
\item $\Pr(\phi,\p)$ is not affine;
\item $\Pr(\phi,\p)$ is strictly convex on $\R$; and,
\item $\p$ and $\phi$ are not equivalent.
\end{enumerate}
Furthermore, the following large deviation bound holds: there exists $\tau(\phi/\p) \in \R$ such that, for every $\epsilon >0$,
\begin{equation}\label{eq:ldbound}
\limsup_{T\to\infty} \frac{1}{T} \log \left( \frac{1}{|\W_\p(T)|} \biggl|\left\{ w \in \W_\p(T) : \left|\frac{\phi(w)}{T} - \tau(\phi/\p)\right| \geq \epsilon \right\}\biggr|\right) < 0\;.
\end{equation}
If $\phi\in \qm^+$, then $\tau(\phi/\p) > 0$ unless $\phi$ is bounded. 
\end{theorem}
﻿
\begin{remark}
The large deviation estimate~\eqref{eq:ldbound} will be strengthened in Theorem~\ref{thm.ldp} to a precise large deviation principle. 
\end{remark}

\subsubsection{The drift}

We further give an interpretation of the quantity $\tau(\phi/\p)$ as drift of $\phi$ with respect to $\p$, in the following sense. Let us denote by $\Minv$ the set of $\sigma$-invariant Borel probability measures on $\Sigma_A$.
Every quasimorphism $\p\colon \W\to \mathbb{R}$ determines a distinguished $\sigma$-invariant measure $\mu_{\p}$ on $\Sigma_A$, called the Gibbs state associated to $\p$; it is characterized by the property that
\[\Pr(\p)=h_{\mu_{\p}}(\sigma)+\mu_{\p}(\p)>h_{\mu}(\sigma)+\mu(\p)\]
for all $\mu_{\p}\neq\mu\in \Minv$, where $h_{\mu}(\sigma)$ denotes the Kolmogorov-Sinai entropy of $\sigma$ with respect to $\mu$; we refer to Theorem~\ref{thm:propiedadesmedidamul} and the surrounding discussion for further details. 

\begin{theorem}\label{thm.drift}
Let $\p\in \qm^{++}$, $\mu_{-h(\p)\p}$ the Gibbs state associated to $-h(\p)\p$. For every $\phi \in \qm$, the quantity $\tau(\phi/\p)$ is the drift of $\phi$ with respect to $\p$: that is, for $\mu_{-h(\p)\p}$-almost every $(x_k)_{k\geq 0}\in\Sigma_A$,
\[
\frac{\phi(x_0,\dots,x_{n-1})}{\p(x_0,\dots, x_{n-1})} \longrightarrow \tau(\phi/\p) 
\]
as $n\to\infty$. 
\end{theorem}
This result establishes positive drift for positive, unbounded quasimorphisms.

\subsection{The space of quasimorphisms}
We now define our moduli space of quasimorphisms in greater detail.

\begin{definition}[Moduli space of quasimorphisms]
Let $(\Sigma_A,\sigma)$ be a subshift of finite type. We define the moduli space of quasimorphisms $\Qm^{++}$ as the space of equivalence classes of quasimorphisms in $\qm^{++}$ for the equivalence relation introduced in Definition~\ref{def:eqqm}.

Given $\p \in \qm^{++}$, we write $[\p]$ for its class in $\Qm^{++}$.
\end{definition}

A word $w\in \W$ is periodic if the product $ww$ belongs to the language $\W$ of the subshift; this clearly implies that $w^n\in \W$ for all $n\geq 1$. Given a periodic word $w$ we use $\{w\}$ to denote the corresponding periodic ``orbit", by which we mean the multiset that contains all cyclic permutations of $w$ with multiplicity. We consider multisets so that, given $m,n \ge 1$,  $\{w^m\}$ and $\{w^n\}$ are considered different unless $m=n$. If $w \in \W_n$ we say that $\{w\}$ has length $n$. The set of periodic orbits of length $n$ is denoted by $\Per_n$, and $\Per=\bigsqcup_{n\geq 1} \Per_n$ stands for the collection of all periodic orbits. For mixing subshifts of finite type, there is an abundance of periodic $\sigma$-orbits, and thus of periodic words in the language of the subshift: their exponential growth rate is known to match the topological entropy of $\sigma$, which is strictly positive in this case (see the comprehensive treatment in Lind-Marcus's textbook~\cite[Chapter 4]{Lind-Marcus}).

 Recall that the homogenization of $\p \in \qm$ is denoted by $\ov\p : \Per \to \R$.
By sub/super-additivity, $\ov \p$ is well defined and 
\[\ov{\p}\{w\} = \inf_{n \ge 1} \frac{\p(w^n) + D_\p}{n} = \sup_{n\ge 1}\frac{\p(w^n) - D_\p}{n}\;.\]

\begin{remark}
The counting results in Theorem \ref{thm.rigidity} and Theorem \ref{thm.rigidity2} also hold when we count over $\Per$ with respect to the homogenization; see \textsection\ref{sec.ldp}.
\end{remark}

Given $\phi, \p \in \qm^{++}$ let 
\[
\Dil(\phi,\p) = \sup_{w \in \Per} \frac{\ov{\phi}\{w\}}{\ov{\p}\{w\}}\;,
\]
which is a strictly positive real number. Observe that 
\[\Dil(\phi,\p)\Dil(\p,\phi)\geq \sup_{w\in \Per}\frac{\ov{\phi}\{w\}}{\ov{\p}\{w\}}\frac{\ov{\p}\{w\}}{\ov{\phi}\{w\}}=1\;;\]
we set
\[
\Qq([\phi], [\psi]) = \log \left( \Dil(\phi,\p) \, \Dil(\p, \phi) \right)\;,
\]
which is easily seen to be well defined. We prove that this assignment defines a distance function, and explore the metric properties of the resulting metric space.

\begin{theorem}\label{thm.spaceprops}
The function $\Qq$ is a metric on $\Qm^{++}$. Further, the metric space $(\Qm^{++}, \Qq)$ is contractible, unbounded, separable and is not complete.
\end{theorem}
\begin{remark}
To show separability we introduce the notion of a `word quasimorphism' which is a dynamical analogue of word metrics on groups; see \textsection\ref{sec.wq}.
\end{remark}

\subsubsection{Subsets in $\Qm^{++}$}
The non-completeness of $(\Qm^{++},\Qq)$ is caused by the divergence of the defect; for every real number $D\geq 0$, let $\Qm_D^{++} \subset \Qm^{++}$ denote the subcollection of classes that contain a quasimorphism $\p$ with $D_\p \le D$ and $h(\p) = 1$. \footnote{The normalization requirement is clearly needed for the definition to be meaningful, since any rescaling of the quasimorphism does not change the corresponding element of $\Qm^{++}$.}

\begin{theorem}\label{thm.def}
For each $D >0$ the metric space $(\Qm_D^{++}, \Qq)$ is complete.
\end{theorem}

Given $\p \in \qm^{++}$ we say that $\p$ is \textit{twin} to a rational valued, locally constant function $f : \Sigma_A \to \Q$ if for any periodic $\{w\} \in \Per_n$ we have that
\[
S_n f(x_w) = \ov\p\{w\}
\] 
where $x_w \in \Sigma_A$ is the point obtained from repeating $w$, $x_w = (w,w,\ldots)$ (see Example \ref{ex.holder}).

Let $\Qm^{++}_{lc, \Q} \subset \Qm^{++}$ denote the collection of elements in $\Qm^{++}$ that are twin to a rational valued, locally constant function. 
\begin{theorem}\label{thm.lc}
The class $\Qm^{++}_{ lc, \Q}$ is dense in $(\Qm^{++}, \Qq)$. 
\end{theorem}

\begin{remark}
1) Note that locally constant functions are H\"older continuous and so, equivalence classes that are twins to H\"older functions are dense in $(\Qm^{++},\Qq)$.

\noindent 2) Since $\Qm^{++}_{ lc, \Q}$ is countable we deduce that $\Qm^{++}$ is separable. 
\end{remark}

\subsubsection{Pressure geodesics}
Inspired by~\cite{cantrell.reyes.man}, we further study the geodesic properties of $(\Qm^{++}, \Qq)$. We prove the following, see

\begin{theorem}\label{thm.geo}
For any $[\phi], [\p] \in \Qm^{++}$ and $t \in \R$, the function $w \mapsto \p_t(w):= \Pr(\p,\phi)(t) \phi(w) - t\p(w)$ determines a quasimophism $\p_t \in \qm^{++}$.
 Furthermore, the path
\[
\mathbb{R}\ni t \mapsto [\p_t] \in \Qm^{++}
\]
is a bi-infinite geodesic for the $\Qq$ metric: that is, for any $r, s,t \in \R$ with $s < r<t$ we have that
\[
\Qq([\p_r],[\p_t]) = \Qq([\p_r],[\p_s]) + \Qq([\p_s], [\p_t])
\]
and $\lim_{t\to\pm \infty} \Qq([\p_0],[\p_t]) = \infty$.
\end{theorem}

We then study a natural boundary for $\Qm^{++}$: the collection of equivalence classes of quasimorphisms that correspond to the endpoints of a pressure geodesics from Theorem \ref{thm.geo}. We show that this boundary is in fact equal to $\Qm^{+}\backslash \Qm^{++}$: the collection of equivalence classes of positive quasimorphisms that do not have linear growth. See \textsection\ref{sec.boundary} for the precise results.

Before concluding the introduction we make the following remark.
\begin{remark}
By recent results of Cantrell, Martínez-Granado and Reyes~\cite{C.D-M.R}, the symmetric reparameterizations of the geodesic flow on a closed, negatively curved Riemannian Manifold, can be identified with a natural moduli space that embeds in $\Qm^{++}(\Sigma_A)$ for some subshift $A$, see \textsection\ref{sec.d_G}. This provides motivation for our study of $\Qm^{++}$ as the length spectra (i.e. the periodic orbit data) of such repartameterizations are therefore realised as the periodic data of quasimorphisms on subshifts.
\end{remark}

\subsection{Other works}

The study of the space of quasimorphisms on a subshift of finite type was introduced in recent work of the second and fourth authors \cite{TermoEqSt}. There, this framework is used to describe the space of all (genuine) quasimorphisms of the fundamental group of a closed hyperbolic manifold and, as a consequence, its second bounded cohomology group. In that setting, a different notion of equivalence is considered: two quasimorphisms are said to be cohomologous if they are at bounded distance from each other. It is shown that each cohomology class admits a unique representative given by a Gibbs measure, which can equivalently be described by a Livšic cohomology class of a weak Bowen function. In particular, Example~\ref{ex.holder} essentially captures all quasimorphisms.

In the present article, we adopt a different viewpoint. We regard the space of quasimorphisms of linear growth (modulo the appropriate equivalence relation) as a moduli space, endowed with a geometry of independent interest.\\

\subsection*{Acknowledgements}

We kindly thank Mark Hagen for suggesting the proof of Proposition \ref{prop.conbound}.\\

P.D. Carrasco was partially supported by CNPq (Produtividade em Pesquisa) 305197/2025-8, Projeto Universal CNPq 401737/2025-0 and FAPEMIG APQ-02338-24. 

E. Corso gratefully acknowledges support by the Swedish Research Council under grant no. 2021-06594, while he was in residence at Institut Mittag-Leffler in Djursholm, Sweden during the program "Interactions between fractal geometry, harmonic analysis, and dynamical systems".

F. Rodriguez-Hertz was partially supported by NSF grant DMS-2453688.

\subsection*{Disclosure}
 
AI was not used for this project, neither to assist with the writing nor to obtain mathematical statements/proofs.


\section{Preliminaries}\label{sec:prelim}
\subsection{Subshifts of finite type}
For the required background in ergodic theory and dynamical systems, we refer to the textbook~\cite{Brin_2002} by Brin and Stuck; much more specifically geared towards symbolic dynamics is the book~\cite{Lind-Marcus} by Lind and Marcus. Let $k\geq 1$ be an integer, $\Sigma=\{1,\dots,k\}^{\mathbb{N}}$, endowed with the product of the discrete topologies on each factor. The topology can be induced by any of the distance functions 
\begin{equation*}
d_{\beta}((x_n),(y_n))=\beta^{\inf \{n\geq 0:x_n\neq y_n\}  }
\end{equation*}
for $0<\beta <1$. For a real-valued function $\varphi$ defined on a subset of $\Sigma$, the notion of Hölder continuity, without specification of the exponent, refers to any of the metrics $d_{\beta}$, and does not depend on the particular choice of $\beta$.

 Let $A=(A_{i,j})_{1\leq i,j\leq k}$ be a square matrix of size $k$ with entries in $\{0,1\}$. The set
\[\Sigma_{A}=\{(x_n)_{n\geq 0}\in \Sigma: A_{x_n, x_{n+1}}=1 \text{ for all }n\in \mathbb{N}  \}\]
is a closed subset of $\Sigma$ which is forward-invariant under the left shift map $\sigma:\Sigma\to\Sigma$, $(x_n)_{n\geq 0}\mapsto (x_{n+1})_{n\geq 0}$; the pair $(\Sigma_A,\sigma)$, where $\sigma$ is the restriction to $\Sigma_A$ of the left shift map, is called the one-sided subshift of finite type on the alphabet $\{1,\dots,k\}$ with transition matrix $A$. 

 If $\sigma$ is topologically mixing on $\Sigma_A$, we say that $(\Sigma_A,\sigma)$ is a mixing subshift of finite type. It is well known that this occurs if and only if $A$ is primitive: there exists an integer $N\geq 1$ such that for all $n\geq N$, all entries of $A^n$ are strictly positive. 

We view finite words in the alphabet $\{1,\dots,k\}$ as elements of the free semigroup generated by $\{1,\dots,k\}$, and employ the notation $w_1w_2$ for the product of two such words. For every $n\geq 1$, let
\[
\W_n = \{(x_0,\dots, x_{n-1})\in \{1,\dots,k\}^n: \text{ there is $\underline{y}=(y_n)_{n\geq 0}\in \Sigma_A$ with $y_0=x_0,\dots,y_{n-1}=x_{n-1}$} \}
\;.
\] 
The exponential growth rate of $|\W_n|$, which is well defined by an easy-to-check submultiplicativity, is known to match the topological entropy of $\sigma$ on $\Sigma_A$, which is strictly positive in the mixing case. We let then $\W = \bigsqcup_{n\ge 1} \W_n$ be the set of all finite nonempty words appearing as initial strings of elements of $\Sigma_A$\footnote{Equivalently, since $\Sigma_A$ is shift-invariant, $\W$ is the set of all finite nonempty words appearing inside elements of $\Sigma_A$ \emph{in any position}.}. We say that a word $w\in \W_n$ has length $n$. Given two words $w_1 = x_0 \cdots x_{n-1}$ and $w_2 = y_0 \cdots y_{m-1}$ in $\W$, we say that $w_1$ and $w_2$ are composable if $A_{x_{n-1}, y_0}= 1$, or equivalently, if the product $w_1w_2$ is an element of $\W$.

Every finite word $(x_0,\dots,x_{n-1})\in \mathcal{W}_n$ determines a cylinder 
\[[x_0,\dots,x_{n-1}]=\{(y_n)_{n\geq 0}\in \Sigma_A:y_0=x_0,\dots,y_{n-1}=x_{n-1}  \}\;;\]
we say in this case that the cylinder has size (or length) $n$.
The set of all cylinders of size $n$ is denoted by $C_n$; the collection $\bigcup_{n\geq 1} C_n$ forms a basis of the topology of $\Sigma_A$.

Throughout the paper we fix  a subshift of finite type $(\Sigma_A,\sigma)$. 

\subsection{Thermodynamic formalism for a single quasimorphism}
\label{subsec:thforquasim}
The notion of quasimorphism on the language of a subshift of finite type, introduced here in Definition~\ref{def:quasimorphism}, was recently studied thoroughly by the second and fourth author in~\cite[Section 4]{TermoEqSt}, where the study of an associated thermodynamic formalism was developed extensively. Such a notion is a natural adaptation of the homonymous notion for groups, which made its first appearance in Ulam's collection of mathematical problems~\cite[Chapter VI]{Ulam}\footnote{Ulam's \emph{stability problem} famously asks how close are quasimorphisms (with values in arbitrary metric groups) to genuine group morphisms; for recent developments around this question, the reader is referred to~\cite{Burger-Ozawa-Thom,Fujiwara-Kapovich,Chapman-Lubotzky}.}, and attracts nowadays widespread attention for its central role in the theory of bounded cohomology; the interested reader may consult Frigerio's monograph~\cite{Frigerio}, in addition to~\cite{TermoEqSt}. In our context, a central idea is that quasimorphisms are equivalent to certain invariant measures for the underlying dynamics, which in our case is given by the shift map, at least modulo some bounded error; this point of view was proposed by the second and the fourth author in \cite{TermoEqSt,thermodynamiclattice}, and gives a convenient framework to use techniques of thermodynamic formalism for the study of quasimorphisms.

A bounded function $\W\to \R$ is plainly a quasimorphism, and so is any function at bounded uniform distance from a quasimorphism. Less trivially, H\"older-continuous functions on $\Sigma_A$ give rise naturally to quasimorphisms.

\begin{example}[Hölder potentials as quasimorphisms]\label{ex.holder}
Let $r\ \colon\ \Sigma_A \to \R$ be a H\"older-continuous function. For every integer $n\geq 1$, let 
\[S_nr\ \colon\ \Sigma_A\to \mathbb{R}\;, \quad \underline{x}\mapsto \sum_{i=0}^{n-1}r\circ \sigma^i(\underline{x})\]
be the $n$-th Birkhoff sum of $r$. Define a function $\p_r \colon \W \to \R$ as follows\footnote{Here we abuse notation in that $\psi_r$ does not depend just on $r$.}: if $(x_0, \ldots, x_{n-1}) \in \W_n$, choose $\underline{x}_w \in \Sigma_A$ belonging to the cylinder $[x_0, \ldots, x_{n-1}]$ and set
\[
\p_r(w) = S_nr(\underline{x}_w).
\]
A different choice of the elements $\underline{x}_w$ produces a function within bounded uniform distance of $\p_r$, as can be directly checked from the H\"older property of $r$; using this fact, it is easy to verify that $\psi_r$ is a quasimorphism. Observe that, if $r$ takes strictly positive values, then $\p_r$ belongs to the cone $\qm^{++}$.

More generally, the construction is easily seen to apply to any $r \colon \Sigma_A\to \R$ with the \emph{Bowen property} (cf.~\cite[Lemma 1.15]{EquSta}): there is $C>0$ such that, for all $n\geq 1$ and all $\underline{x},\underline{x}'\in \Sigma_A$ in the same cylinder $[x_0,\dots,x_{n-1}]$, $|S_nr(\underline{x})-S_nr(\underline{x}')|\leq C$. 

Notice that, for a Hölder-continuous function $r$, the resulting quasimorphism $\psi_r$ is bounded if and only if $r$ is a continuous coboundary, namely if and only if there is a continuous function $u\colon \Sigma_A\to \R$ such that $r=u-u\circ \sigma$; this follows at once from Livšic's celebrated theorem~\cite{Livsic} on the cohomology of hyperbolic systems. The assignment $r\mapsto \psi_r$ induces thus a linear injective map from the quotient vector space of Hölder-continuous functions $\Sigma_A\to \mathbb{R}$ modulo coboundaries to the quotient vector space of $\qm$ modulo bounded quasimorphisms. 
\end{example}

Quasimorphisms also appear in the literature under the guise of locally constant almost additive sequences of potentials: given a quasimorphism $\psi\colon \W\to \mathbb{R}$, the sequence of locally constant functions $(\varphi_n)_{n\geq 1}$ defined as $\varphi_n((x_k)_{k\geq 0})=\psi(x_0,\dots,x_{n-1})$ is almost additive in the sense of Barreira~\cite[Definition 10.1.1]{Barreira_book}, namely there exists $C>0$ such that 
\[
  \varphi_{m}\circ \sigma^n (x)+\varphi_n(x)-C\leq \varphi_{n+m}(x)\leq \varphi_{m}\circ \sigma^n (x)+\varphi_n(x)+C
\]
for all $x\in \Sigma_A$ and all $m,n\geq 1$. For almost additive sequences, there has recently been much interest in developing a non-additive thermodynamic formalism, chiefly with a view towards potentials arising from matrix cocycles; we refer, for instance, to the works of Cao, Feng and Huang~\cite{Cao_2008}, Barreira and Doutor~\cite{Barreira_2009}, and Feng~\cite{Feng_2011}. We remark that the pressure of a quasimorphism, introduced in Definition~\ref{def:pressurefunction}, coincides with the  topological pressure of the associated almost additive sequence of locally constant functions. The latter was introduced by Falconer~\cite{Falconer}, and then more thoroughly developed by Barreira~\cite{Barreira_nonadditive}.

\begin{remark}
  In~\cite[Section 6]{TermoEqSt}, the notion of pressure for a quasimorphism is defined summing over the periodic words. Using that the subshift is topologically mixing, it is easy to see that this notion is equivalent to Definition~\ref{def:pressurefunction}; see \textsection\ref{ssec:Gibbsconstants} below.
\end{remark}

We will be interested in the properties of the pressure function. Key technical tools for its study are the following theorems.

Recall that the Kolmogorov--Sinai entropy of $\sigma$ with respect to a measure $\mu\in \Minv$ is 
\[
  h_{\mu}(\sigma)=\lim_{n\to\infty}-\frac{1}n\sum_{w\in\W_n}\mu([w])\log\mu([w]).
\]
Given $\p\in \qm$, the sequence with general term $a_n=\int \psi_n \dd \mu$ verifies
\[
  \forall n,m,\ |a_{n+m}-a_n-a_m|<C
\]
for some constant $C$, therefore
\[
  \exists \mu(\psi)\coloneq \lim_{n}\frac{a_n}{n}=\lim_{n}\frac{1}{n}\int \psi_n \dd \mu.
\]

\begin{theorem}[Variational principle]\label{thm:variationalprinciple}
Let $\p\colon \W\to \mathbb{R}$ be a quasimorphism. Then 
\[
\Pr(\p)=\sup\{h_{\mu}(\sigma)+\mu(\p):\mu\in \Minv\}.
\]
\end{theorem}

\smallskip

\begin{theorem}\label{thm:propiedadesmedidamul}
Let $\p \colon \W\to \mathbb{R}$ be a quasimorphism. There exists $\mu_{\p}\in \Minv$ with full support satisfying the following properties:
\begin{enumerate}
\item the measure-preserving system $(\Sigma_A,\sigma,\mu_{\p})$ has the Bernoulli property, and in particular $\mu_{\p}$ is ergodic for $\sigma$; 
\item there exists $E>0$ such that, for every $n\geq 1$ and $w \in \W_n$,
\begin{align}\label{eq:gibbs1fs}
 E^{-1}&\leq \frac{\mu_{\p}([w])}{\exp\left(\p(w)-n\Pr(\p)\right)} \leq  E\;;
  \end{align}
\item $\mu_{\p}$ is the unique element $\mu$ of $\Minv$ satisfying 
\begin{equation*}
\Pr(\p)=h_{\mu}(\sigma)+\mu(\p)\;.
\end{equation*}
\end{enumerate}
If $\phi$ and $\p$ are two quasimorphisms, then equality $\mu_{\phi}=\mu_{\p}$ holds if and only if the quasimorphism $\p-\phi$ is bounded. 
\end{theorem}

We refer to the measure $\mu_{\p}$ of the theorem as the Gibbs state or Gibbs measure associated to $\p$, a nomenclature which in the literature points to the property~\eqref{eq:gibbs1fs}, or as the equilibrium measure or equilibrium state associated to $\p$, a terminology used routinely to refer to the uniqueness property of Theorem~\ref{thm:propiedadesmedidamul}(3).

Define for $x\in \Sigma, k\in\N_{>0}$
\begin{align}
f_\p^k(x)=\log \frac{\mu_\p([x_0\cdots x_{k-1}])}{\mu_\p([x_1\cdots x_{k-1}])}.
\end{align}

\begin{theorem}\label{pro:potencialloo}
The sequence $(f_\p^n)_{n}$ converges both $\mu_\p$ a.e. and in $L^p, \forall p\in[1,+\infty)$ to some function $f_\p\in L^\infty(\mu_\p)$. 
Moreover, there exists $C >0$ such that for every $n\in \N$ and $\mu_\p$ almost every $x\in\Sigma_A$, it holds 
\[
  |S_nf_\p(x)+n\Pr(\p)-\p(x_{(n)})|\leq C.
\]
\end{theorem}

For arbitrary invariant measures we have the following.

\begin{theorem}\label{pro:potencialloogeneral}
  Fix $\mu \in \Minv$. Then, there exists $f_{\mu,\p}\in L^{\oo}(\mu)$ so that for $\mu$-almost every point, for every $n\in\N$,
\[
 |S_nf_{\mu,\p}(\underline x)-\p(x_0\ldots x_{n-1})|\leq D(\p).
\]
\end{theorem}

See the recent \cite{TermoEqSt} for the proof of these results.

\begin{remark}\label{rem:sameintegral}
  If $f_{\mu,\p}\in L^{\oo}(\mu)$ is as in the previous theorem, then it follows in particular that
  \[
  \mu(\p)=\int f \dd\mu. 
  \]
\end{remark}

\subsubsection{The equilibrium state and the Gibbs property}\label{ssec:Gibbsconstants}
Let us give an outline of the construction of the equilibrium measure referred to in Theorem \ref{thm:propiedadesmedidamul}. More details can be found in \cite[Section 6]{TermoEqSt}.

In the following we will write $\Fix_n$ for the set of periodic words of length $n$ and use $\Fix$ to denote the collection of all periodic words.
\begin{definition}
For $\p \in \qm$ set
\[
P_n(\p) = \sum_{w \in \Fix_n} e^{\p(w)} \ \ \text{ and } \ \ \widetilde{P}_n(\p) = \sum_{w \in \W_n} e^{\p(w)}
\]
for each $n \ge 1$.
\end{definition}

Let $M \ge 1$ be the mixing constant (that means $A^k$ has positive entries, for $k\ge M$) and define 
\begin{align}\label{eq:normM}
\|\p\|_{\le M} = \max\{\p(w) : |w|\le M\}.
\end{align}
Clearly, $P_n(\p)\leq \widetilde{P}_n(\p)$, and for every $n\geq 1$,
\begin{align}\label{eq:basicPP}
e^{-\|\p\|_{\le M}-D_{\p}}\widetilde{P}_n(\p)\le\widetilde{P}_{n+M}(\p)\le e^{\|\p\|_{\le M}+D_{\p}}\widetilde{P}_n(\p).
\end{align}

\begin{lemma}\label{lem.basic}[Basic relation between $P, \widetilde{P}$]

The following holds for each $n\ge 1$:
\[
  P_{n+M}(\p)\le\widetilde{P}_n(\p)\le e^{\|\p\|_{\le M} + D_\p}P_{n+M}(\p)\le  e^{2\|\p\|_{\le M} + 2D_\p}\widetilde P_{n}(\p)
\]
\end{lemma}
\begin{proof}
For each $w_1\in \W_n$ we can find $w_2\in \W_M$ so that $w_1w_2\in \Fix_{n+M}$. Therefore
\[
  P_{n+M}(\p)\geq \sum_{w_1\in\W_n} e^{\p(w_1w_2)}\geq e^{-|\p\|_{\le M} - D_\p}\widetilde{P}(\p). 
\]
The other inequality is a consequence of \eqref{eq:basicPP}.
\end{proof}

\begin{lemma}\label{lem.boundPP}
We have that
\[
P_n(\p) \le \widetilde{P}_n(\p) \le e^{2\|\p\|_{\le M} + 2D_\p}P_{n}(\p)  
\]
for all $n\ge M$.
\end{lemma}

\begin{proof}
Write $n=(n-M)+M$ and use the previous inequality to get 
\begin{align*}
P_n(\p) \le e^{\|\p\|_{\le M} + D_\p} \widetilde P_{n-M}(\p)\le e^{2\|\p\|_{\le M} + 2D_\p}P_{n}(\p)
\end{align*}
\end{proof}

\begin{lemma}\label{lem.subsup}
For all $n,m \ge M$,
\[
e^{-3(D_\p+ \|\p\|_{\le M})}\le \frac{\widetilde{P}_{n+m}(\p)}{\widetilde{P}_n(\p)  \widetilde{P}_m(\p)} \le e^{3(D_\p + \|\p\|_{\le M})}
\]
The same is true for $P$ (with constant $6$ instead of $3$).
\end{lemma}

\begin{proof}
For the upper bound notice that
\[
\widetilde{P}_{n+m}(\p) \le e^{D_{\p}} \widetilde{P}_n(\p) \widetilde{P}_m(\p).
\]
For the lower bound
\begin{align*}
\widetilde{P}_n(\p) \widetilde{P}_m(\p) &\le e^{2D_\p+ \|\p\|_{\le M}} \widetilde{P}_{n+m+M}(\p)\\
&\le e^{3D_\p + 2\|\p\|_{\le M}} \widetilde{P}_{n+m}(\p) 
\end{align*}
where in the first inequality we use mixing to find, given $w_1 \in \W_n, w_2 \in \W_m$, a word $w' \in \W_M$ with $w_1 w' w_2 \in \W$ (and used the quasimorphism property). This concludes the proof for $\widetilde{P}$ and by Lemma \ref{lem.boundPP} the same holds for $P$.
\end{proof}

We deduce,  by subadditivity, that 
\[
\Pr(\p) = \lim_{n\to\infty} \frac{1}{n} \log P_n(\p) = \lim_{n\to\infty} \frac{1}{n} \log \widetilde{P}_n(\p) \in \mathbb{R}
\]
exists.

Now we construct the measure.
 \begin{definition}
 For $N \ge 1$, let 
\[
 \mu_\p^N = \frac{1}{P_N(\p)} \sum_{w \in \Fix_n} e^{\p(w)} \delta_{x_w}
 \]
 where $\delta_{x_w}$ is the Dirac measure based at the periodic orbit $x_w= (w,w,\ldots) \in \Sigma_A$ obtained from repeating $w$.
 \end{definition}

We note that, given $w_1 \in \W_n$ and $N \ge M$ we have that
\begin{align*}
P_{n+N}(\p) \mu_\p^{n+N}[w_1] &= \sum_{w_1w_2 \in \Fix_{n+N}} e^{\p(w_1w_2)}\le e^{D_\p} e^{\p(w_1)} \widetilde{P}_N(\p)\le e^{2D_\p + \|\p\|_{\le M}} e^{\p(w_1)}  P_{N+M}(\p)\\
&\le e^{3(D_\p + \|\p\|_{\le M})} e^{\p(w_1)}  P_{N}(\p).
\end{align*}

Similarly, one can prove
\[
P_{n+N}(\p) \mu_\p^{n+N}[w_1]\ge e^{-6(D_\p- \|\p\|_{\le M})} e^{\p(w_1)}
\]

Combining the above we get the following.
\begin{lemma}
There exist $C_1,C_2>0$ so that the following holds. If $\p\in\qm$ then, for any $n, w \in \W_n$ and $N \ge M$,
\[
e^{-C_1(D_\p + \|\p\|_{\le M})}\le \frac{\mu_\p^{n+N}[w]}{e^{\p(w) - P_n(\p)}} \le e^{C_2(D_\p + \|\p\|_{\le M})}.
\]
\end{lemma}

Note that, by the quasimorphism property
\[
\sup_{n\ge 1}|P_n(\p) - n\Pr(\p)| \le D_\p.
\]
So we can alter the above inequality to
\[
e^{-C_1(D_\p + \|\p\|_{\le M})}\le \frac{\mu_\p^{n+N}[w]}{e^{\p(w) - n\Pr(\p)}} \le e^{C_2(D_\p + \|\p\|_{\le M})}.
\]
Then, using the fact that cylinder sets are both open and closed, we have that any accumulation point of $(\mu_\p^{N})_N$ verifies the same inequalities. It is possible to prove that such accumulation points are ergodic, and since the previous inequality shows that they are absolutely continuous with respect to each other, there is only one accumulation point $\mu_\p$, which thus verifies

\begin{equation}\label{eq.gibbs}
e^{-C_1(D_\p + \|\p\|_{\le M})}\le \frac{\mu_\p[w]}{e^{\p(w) - n\Pr(\p)}} \le e^{C_2(D_\p + \|\p\|_{\le M})}
\end{equation}
for all $w \in \W$.

\subsection{The pressure function}
\label{subsec:pressurefunction}
We now turn our attention to the pressure function of Definition~\ref{def:pressurefunction}. We propose here to establish that it is convex and of class $C^1$ on $\mathbb{R}$, but before embarking on it, some preliminary comments are in order. The value of the pressure function at $0$ is the solution to the equation $P(-s\psi)=0$, which in the case of $-\psi$ arising from a Hölder-continuous potential is known in the literature as \emph{Bowen's equation}, introduced by Bowen in his seminal work on quasicircles~\cite{Bowen-quasicircles} and subsequently proven to yield the Hausdorff dimension of certain invariant sets in a variety of dynamical situations. 

\begin{example}
Define $\psi\colon \W\to \mathbb{R}$ as $\psi(w)=n$ for all $w\in \W_n$ and all $n\geq 1$. Then $\psi\in \qm^{++}$, and it is immediate to check that $\Pr(\phi,\psi)(t)=\Pr(t\phi)$ for all $t\in \mathbb{R}$. In the case of $\phi$ generated by a Hölder-continuous potential $r$ as in Example~\ref{ex.holder}, our definition of pressure function recovers the classical pressure function for $r$, the study of which is ``the main object of thermodynamic formalism", to quote Ruelle in the introduction to~\cite{Ruelle-book}.
\end{example}

When $\phi, \p$ are obtained from a H\"older function as in Example~\ref{ex.holder}, the pressure function $\Pr(\phi,\p)$ has been investigated in the monograph~\cite{PP}, where it is proven to be a real analytic function \cite{PP}. There is a rich theory of thermodynamic formalism associated to this case.  More recently, the pressure function has been studied by Barreira and Doutor~\cite[Section 4.1]{Barreira_2009} as part of their pursuit to establish a conditional variational principle for the dimension spectra of almost additive sequences. In particular, their setting is at the same time more general and more restrictive than ours: instead of a single quasimorphism $\phi$, they deal with an arbitrary finite collection of almost additive sequences, on the other hand they only contemplate quasimorphisms $\psi$ arising from Hölder-continuous potentials as in our Example~\ref{ex.holder}. 

Let us prove directly that the pressure curve is $C^{1}$. First of all, we claim that for every $t$ there is unique value $\beta(t)$ so that $P(t\phi-\beta(t)\p)=0$. Fix $t$ and consider $F(s)=P(t\phi-s\p)$: $F$ is continuous. Take $s'>s$ and compute
\begin{align*}
F(s')&=\lim_{n\to\infty} \frac{1}{n} \sum_{w \in \W_n} e^{t\phi-s'\p(w)}=\lim_{n\to\infty} \frac{1}{n} \sum_{w \in \W_n} e^{t\phi-s\p(w)}e^{(s-s')\p(w)}\\
&\leq \limsup_{n\to\infty} \frac{1}{n} \sum_{w \in \W_n} e^{t\phi-s\p(w)}e^{(s-s')(nC-\tau)}=F(s)-\tau(s'-s).
\end{align*}
It follows that $\lim_{s\to -\oo}F(s)=+\oo, \lim_{s\to+\oo}F(s)=-\oo$, therefore by continuity and strict monotonicity, $F$ has a unique zero $s(t)=\beta(t)$.

Next, we use the variational principle. Note that, since $\p\in \qm^{++}$, for every $\mu\in \Minv$ we have 
\begin{align}\label{eq:positiveintegral}
\nonumber\mu(\psi) &=\lim_{n\to\infty}\frac{1}{n}\sum_{\omega \in \W_n}\int_{[\omega]}\psi(\omega) \ d\mu \geq \limsup_{n\to\infty}\frac{1}{n}\sum_{\omega \in \W_n}\int_{[\omega]}(\tau n- C)\ d\mu\\
&=\limsup_{n\to\infty}\frac{\tau n-C}{n}\sum_{\omega \in \W_n}\mu([\omega])=\tau>0.
\end{align}
Applying \Cref{thm:variationalprinciple} to the quasimorphism $t\phi-\beta(t)\p$, we get 
\[
 0\geq h_{\mu}(\sigma)+\mu(t\phi-\beta(t)\p)=h_{\mu}(\sigma)+t\mu(\phi)-\beta(t)\mu(\p)\;,\]
from which 
\[
\beta(t)\geq \frac{h_{\mu}(\sigma)+t\mu(\phi)}{\mu(\p)\;,}  
\]
with equality if and only if $\mu=\mu_t$, the equilibrium state of $t\phi-\beta(t)\p$. Writing $a_{\mu}=\mu(\phi)/\mu(\psi)$ and $b_{\mu}=h_{\mu}(\sigma)/\mu(\psi)$ for all $\mu\in \Minv$, it follows that 
\begin{align}\label{eq:beta}
\beta(t)=\sup_{\mu\in \Minv}\{A_{\mu}(t)=a_{\mu}t+b_{\mu}\},
\end{align}
hence $\beta(t)$ is the supremum of a family affine functions, and is therefore convex. 

We fix $t_0$ and consider $\beta(t_0)=A_{\mu_{t_0}}(t_0)$. For every $t$ we get
\[
  \beta(t)\geq A_{\mu_{t_0}}(t)=(t-t_0)a_{\mu_{t_0}}+\beta(t_0)
\]
that is, the graph of $\beta$ is over the line $t\mapsto (t-t_0)a_{\mu_{t_0}}+\beta(t_0)$. From here we deduce that $a_{\mu_{t_0}}=\frac{\mu_{t_0}(\phi)}{\mu_{t_0}(\p)}$ is a subgradient of $\beta$ at $t_0$. It is not much work to show that in fact $a_{\mu_{t_0}}$ is the derivative of $\beta$ at $t_0$.

\begin{lemma}
It holds $\beta'(t_0+)=\beta'(t_0-)=a_{\mu_{t_0}}$.  
\end{lemma}

\begin{proof}
Let $(r_n)_n \subset\R_{>0}$, and for every $n$ let $\mu_n$ be the equilibrium measure of $\beta(t_0+r_n)$, thus $\beta(t_0+r_n)=A_{\mu_n}(t_0+r_n)$. Then, since $\beta(t_0)\geq A_{\mu_n}(t_0)$
\begin{align*}
\frac{\beta(t_0+r_n)-\beta(t_0)}{r_n}\leq\frac{A_{\mu_n}(t_0+r_n)-A_{\mu_n}(t_0)}{r_n}=a_{\mu_n}. 
\end{align*}
Likewise, $\beta(t_0+r_n)\geq A_{\mu_{t_0}}(t_0+r_n), \beta(t_0)=A_{\mu_{t_0}}(t_0+r_n)$, and 
\begin{align*}
a_{\mu_{t_0}}\leq \frac{\beta(t_0+r_n)-\beta(t_0)}{r_n}\leq a_{\mu_n}.
\end{align*}
It is no loss of generality to assume that $(\mu_n)_n$ is convergent to some $\sigma$-invariant measure $\nu$. Note that 
\begin{enumerate}
  \item $\Minv\ni\mu \mapsto \mu(\p), \mu(\psi)$ are continuous (see Proposition $5.7$ in \cite{TermoEqSt}).
  \item $\Minv\ni\mu\mapsto h_{\mu}(\sigma)$ is upper-semicontinuous (because $\sigma$ is expansive).
\end{enumerate}
Hence, $A_{\nu}(t_0)\geq \limsup_{n\to\infty} A_{\mu_n}(t_0)$. On the other hand we can compute
\begin{align*}
A_{\mu_n}(t_0)=A_{\mu_n}(t_0+r_n)-r_n a_{\mu_n}=\beta(t_0+r_n)-r_n a_{\mu_n}
\end{align*}
and since $(a_{\mu_n})_n$ is bounded, we deduce $\lim_{n\to\infty} A_{\mu_n}(t_0)=\lim_{n\to\infty} \beta(t_0+r_n)=\beta(t_0)$. Hence, 
\[
  A_{\nu}(t_0)\geq \beta(t_0)
\]
and by uniqueness of the maximizer, $\nu=\mu_{t_0}$. We have proved
\[
  \beta'(t_0+)=\lim_{n\to\infty} \frac{\beta(t_0+r_n)-\beta(t_0)}{r_n}=a_{\mu_{t_0}}.
\]
Likewise, $\beta'(t_0-)=a_{\mu_{t_0}}$.
\end{proof}

Observe that $\beta'$ is monotone, and since it has the intermediate value property, it cannot have jumps. Hence it is continuous and $\beta$ is $C^1$.

\begin{corollary}\label{cor:derivativeofpressure}
  The pressure function $P(\phi,\p)$ is of class $C^{1}$. Moreover, 
  \[
  P(\phi,\p)'(t)=\frac{\int \phi \dd\mu_{t}}{\int \p \dd\mu_{t}}
  \]
  where $\mu_t$ is the equilibrium state of $t\phi-P(\phi,\p)(t)\p$. 
\end{corollary}

\subsection{The cocycle associated to a linear growth quasimorphism}\label{ssec:cocyclepositive}
We finish this part by characterizing the elements in $\qm^{++}$. We will establish that these correspond to positive Bowen functions, which can be thought of as a version of the positive Livšic theorem.  Fix $\p\in \qm^{++}$ and let $\tau, C>0$ so that $\p(w)\geq \tau|w|-C$, for every $w\in\mathcal W$. Fix also a reference measure $\mathrm{m}\in \mathcal{M}(\sigma)$ and consider $f=f_{\mathrm{m},\p}:D\to \mathbb R, f \in L^{\oo}(\mathrm{m})$ as given in \Cref{pro:potencialloogeneral}, where with no loss of generality we assume that $D$ is $\sigma$-invariant and $|f|\leq R$ on $D$. 

\begin{remark}\label{rmk:Mindependentmu}
  Note that $R$ only depends on $\p$, and not on the chosen invariant measure $\mathrm m$.
\end{remark}

Recall that, due to \eqref{eq:positiveintegral}, 
\[
  \inf\{\mu(\p):\mu\in \mathcal M(\sigma)\}=\tau>0.
\]

\begin{lemma}
  There exists $N$ so that for every $x\in D$, $S_Nf(x)\geq \frac{\tau}2$.
\end{lemma}

\begin{proof}
Assume the contrary: then, for arbitrary large $n$ we can find $x^{(n)}\in D$ so that 
 \[
   S_nf(x^{(n)})<\frac{\tau}2
 \]
We want to improve the previous inequality to 
\[
  \max_{0\leq k<n} S_nf(\sigma^k x^{(n)})<\tau'<\tau.
\]
We proceed as follows. Consider 
\[
  \Psi: D\to \Omega=[-R,R]^{\mathbb N}\qquad \Psi(x)=(f(x),f(\sigma x),\cdots)
\]
and equip $\Omega$ with the product topology. Let $\tau\colon\Omega\to\Omega$ be the shift map, and let $X\colon\Omega\to \mathbb R$ be the projection onto the first coordinate, which is continuous. Then, $\Omega_f=\overline{\Psi(D)}$ is a compact $\tau$-invariant set, and for every $x\in D$,
\[
  \Psi(f(\sigma x))=\tau(\Psi(x)).
\]

Denote $\underline{y}^{(n)}=\Psi(x^{(n)})$ and consider 
\[
  \nu_n=\frac{1}n\sum_{k=0}^{n-1}\tau_*^{k}\delta_{\underline{y}^{(n)}}=\frac{1}n\sum_{k=0}^{n-1}\delta_{\Psi\left(f(\sigma^k x^{(n)})\right)}.
\]
If $\nu$ is an accumulation point of $(\nu_n)$, then it is $\tau$-invariant, and
\[
  \int X\ d \nu=\lim_{n\in S_{\nu}} \frac{1}{n}S_nf(x^{(n)})\leq \frac{\tau}{2}.
\]
With no loss of generality (using the ergodic decomposition), we assume that $\nu$ is ergodic. Fix $\underline{z}\in \Omega_f$ a generic point for $\nu$: we can then write
\[
  \frac{S_nX(\underline z)}{n}=n\int X\ d \nu+r_n\quad r_k=o(n),
\]
and therefore, given $n$ and $0\leq k<n$,
\begin{align*}
\frac{S_nX(\sigma^k\underline z)}{n}=\frac{S_{n+k}X(\underline z)-S_{k}X(\underline z)}{n}=\int X\ d \nu+\frac{r_{n+k}-r_k}{n}.
\end{align*}
Since 
\[
  \max_{0\leq k<n}\left|\frac{r_{n+k}-r_k}{n}\right|\xrightarrow[n\mapsto\oo]{}0
\]
we finally get
\[
  \max_{0\leq k<n} \frac{S_nX(\sigma^k\underline z)}{n}\xrightarrow[n\mapsto\oo]{}\int X\ d \nu\leq \frac{\tau}{2}.
\]
and therefore, for $n$ large
\[
  \max_{0\leq k<n} \frac{S_nX(\sigma^k\underline z)}{n}\leq \frac{3\tau}{4}.
\]
Take $y^{(n)}\in D$ so that 
\[
  |f(\sigma^{k}y^{(n)})-z_k|<\frac{\tau}{8}\quad k=0,\ldots, 2n.
\]
Then, for $0\leq k<n$,
\begin{align}\label{eq:inequalityBirkhoff}
 \frac{S_nf(\sigma^{k}y^{(n)})}{n}<\frac{S_nX(\tau^{k}\underline z)}{n}+\frac{\tau}{8}\leq \frac{7\tau}{8}<\tau.
\end{align}

\smallskip

Next, consider the measures
\[
  \mu_n=\frac{1}{n}\sum_{k=0}^{n-1}\sigma_*^{k}\delta_{y^{(n)}},
\]
which are supported on $D$. Take $\mu$ an accumulation point of them: $\mu$ is $\sigma$-invariant. Also,
\begin{align*}
\mu(\p)=\lim_{n\in S_{\mu}} \frac{1}{n}\int \p_n \ d\mu_n=\lim_{n\in S_{\mu}} \frac{1}{n}\int (\p_n-S_nf) \ d\mu_n+\frac{1}{n}\int S_nf \ d\mu_n\\
=\lim_{n\in S_{\mu}}\frac{1}{n}\int S_nf \ d\mu_n=\lim_{n\in S}\frac{1}{n}\sum_{k=0}^{n-1} S_nf(\sigma^k y^{(n)}) \leq \frac{7\tau}{8}<\tau,
\end{align*}
which is a contradiction, due to our assumption.
\end{proof}

Let $N$ be so that for every $x\in D$, $S_Nf(x)\geq \frac{\tau}2$, and define $u:D\to\mathbb R$,
\[
  u(x)=\frac{1}{N}\sum_{k=0}^{N-2}(N-1-k)f(\sigma^k x).
\]
Naturally, $u\in L^{\oo}(\mathrm{m})$, and
\[
  f+u\circ\sigma-u=\frac{1}NS_Nf\geq \frac{\tau}{2N}.
\]
Note that $\frac{1}NS_Nf$ has the Bowen property. We have proved the following.

\begin{theorem}[Positive Livšic theorem for Bowen functions]\label{thm:positiveLivsic}
  Let $\p\in\qm^{++}$ and $\mu\in \mathcal{M}(\sigma)$. Then, there exist $g_{\mu,\p}:D\to\mathbb R\in L^{\oo}(\mu)$ with the Bowen property, and $E,\tau'>0$ so that for every $x\in D$,
  \begin{enumerate}
    \item $g_{\mu,\p}(x)>\tau'$,
    \item $|S_ng_{\mu,\p}(x)-\p(x_0, \ldots , x_{n-1})|<E$, for every $n\geq 1$.
  \end{enumerate}
\end{theorem}

\begin{corollary}
We have that
  \[
  \qm^{++}=\{\p\in\qm:\inf\{\mu(\p):\mu\in \Minv\}>0\}.
  \]
\end{corollary}

\begin{proof}
  Indeed, fix $\mu\in \mathcal{M}(\sigma)$ of full support and consider $g_{\mu,\p}:D\to\mathbb R$ given by the theorem above. $D$ is dense, thus given $x\in \Sigma, n\geq 1$ there exists $y\in D$ so that $x_0x_1\cdots x_{n-1}=y_0y_1\cdots y_{n-1}$. It follows that
  \begin{align*}
  \p(x_0,x_1, \ldots, x_{n-1})=\p(y_0, y_1,\ldots, y_{n-1})\ge S_ng_{\mu,\p}(y)-E\ge n\tau'-E.
  \end{align*}
\end{proof}


\section{Rigidity and counting for quasimorphisms}\label{sec:RigidityAndCounting}
\subsection{Rigidity}
In this section we establish Theorems \ref{thm.rigidity} and \ref{thm.rigidity2}.
Our first result will help us to characterise the equivalence of quasimorphisms in terms of the measure classes of their associated Gibbs states.
\begin{proposition}
Take $\p, \phi \in \qm$ then $\Pr(\p,\phi)$ is a strictly convex everywhere if and only if $\p$ and $\phi$ are not roughly similar.
\end{proposition}

\begin{proof}
Note that if $\p$ and $\phi$ are equivalent then $\Pr(\p,\phi)$ is a straight line. Therefore we just need to show that if $\Pr(\p,\phi)$ is not strictly convex everywhere then $\p \sim \phi$.

To that end, suppose that $\Pr(\p,\phi)$ is not strictly convex everywhere. This implies that $\Pr(\p,\phi)$ contains a straight line segment. That is there exist $a,b,u,v  \in \R$ with $u < v$ such that $\Pr(\p,\phi)(t) = at + b$ for $t \in [u,v]$. Now consider the quasimorphisms
\[
\varphi_u = -u \p + \Pr(\p,\phi)(u) \phi, \ \ \varphi_v = -v \p + \Pr(\p,\phi)(v) \phi 
\ \ \text{ and }  \ \ \varphi_w= -w \p + \Pr(\p,\phi)(w) \phi
\]
where $w = (u+v)/2$. Note that, by assumption,
\[
\Pr(\p,\phi)(w) = \frac{\Pr(\p,\phi)(u) + \Pr(\p,\phi)(v)}{2}.
\]
Now let, $\mu_u, \mu_v, \mu_w$ be the Gibbs measures associated to the above quasimorphisms.

For any $x \in \Sigma$ the Gibbs property implies that
\[
\frac{\mu_u[x_0\ldots x_{n-1}]}{\mu_w[x_0\ldots x_{n-1}]} \cdot \frac{\mu_v[x_0\ldots x_{n-1}]}{\mu_w[x_0\ldots x_{n-1}]} 
\]
is bounded away from $0$ and $\infty$ for all $n \ge 1$.

Then, by~\cite[Theorem 1.8]{Heinonen}, this implies that neither $\mu_u$ or $\mu_v$ are mutually singular with respect to $\mu_w$. By ergodicity this implies that $\mu_w = \mu_v = \mu_u$. In particular there exists $E > 1$ such that 
\[
E^{-1} \le \frac{\mu_u [w]}{\mu_v[w]} \le E
\]
for any $w \in \W$. This implies, by the Gibbs property, that $\p \sim \phi$.
\end{proof}

We then immediately deduce the following corollary.

\begin{corollary}\label{cor.line}
The following are equivalent:
\begin{enumerate}
\item  $\Pr(\p,\phi)$ is a straight line;
\item  $\Pr(\p,\phi)$ contains a straight line segment; and,
\item  $\p \sim \phi$.
\end{enumerate}
\end{corollary}

We now study natural some quantities associated to quasimorphisms. Let us start with their growth rates \eqref{eq:hquasi}. For $\p \in \qm^{+}$ we define
\begin{equation}
\label{eq:hp}
h(\p) = \lim_{T \to\infty} \frac{1}{T} \log |\{ w \in \W : \p(w) \leq T\}|
\end{equation}
if $ |\{ w \in \W : \p(w) \leq T\}|$ is finite for each $T >0$, in which case the limit exists by subadditivity, and set $h(\p) = \infty$ otherwise. We can also realise $h(\p)$ as
\[
h(\p) = \lim_{T \to\infty} \frac{1}{T} \log \#\{ \{w\} \in \Per : \ov{\p}\{w\} \leq T\};
\]
see Lemma \ref{lem.growcomp}.

The following proposition shows that finite growth rate is equivalent to linear growth for quasimorphisms.
\begin{proposition}\label{prop.growthrate}
For a quasimorphism $\p \in \qm^{+}$ the following are equivalent:
\begin{enumerate}
\item $\p$ belolngs to $\qm^{++}$;
\item $h(\p) <\infty$;
\item $\{ w \in \W : \p(w) \leq T\}$ is finite for each $T >0$;
\item $\{ w \in \W : \p(w) \leq T\}$ is finite for some $T\ge D_{\p}$.
\end{enumerate}
Hence,  $\qm^{++} = h^{-1}(0,\infty)$ and $\qm^{+}\setminus \qm^{++}=h^{-1}(\{\oo\})$.
\end{proposition}

\begin{proof}
$(1) \implies (2)$ and $(2) \implies (3) \implies (4)$ are simple. We therefore just need to check the implication $(4) \implies (1)$.

Suppose $(4)$ holds. Then there exists $N \ge1$ and $\epsilon >0$ such that $w\in \W_N$ implies $\p(w)>D_{\p}+\epsilon$. Given $w \in \W$ we can write $w = w_1 \cdots w_k w_{k+1}$ where $w_1,\ldots, w_k \in \W_N$, $k = \lfloor |w|/N\rfloor$ and $|w_{k+1}| < N$, i.e. we write $w$ as a product of words of length $N$ with a final remainder word of length less than $N$. Then, by applying the  quasimorphism identity $k+1$ times we see that
\begin{align*}
\p(w) &\ge \sum_{j=1}^k \p(w_j)  - (k+1)D_\p + \p(w_{k+1}) \\
&\ge k(D_\p) - (k+1)D_\p + \p(w_{k+1}) \\
&=k\epsilon -D_\p + \p(w_{k+1}) \ge \left(\frac{|w|}{N}-1\right)\epsilon - D_\p + \p(w_{k+1})
\end{align*}
and it follows that $\p \in \qm^{++}$. 
\end{proof}

Recall the definition of dilation or a pair of quasimorphisms $\p, \phi \in \qm^{++}$ given in \eqref{eq:defDil}:
\[
\Dil(\p,\phi) = \sup_{w \in \Per} \frac{\ov{\p}\{w\}}{\ov{\phi}\{w\}}.
\]
Note that $\Dil(\p,\phi)$ is strictly positive and finite.
ee that the dilation can be obtain as a largest stretch factor for the quasimorphisms (opposed to their homogenizations).
\begin{lemma}
We have that
\[
\Dil(\p,\phi) = \lim_{n\to\infty} \sup_{w \in \W_n} \frac{\p\{w\}}{\phi\{w\}} \ \ \text{ and } \ \ \Dil(\phi,\p)^{-1} = \lim_{n\to\infty} \inf_{w \in \W_n} \frac{\p\{w\}}{\phi\{w\}}.
\]
\end{lemma}
\noindent We leave the details of the proof to the reader.

The following result shows that the dilations are the optimal comparison constants between two quasimorphisms.
\begin{lemma}\label{lem.comp}
Given $\p, \phi \in \qm^{++}$ there exists $C >0$ such that
\[
\Dil(\p, \phi)^{-1}\p(w) - C \le \phi(w) \le \Dil(\phi, \p)\p(w) + C
\]
for all $w \in \W$.
\end{lemma}
\begin{proof}
Let $M\geq 1$ be the mixing constant of $\Sigma$. For any word $w \in \W$ there exists  $w' \in \W_M$ such that $\{ww'\} \in \Per$. Then
\begin{align*}
\phi(w) \le \phi(ww') - \phi(w') + D_\phi &\le \ov\phi\{ww'\} + 2D_\phi - \phi(w')  \\
&\le \Dil(\phi,\p)\ov\p\{ww'\}+ 2D_\phi - \phi(w')\\ 
&\le \Dil(\phi,\p)\p(ww')+ 2D_\phi + \Dil(\phi,\p)D_\p - \phi(w') \\
&\le \Dil(\phi,\p)\p(w) +2D_\phi + \Dil(\phi,\p)D_\p - \phi(w')  + \Dil(\phi,\p)\p(w') + D_\p.
\end{align*}
The required upper bound then holds with 
\[
2D_\phi + D_\p+\Dil(\phi,\p)(D_\p + \|\p\|_{\le M}\|) - \min_{w \in \W_M} \phi(w).
\]
For the lower bound
\begin{align*}
\phi(w) \ge \phi(ww') - \phi(w') - D_\phi &\ge \ov\phi\{ww'\} - \phi(w') - 2D_\phi\\
&\ge \Dil(\p,\phi)^{-1}\ov\p\{ww'\} - \phi(w') - 2D_\phi\\ 
&\ge \Dil(\p,\phi)^{-1}\p(ww') - \phi(w') - 2D_\phi - D_\p\\
&\ge \Dil(\phi,\p)^{-1}\p(w) - (\phi(w') + 2D_\phi) - \Dil(\p,\phi)^{-1}(D_\p - \p(w')) -D_\p\\
\end{align*}
and the lower bound holds with constant
\[
2 D_\phi + D_\p +  \Dil(\p,\phi)^{-1} ( D_\p - \min_{w\in\W_M}\{\p(w)\} ) + \|\phi\|_{\le M}.
\]
\end{proof}

\begin{remark}\label{rem.indep}
Notice that in the above lemma $C$ can be taken to be
\[
2 D_\phi + D_\p +  \Dil(\p,\phi)^{-1} ( D_\p - \min_{w\in\W_M}\{\p(w)\} ) + \|\phi\|_{\le M} +\Dil(\phi,\p)(D_\p + \|\p\|_{\le M}) - \min_{w \in \W_M} \phi(w).
\]
\end{remark}

\begin{remark}
The same proof shows that for $\p \in \qm^{++}$ and $\phi \in \qm$ there exists $C>0$ such that 
\[
\left(\inf_{\{u\} \in \Per} \frac{\ov\phi\{u\}}{\ov\p\{u\}}\right) \p(w) - C \le \phi(w) \le \left(\sup_{\{u\} \in \Per} \frac{\ov\phi\{u\}}{\ov\p\{u\}}\right) \p(w) + C
\]
for all $ w\in \W$.
\end{remark}
We note the following immediate corollary.
\begin{corollary}\label{cor.od}
Suppose that $\psi, \phi \in \qm^{++}$. Then, $\psi \sim \phi$ if and only if there exists $\tau > 0$ such that $\ov \psi\{w\} = \tau \ov\phi \{w\}$ for all $\{w\} \in \Per$.
\end{corollary}

We will need the following simple lemma later.
\begin{lemma}\label{lem.hombound}
A quasimorphism $\p \in \qm$ belongs to $\qm^+$ if and only if $\ov{\p}\{w\} \ge 0$ for all $\{w\} \in \Per$.
\end{lemma}
\begin{proof}
The only if statement is easy. To prove the other implication suppose that $\ov{\p}\{w\} \ge 0$ for all $\{w\} \in \Per$. Then, given any $w \in \W$ we can find $w' \in \W_{\le M}$ such that $\{ww'\} \in \Per$. Then $\ov{\p}\{ww'\} \ge 0$ and
\[
|\ov{\p}\{ww'\} - \p(ww')|  \le D_\p  \ \ \text{ and } \ \ | \p(ww') - \p(w)| \le D_\p + \max_{w \in \W_{\le M}} \p(w)
\]
where we have used the sub/superadditivity in the first inequality and the quasimorphism property in the latter.  In particular there exists $C \ge 0$ such that $\p(w) + C > 0$ for all $w \in \W$.
\end{proof}

\begin{remark}\label{rem:integralhomogeneous}
  For $\{w\}\in \Per$ denote $\delta_{\{w\}}\in\mathcal M(\sigma)$ the $\sigma$-invariant measure supported on $\{w\}$. Then it is possible to show (\cite[Example $5.10$]{TermoEqSt}) that
  \[
  \ov{\p}\{w\}=\delta_{\{w\}}(\p).
  \]
  Furthermore, by \cite{Sigmund1974}, the set $\{\delta_{\{w\}}:\{w\}\in\Per\}$ is dense in $\mathcal M(\sigma)$. It follows that for $\p\in\qm$,
  \[
  \p\in\qm^{++} \Leftrightarrow \ov{\p}\{w\} \ge 0\ \forall\{w\} \in \Per \Leftrightarrow \mu(\p)\ge 0\ \forall \mu\in\mathcal M(\sigma).
  \] 
\end{remark}

The following lemma shows that the dilations can be thought of as the asymptotic gradients of the pressure functions.

\begin{lemma}\label{lem.asympgrad}
For $\p, \phi \in \qm^{++}$ we have that
\[
\Dil(\phi, \p) = \lim_{t\to\infty} \frac{\Pr(\phi,\p)(t)}{t} \ \ \text{ and } \ \ \Dil(\p,\phi)^{-1} = \lim_{t\to -\infty} \frac{\Pr(\phi,\p)(t)}{t}.
\]
\end{lemma}

\begin{proof}
Consider $t > 0$.
By Lemma \ref{lem.comp} there exists $C_t > 0$ such that 
\[
\sum_{|\p(w) - n| < R} e^{t\phi(w)} \le \sum_{|\p(w) - n| < R} e^{t\Dil(\phi,\p)\p(w) + tC_t} \le \ov{C}_t e^{tn \Dil(\phi,\p) + nh(\p)} 
\]
for any $n\ge 1$. Here $\ov{C}_t$ is a constant depending only on $t$. Taking exponential growth rates of the above gives
\[
\Pr(\phi,\p)(t) \le t \Dil(\phi,\p) + h(\p).
\]
This implies that
\[
\Dil(\phi, \p) \ge \limsup_{t\to\infty} \frac{\Pr(\phi,\p)(t)}{t}.
\]

Now note that for any $\epsilon >0$ there are infinitely many $w \in \W$ for which we have $\phi(w) \ge \Dil(\phi,\p)\p(w) - \epsilon$. In particular there is $C_t$ depending only on $t$ and infinitely many $n \ge 1$ for which
\[
\sum_{|\p(w) - n| < R} e^{t\phi(w)} \ge C e^{nt(\Dil(\phi,\p)-\epsilon)}.
\]
Taking exponential growth rates of either side of this equation and using the fact that $\epsilon$ is arbitrary we see that
\[
\Dil(\phi, \p) \le \liminf_{t\to\infty} \frac{\Pr(\phi,\p)(t)}{t}.
\]
Combining this with the other inequality proves that 
\[
\Dil(\phi, \p) = \lim_{t\to\infty} \frac{\Pr(\phi,\p)(t)}{t} 
\]
as required. The other equality follows in the same way.
\end{proof}

Note that in the above lemma we established the inequality
\begin{equation} \label{eq.asymp1}
\Pr(\phi,\p)(t) \le t \Dil(\phi,\p) + h(\p)
\end{equation}
when $t >0$. We can similarly show that
\begin{equation} \label{eq.asymp2}
\Pr(\phi,\p)(t) \le t \Dil(\p,\phi)^{-1} - h(\phi)
\end{equation}
for $t <0$.
\begin{remark}
When $\p \in \qm^{++}$ and $\phi \in \qm$ the same proof shows that
\[
\lim_{t\to\infty} \frac{\Pr(\phi,\p)(t)}{t} = \sup_{\{u\} \in \Per} \frac{\ov\phi\{u\}}{\ov\p\{u\}}\ \ \text{ and } \ \ \lim_{t\to -\infty} \frac{\Pr(\phi,\p)(t)}{t} = \inf_{\{u\} \in \Per} \frac{\ov\phi\{u\}}{\ov\p\{u\}}.
\]
\end{remark}

\begin{remark}\label{rem.asympupper}
By convexity and the previous lemma we have that
\[
\Pr(\phi,\p)(t) \le \Dil(\phi,\p) t + \Pr(\phi,\p)(0)
\]
for all $t \ge 0$.
\end{remark}

As a corollary we deduce the following.

\begin{lemma}\label{lem.bounds}
Given $\p, \phi \in \qm^{++}$ we have that
\[
\Dil(\phi,\p) \ge \frac{h(\p)}{h(\phi)}
\]
and we have equality if and only if $\p \sim \phi$. In particular,  $\p \sim \phi$ if and only if there exists $\tau \in \R$ such that  $\ov{\p} = \tau \ov{\phi}$.
\end{lemma}
\begin{proof}
The inequality follows from convexity of $\Pr(\phi,\p)$ and the fact that $\Pr(\phi,\p)(0) = h(\p)$ and $\Pr(\phi,\p)(-h(\phi)) = 0$. The final statement is then a consequence of Corollary \ref{cor.line}.
\end{proof}

\subsubsection{Identifying the drift with the derivative}
Suppose that $\p \in \qm^{++}$ and $\phi \in \qm$.
For each $t \in \R$ let $\mu_t$ be the equilibrium state for
\[
  \Phi_t = t\phi - \beta(t)\p
\]
where $\beta(t) = \Pr(\phi,\p)(t)$.

Without loss of generality, we will, in this section, make the following standing, normalisation assumption:

\noindent\textbf{Assumption:} $\Pr(\phi) = 0$, i.e. $h(\phi) =1$.

\smallskip 

Let us record some properties of the family $(\mu_t)_t$. 

\begin{lemma}\label{lem.gibbuni}
For each $t \in \R$ there exists $E_t > 1$ such that for any $ w \in \W$,
\begin{align}\label{eq:eqmut}
E_t^{-1}&\leq \frac{\mu_t([w])}{\exp\left(\Phi_t(w)\right)} \leq  E_t.
  \end{align} 
Furthermore, each $E_t$ can be chosen such that for any closed interval $I \subset \R$, $\sup_{t \in I} E_t < \infty$.
\end{lemma}

\begin{proof}
This follows directly from equation $(\ref{eq.gibbs})$ above.
\end{proof}

\begin{lemma}
As $t \to t_0$ we have that $\mu_t \to \mu_{t_0}$ in the weak $*$ topology.
\end{lemma}
\begin{proof}
Since cylinder sets are both open and closed, any weak $*$ limit of $\mu_t$ as $t\to t_0$ satisfies the same Gibbs property. In fact, by Lemma \ref{lem.gibbuni} the Gibbs constants are uniformly bounded away from zero and infinity. It is easy to check that any two $\sigma$-invariant measures satisfying the Gibbs property for fixed $\p$ and with same exponent are mutually absolutely continuous. Hence, any two limit points for $\mu_t$ must be identical. 
\end{proof}

The following follows easily from the definition of $\qm^{++}$.
\begin{lemma}\label{lem.grow1}
Fix $R>0$. Then there exists $C_R > 0$ depending only on $R$ such that the following holds: if $w = (x_1, \ldots, x_k) \in \W_k$ and $n \ge 1$ then there are at most $C_R$ values of $j$ such that
\[
|\p(x_1, \ldots, x_j) - n | < R.
\]
\end{lemma}

In the following two lemmas we use the notation $\asymp$ to denote that two sequences (in $n$) have asymptotically the same growth: two sequences $a_n, b_n$ satisfy $a_n \asymp b_n$ if there exists $C > 1$ such that eventually $C^{-1} a_n \le b_n \le C a_n$. We write $a_n \preceq b_n$ to denote that there exists $L >0$ such that eventually $a_n \le L \, b_n$.

\begin{lemma}\label{lem.expgrow}
Suppose that $\p \in \qm^{++}$, $\p \in \qm$ and $t \in \R$. Then, for any fixed, sufficiently large $R > 0$,
\[
\sum_{|\p(w) - n|<R} e^{t\phi(w)} \asymp e^{n \Pr(\phi,\p)(t)}
\]
as $n \to\infty$.
\end{lemma}
\begin{proof}
Note that
\begin{align*}
\sum_{|\p(w) - n|<R} e^{t\phi(w)} &\asymp e^{n \Pr(\phi,\p)(t)} \, \sum_{|\p(w) - n|<R} e^{t\phi(w) - \Pr(\phi,\p)(t) \p(w)} \\
&\asymp e^{n \Pr(\phi,\p)(t)} \, \sum_{|\p(w) - n|<R} \mu_t[w]\\
&\asymp e^{n \Pr(\phi,\p)(t)}.
\end{align*}
Where the last inequality follows from Lemma \ref{lem.grow1}.
\end{proof}

Before we continue we record the following remarks that will be useful for applications.
\begin{remark}\label{rem.expgrow}
Since Gibbs measures give positive measure to any cylinder, the same proof as above can easily be adapted to show that 
\[
\sum_{|\p(w) - n|<R, w \in U} e^{t\phi(w)} \asymp e^{n \Pr(\phi,\p)(t)}
\]
as $n\to\infty$ where $U = \bigcup_{j=1}^k [x_j] \subset \Sigma$ is a finite, non-empty, union of one-cylinders.
\end{remark}

\begin{remark}\label{rem:identifybetazero}
  By the previous lemma, if we fix $\p\in\qm^{++}$ then for every $\phi\in\qm$, 
 \[P(\phi,\p)(0)=h(\psi).\] 
\end{remark}

\begin{remark}\label{rem:derivativeatzero}
Note that by taking the derivative at $t=0$ we get
\[
P(\phi,\p)(0) = \lim_{T\to\infty} \frac{1}{|\W_\p(T)|} \sum_{w \in \W_\p(T)} \frac{\phi(w)}{T} 
\]
\end{remark}

Before moving on we note the following lemma, which shows that similar growth estimates as the ones obtained in \ref{lem.expgrow} hold when we count over orbits with respect to the homogenization.

\begin{lemma}\label{lem.growcomp}
Suppose that $\p \in \qm^{++}$, $\p \in \qm$ and $t \in \R$. Then, for any fixed, sufficiently large $R > 0$ there exists $R' > R$ such that
\[
\frac{1}{n} \sum_{|\p(w) - n|<R} e^{t\phi(w)} \preceq \sum_{|\ov\p\{w\} - n|<R'} e^{t\ov\phi\{w\}} 
\]
and 
\[
\sum_{|\ov\p\{w'\} - n|<R} e^{t\ov\phi\{w'\}} \preceq \sum_{|\p(w) - n|<R'} e^{t\phi(w)}
\]
as $n\to\infty$. 
\end{lemma}
\begin{proof}
The second estimate is simple. For the first one, given $w \in \W$ such that $|\p(w) - n| < R$, we arbitrarily extend $w$ by some $w' \in \W_{\le M}$ (where $M$ is a number that exists by mixing) so that $ww'ww' \in \W$ and $\{ww'\} \in \Per$. The period of $\{ww'\}$ is bounded above by $Cn$ for some $C >0$ since $\p(ww')$ is approximately $n$ and $\p$ is linearly comparable to the number of symbols in the input word.  Suppose that we ran the same process with a different $\tilde{w}$ satisfying $|\p(\tilde{w}) - n| < R$ and obtained a periodic point $\tilde{w}\tilde{w}'$ such that $\{\tilde{w}\tilde{w}'\} = \{ww'\}$. Then we deduce that $\tilde{w}$ must correspond to a starting segment of one of the cyclic permutations of $ww'$. There are at most $Cn$ permutations of $ww'$ and for each of these, there are at most $C_R$ initial segments that could correspond to $\tilde{w}$ by Lemma \ref{lem.grow1}. In particular there are at most $C_R Cn$ many possible words that $w'$ could be. This implies that the map sending $w \mapsto \{ww'\}$ is at most $L n$ to $1$ for some $L >0$. Note also that there exists $C' >0$ independent of $w$ and its extension $ww'$ such that $|\p(w) - \ov\p\{ww'\}|<C'$ and $|\phi(w) - \ov\phi\{ww'\}|<C'$. Combining these facts gives
\[
\sum_{|\p(w) - n|<R} e^{t\phi(w)} \le L n e^{|t|C'} \sum_{|\ov\p\{w\} - n|<R + C'} e^{t\ov\phi\{w\}} 
\]
as required.
\end{proof}

To simplify notation, for each $n\ge 1$, we define $\ev_n : \Sigma_A \to \W_n$ by $\ev_n(x) = (x_0, \ldots, x_{n-1})$.

\begin{proposition}\label{prop.agree}
For each $t \in \R$ it holds that for $\mu_t$ almost every $x \in \Sigma$,  
\[
\lim_{n\to\infty} \frac{\phi(\ev_n(x))}{\p(\ev_n(x))}=P(\phi,\p)'(t). 
\]
\end{proposition}

\begin{proof}
Denote $f_n=\phi\circ\ev_n, g_n=\p\circ\ev_n$, and let $f,g\in L^{\oo}(\mu_t)$ be Bowen functions so that 
$(|S_nf-f_n|)_n$ and $(|S_ng-g_n|)_n$ are uniformly bounded. In particular (see remark~\ref{rem:sameintegral})), 
\[
  \mu_t(\phi)=\int f\dd\mu_t\qquad \mu_t(\p)=\int g\dd\mu_t.
\]
Since $\p\in\Qm^{++}$, it follows that $\lim_n S_ng=\infty$, and hence $\mu_t$-almost everywhere,
\[
   \lim_{n}\frac{f_n(x)}{g_n(x)}=\lim_n \frac{S_nf(x)}{S_ng(x)}=\frac{\int f\dd \mu_t}{\int g\dd\mu_t} 
\]
where in the last part we have used Birkhoff's theorem. Using Corollary \ref{cor:derivativeofpressure} we conclude that
\[
  \mu_t-a.e.(x),\quad \lim_{n}\frac{f_n(x)}{g_n(x)}=P(\phi,\p)'(t).
\]
\end{proof}

In particular we get the following, that strengthens Theorem~\ref{thm.drift}.

\begin{theorem}\label{thm.posdrift}
Suppose that $\p \in \qm^{++}$ and $\phi \in \qm^+$. Then for $\mu_{0}$ almost every $x \in \Sigma$,
\[
\lim_{n\to\infty}\frac{\phi(\ev_n(x))}{\p(\ev_n(x))} = P(\phi,\p)'(0) = \frac{\int \phi\dd\mu_0}{\int \p\dd\mu_0} := \tau(\phi/\p).
\]
Furthermore, $\tau(\phi/\p)$ is strictly positive unless $\phi$ is bounded. 
\end{theorem}

\begin{proof}

The limit exists and equals the derivative of the pressure by the previous Proposition, together with \Cref{cor:derivativeofpressure}. Due to our positiveness assumptions, $\tau(\phi/\p)\geq 0$. Suppose that this drift vanishes. Then, since the pressure $\Pr(\phi,\p)(t)$ is convex (and so its derivative is increasing), if $\tau(\phi/\p) = 0$ then the derivative of $\Pr(\phi,\p)(t)$ are $0$ at every point $t \le 0$. This implies that $\Pr(\phi,\p)(t)$ is not strictly convex and so, by Theorem \ref{thm.rigidity}, $\ov{\p}$ and $\ov{\phi}$ are proportional with proportionality constant $\tau(\phi/\p) = 0$. This means that $\ov{\phi}=0$, which in turn forces $\phi$ to be bounded.
\end{proof}

Armed with this we can conclude Theorem~\ref{thm.rigidity2}: observe that the asymptotic characterization of the drift was given in Remark~\ref{rem:derivativeatzero}.

\begin{proposition}
  If $\phi,\p \in \qm^{++}$, then $\tau(\phi/\p) \ge h(\p)/h(\phi)$. Moreover, the following are equivalent:
\begin{enumerate}
\item $\tau(\phi/\p) = h(\p)/h(\phi)$;
\item $\Pr(\p,\phi)$ is affine;
\item $\Pr(\p,\phi)$ is not strictly convex on $\mathbb{R}$; and,
\item $\phi \sim \p$.
\end{enumerate}
\end{proposition}

\begin{proof}
The function $\Pr(\phi,\p)(t)$ is convex, and  $\Pr(\phi,\p)(0) = h(\p), \Pr(\phi,\p)(-h(\phi)) = 0$. Since $\phi\in\qm^{++}$, $h(\phi)>0$, and the derivative of the pressure function at zero ($\tau(\phi/\p)$), verifies
\[
  \tau(\phi/\p)\geq \frac{h(\p)}{h(\phi)}.
\]
In case of equality, the derivative of $\Pr(\phi,\p)(t)$ is constant and equal to $\tau(\phi/\p)$ on the interval $[-h(\phi),0]$, i.e. the pressure curve is a straight line. By Corollary \ref{cor.line}, we get that $\phi\sim\p$.
\end{proof}

\subsection{Large deviations and Theorems~\ref{thm.rigidity} and~\ref{thm.rigidity2}}\label{sec.ldp}
We are now ready to prove our main results from Section \ref{sec.rigid}.

\begin{proof}[Proof of Theorem \ref{thm.rigidity}]
The equivalence of $(1), (2)$ and $(3)$ follows from Corollary \ref{cor.line}. The furthermore statement (i.e. the large deviation estimate) follows as a corollary of the large deviation principle below. Lastly, the fact that $\tau(\phi/\p)$ is strictly positive when $\phi$ is unbounded follows from Theorem \ref{thm.drift}.
\end{proof}

In order to state the large deviation principle, some terminology and notation is in order. We recall that, given a function $f\colon \mathbb{R}\to \overline{\mathbb{R}}$, the Fenchel-Legendre transform of $f$ is defined as the function
\[f^*\colon \mathbb{R}\to \overline{\mathbb{R}}\;, \quad s\mapsto \sup_{t\in \mathbb{R}}st-f(t)\;.\]
For $\p\in \qm^{++}$, recall the notation $\W_{\p}(T)=\{w\in \W:\p(w)\leq T\}$ from the introduction. The finite discrete set $\W_{\p}(T)$ is equipped with the uniform probability measure $\mu_T$; any function $A_T\colon \W_{\p}(T)\to \mathbb{R}$ becomes thus a random variable, which we indicate with
\[A_T(w)\;,\quad w\sim \mu_T\;.\] The expected value of $A_T$ with respect to $\mu_T$ is denoted $\mathbf{E}_T$.

\begin{theorem}\label{thm.ldp}
Suppose that $\p \in \qm^{++}$ and $\phi \in \qm$. Let $I(\phi,\p)\colon \mathbb{R} \to \mathbb{R}$ be the Fenchel-Legendre transform of the function $\Pr(\phi,\p) - h(\p)$. Then the family of random variables
\begin{equation}
\label{eq:familyrv}
\frac{1}{T}\phi(w)\;, \quad w\sim \mu_T,\;T>0
\end{equation}
 satisfies the large deviation principle with rate function $I(\phi,\p)$: for every open $U \subset \R$ and every closed $V \subset \R$ with $U \subset V$, 
\begin{align*}
-\inf_{s\in U}I(\phi,\p)(s) &\le \liminf_{T \to\infty} \frac{1}{T} \log \left( \frac{1}{|\W_\p(T)|} \left| \left\{ w \in \W_\p(T) : \frac{\phi(w)}{T} \in U  \right\} \right|  \right) \\
&\le \limsup_{T \to\infty} \frac{1}{T} \log \left( \frac{1}{|\W_\p(T)|} \left| \left\{ w \in \W_\p(R,T) : \frac{\phi(w)}{T} \in V  \right\} \right|  \right)  \le -\inf_{s\in V} I(\phi,\p)(s).
\end{align*}
Furthermore, the effective domain $\{s\in \mathbb{R}:I(\phi,\p)(s)<+\infty\}$ of $I(\phi,\p)$ is the closed interval $[\alpha_{\min}, \alpha_{\max}]$ where
\begin{equation}
\label{eq:effdom}
\left[\inf_{\{u\} \in \Per} \frac{\ov\phi\{u\}}{\ov\p\{u\}},\sup_{\{u\} \in \Per} \frac{\ov\phi\{u\}}{\ov\p\{u\}}\right]\;;
\end{equation}
on the interior of this interval, $I(\phi,\p)$ is strictly convex and of class $C^1$.  
\end{theorem}
\noindent For a comprehensive introduction to the large deviation principle, we refer to the standard monograph~\cite{DZ} of Dembo and Zeitouni. Notice that, since $\Pr(\phi,\p)(0)-h(\p)=0$, the definition of Fenchel-Legendre transform gives at once that $I(\phi,\p)$ takes values in $\mathbb{R}_+$, which is consistent with the fact that it governs a large deviation principle.
\begin{remark}
The effective domain of the rate function $I(\phi,\p)$ is reduced to a singleton precisely when $\phi\sim\p$, as follows from~\eqref{eq:effdom} and the already established facts in the present section.
\end{remark} 
\begin{proof}
The statement is a consequence of the already established analytic properties of the pressure function, together with classical large deviation theory and convex analysis; we give the details. 

Define the limiting logarithmic moment-generating function of the family of random variables in~\eqref{eq:familyrv} as
\[\Lambda\colon \mathbb{R}\to \overline{\mathbb{R}}\;, \quad t\mapsto \limsup_{T\to\infty}\frac{1}{T}\log \mathbf{E}_T[e^{t\phi}]\;.\]
Unravelling the definition, we have
\begin{align*}
  \Lambda(t)&=\limsup_{T\to\infty}\frac{1}{T}\log\biggl(\frac{1}{|\W_{\p}(T)|}\sum_{w\in \W_{\p}(T)}e^{t\phi(w)}\biggr)\\
  &=-h(\p)+\limsup_{T\to\infty}\frac{1}{T}\log\sum_{w\in \W_{\p}(T)}e^{t\phi(w)}=-h(\p)+\Pr(\phi,\p)(t)\;,
\end{align*}
using~\eqref{eq:hp}, Lemma~\ref{lem.expgrow} and the fact that, for all sufficiently large $R>0$,
\[
\lim_{n\to\infty}\frac{1}{n}\log \sum_{w\in \W,\;|\p(w)-n|\leq R}e^{t\phi}=\lim_{T\to\infty}\frac{1}{T}\log \sum_{w\in \W_{\p}(T)}e^{t\phi(w)}\;,  
\]
which is a matter of elementary analytic considerations. We deduce, in particular, that $\Lambda$ takes finite values and is of class $C^1$ on $\mathbb{R}$, by the analogous property for the pressure function established in \textsection\ref{subsec:pressurefunction}. The Gärtner-Ellis theorem in large deviation theory, for which we refer to~\cite[Theorem 2.3.6]{DZ}, delivers then directly that the family in~\eqref{eq:familyrv} satisfies the large deviation principle with rate function given by the Fenchel-Legendre transform of $\Lambda$. 

The fact that $I(\phi,p)$ is strictly convex and of class $C^1$ on the interior of its effective domain follows from the analogous property for $\Lambda$ (hence, for $\Pr(\phi,\p)$) and classical convex analytic considerations, for which we refer to Rockafellar's textbook treatment in~\cite[\textsection 26]{Rockafellar}. It is also a well known convex-analytic fact that the effective domain of the Fenchel-Legendre transform $I(\phi,\p)$ is the closed interval 
\begin{align*}
 [\inf_{\mathbb{R}}\Lambda',\sup_{\mathbb{R}}\Lambda']=[\inf_{\mathbb{R}}\Pr(\phi,\p)',\sup_{\mathbb{R}}\Pr(\phi,\p)']&=[\lim_{t\to-\infty}\Pr(\phi,\p)(t)/t,\lim_{t\to+\infty}\Pr(\phi,\p)(t)/t]\\
 &=[\Dil(\p,\phi)^{-1},\Dil(\phi,\p)]\;,
 \end{align*}
the last equality being given by Lemma~\ref{lem.asympgrad}. This concludes the proof.
\end{proof}

\begin{remark}
The same large deviation principle holds when
\begin{enumerate}
\item we count over periodic orbits and use the homogenizations of $\p$ and $\phi$,
\item and when we count over the sets $\{w\in \W:|\p(w)-n|\leq R\}$, as $n\to\infty$, for all sufficiently large $R>0$.
\end{enumerate}
\end{remark}

\begin{proof}[Proof of Theorem \ref{thm.rigidity2}]
The equivalence statements follow easily from the convexity of $\Pr(\phi,\p)$ and the asymptotic average exists the above large deviation principle.
\end{proof}


\section{The space of quasimorphisms}\label{sec:spaceofquasimorphisms}

In this section we study the metric properties of $\Qm^{++}$ equipped with the symmetrized Thurston distance.

\subsection{Basic properties}
We begin by verifying that $\Qq$ is indeed a metric on $\Qm^{++}$.
\begin{lemma}\label{lem.hinequ}
The function $\Qq : \Qm^{++} \times \Qm^{++} \to \R_{\ge 0}$ is a metric on $\Qm^{++}$.
\end{lemma}

\begin{proof}
It is easy to check that $\Qq$ is positive and satisfies the triangle inequality. Lastly, by Lemma \ref{lem.comp}, $\Qq([\p],[\phi]) =0$
if and only if $\Dil(\p,\phi)\Dil(\phi,\p) = 1$ which occurs if and only if $\Pr(\p,\phi)$ is a straight line by Lemma \ref{lem.asympgrad} (since $\Pr(\p,\phi)$ is convex). Corollary \ref{cor.line} shows that this happens if and only if $\p \sim \phi$, i.e. $[\p] = [\phi]$.
\end{proof}

We now verify that $(\Qm^{++},\Qq)$ is contractible.

\begin{theorem}
The space $(\Qm^{++},\Qq)$ is contractible. More precisely, given $[\psi_0]\in\Qm^{++}$ there is a
continuous homotopy
\[
\Hc:\Qm^{++}\times[0,1]\longrightarrow\Qm^{++}
\]
such that $\Hc([\varphi],0)=[\varphi]$ and $\Hc([\varphi],1)=[\psi_0]$ for every $[\varphi]\in\Qm^{++}$.
\end{theorem}

\begin{proof}
Fix $\psi_0\in\mathcal \qm^{++}$ with $h(\p_0)=1$. For $[\varphi]\in\Qm^{++}$ with $h(\varphi) =1$ define
\[
\Hc([\varphi],t) = (1-t)\varphi+t\psi_0 \ \  \ 0\leq t\leq1.
\]
Clearly $\Hc([\varphi],t) \in \Qm^{++}$ for all $t \in [0,1]$.
Also,
\[
\Hc([\varphi],0)=[\varphi]\ \ \text{ and }  \ \ \Hc([\varphi],1)=[\psi_0].
\]
Therefore, we just need to verify that $\Hc$ is continuous.

We first consider continuity in the first variable. Let $[\varphi],[\psi]\in\mathcal{QM}^{++}$ with $h(\varphi) = h(\psi)=1$.
Throughout the rest of the proof, given $\p_1,\p_2 \in \Qm^{++}$ we write $\ov\p_1 \le \ov\p_2$  if $\ov\p_1\{w\} \le \ov\p_2\{w\}$  for all $\{w\} \in \Per$. 

Set $A=\Dil(\varphi,\psi)$ and note that, by Lemma \ref{lem.bounds}, $A\geq1$. We then see that
\[
\ov{\varphi_t} = (1-t)\ov\varphi+t\ov{\psi_0}  \leq (1-t)A\ov\psi+t\ov{\psi_0} \leq A(1-t)\ov\psi + At\ov{\psi_0} = A\ov{\psi_t}.
\]
Hence $\Dil(\varphi_t,\psi_t) \le \Dil(\varphi,\psi)$
 and interchanging $\varphi$ and $\psi$ gives
$\Dil(\psi_t,\varphi_t) \leq \Dil(\psi,\varphi).$
Therefore
\[
\begin{aligned}
\Qq([\varphi_t],[\psi_t]) &= \log\Dil(\varphi_t,\psi_t) + \log\Dil(\psi_t,\varphi_t)\\
&\leq \log\Dil(\varphi,\psi) + \log\Dil(\psi,\varphi) = \Qq([\varphi],[\psi]).
\end{aligned}
\]
Thus, for every fixed $t\in[0,1]$,
\begin{equation}\label{eq.lip}
\Qq \left(\Hc([\varphi],t),  \Hc([\psi],t)\right) \leq \Qq([\varphi],[\psi]).
\end{equation}
In particular, the homotopy is $1$-Lipschitz in the first variable.

For the continuity in time we begin by setting $ M= e^{\Qq([\varphi],[\p_0])}.$
For $s\in[0,1]$, we have that 
\[
\ov{\varphi_s} =(1-s)\ov\varphi+s\ov{\psi_0} \geq (1-s)\ov\varphi+\frac{s}{M}\ov\varphi\\ = \left(1-s\left(1-\frac{1}{M}\right)\right)\ov\varphi \geq \frac1M\ov\varphi.
\]
Hence $ \ov\varphi\leq M\ov{\varphi_s}$ and similarly $ \ov{\psi_0}\leq M\ov{\varphi_s}.$

Now suppose $0\leq s<t\leq1$. Since
\[ 
\ov{\varphi_t} = \ov{\varphi_s} + (t-s) \left(\ov{\psi_0}-\ov\varphi \right) \ \ \text{ and } \ \ 
\ov{\psi_0}-\ov\varphi \leq \ov{\psi_0} \leq M\ov{\varphi_s}
\]
we obtain that
\[
\ov{\varphi_t} \leq \left(1+M(t-s)\right)\ov{\varphi_s}.
\]
It follows that $ \Dil(\varphi_t,\varphi_s) \leq 1+M(t-s).$
Running the same argument with $s$ and $t$ swapped gives 
$\Dil(\varphi_s,\varphi_t) \leq 1+M(s-t).$
Therefore
\begin{equation}\label{eq.time}
Q([\varphi_t],[\varphi_s]) \leq 2\log\left(1+M|t-s| \right) = 2 \log\left(1+e^{\Qq([\varphi],[\p_0])}|t-s| \right),
\end{equation}
which tends to zero as $t\to s$. 

Lastly, for joint continuity we simply note that, by the triangle
inequality,
\[
\Qq(\Hc([\varphi],t), \Hc([\psi],s)) \leq \Qq(\Hc([\varphi],t), \Hc([\psi],t)) + \Qq(\Hc([\psi],t),\Hc([\psi],s)).
\]
Then, using $(\ref{eq.lip})$ and $(\ref{eq.time})$ above we have that
\[
\Qq(\Hc([\varphi],t), \Hc([\psi],s)) \leq \Qq([\varphi],[\psi]) + 2\log(1+e^{\Qq([\varphi],[\p_0])}|t-s|).
\]
This proves joint continuity of $\Hc$ since as $t \to s$ and $[\varphi] \to [\psi]$ we see that $\Qq(\Hc([\varphi],t), \Hc([\psi],s)$ tends to $0$.
\end{proof}

\subsection{Word quasimorphisms}\label{sec.wq}
To prove that the space of quasimorphisms is separable, we will study a nice class of quasimorphisms which we call word quasimorphsms. Their definition is inspired by word metrics on groups.
\begin{definition}
A collection of words $E \subset \W$ is called generating if, given any $w \in \W$, we can find $e_1, \ldots, e_k \in E$ such that $w = e_1 \cdots \, e_k$.
\end{definition}
\begin{definition}
Given generating $E \subset \W$, define $\phi_E \in \qm^{++}$ as follows:
\[
\phi_E(w) = \inf\{ k : w = e_1 \cdots \, e_k \text{ with } e_1, \ldots, e_k \in E \}.
\]
We call $\phi_E$ constructed in this way a word quasimorphism.
 \end{definition}
It is simple to check that this is indeed a member of $\qm^{++}$.

\begin{proposition}
The space $(\Qm^{++}, \Qq)$ is unbounded.
\end{proposition}

\begin{proof}
It suffices to construct a sequence of quasimorphisms $\p_n \in \qm^{++}$ with $h(\p_n) =1$ such that there exists $\{w\} \in \Per$ with
\[
\ov{\p}_n\{w\} \to 0
\]
as $n\to\infty$. To that end, fix $\{\widetilde{w}\} \in \Per$ and define, for each $n\ge 1$, $E_n$  to be the set containing $\{ w \in \W_1 \}$ and the word $\widetilde{w}^n$. Then set $\p_n = h(\p_{E_n})\p_{E_n}$ and note that $h(\p_n) = 1$, $h(\p_{E_n}) \le |\{w \in \W_1 \}| + 1$  and so
\[
\ov\p_n\{\widetilde{w}\} = \lim_{k\to\infty} \frac{\p_n(\widetilde{w}^k)}{k} =  \lim_{k\to\infty}  \frac{\p_n(\widetilde{w}^{nk})}{nk}  \le
\frac{(|\{w \in \W_1 \}| + 1)k}{nk} \to 0
\]
as $n\to\infty$.
\end{proof}

We now want to show that the collection of equivalence class containing word quasimorphisms is dense in $(\Qm^{++}, \Qq)$. For this we need a couple of preliminary results.

\begin{lemma}\label{lem.k}
Take $\p \in \qm^{++}$. Then there exists $\kappa \ge 1$ such that if $w = w_1 w_2$ or $w = w_2w_1$ for words $w_1, w_2 \in \W$ with $w_2 \in \W_\kappa $. Then $\p(w_1) < \p(w)$.
\end{lemma}
\begin{proof}
This follows from the quasimorphism property and the assumption that there exists $\tau, C >0$ such that $\p(w) \ge \tau n - C$ for $w \in \W_n$ for any $n\ge 1$.
\end{proof}
We then immediately deduce the following.
\begin{corollary}\label{lem.qmint}
Let $\p \in \qm^{++}$ and $\kappa$ be as in Lemma \ref{lem.k}. Define $C = \max\{ \p(w) : w \in \W_\kappa \} + D_\p$. Then if $\p(w) < n -C$ and $ww' \in \W$ with $w' \in \W_\kappa $ then
\[
\p(w) < \p(ww') \le n.
\]
\end{corollary}

\begin{proposition}
The collection of equivalence classes in $\Qm^{++}$ than contain a word quasimorphism are dense in $(\Qm^{++}, \Qq)$.
\end{proposition}

\begin{proof}
Take $\p \in \Qm^+$ and define, for each $n\ge 1$,  $E_n = \{w \in \W: \p(w) \le n\}$. Set $\p_n = \p_{E_n}$ for each $n\ge1$. We will prove that $[\p_n] \to [\p]$ in the $\Qq$ metric.

To that end, suppose that $w \in \W$ has $\p_n(w) = k$ and that $w = w_1\cdots w_k$ where each $w_j \in E_n$. Then note that
\[
\p(w) = \p(w_1\cdots w_k) \le \sum_{j=1}^k \p(w_j) + kD_\p \le (n+D_\p)k = (n+D_\p)\p_n(w).
\]

We now want to obtain a similar lower bound for $\p$ in terms of $\p_n$. We will assume that $n$ is sufficiently large so that $\W_{\le 2\kappa} \subset E_n$ where $\kappa$ is as in Lemma \ref{lem.k}.  Let $C >0$ be the constant from Corollary \ref{lem.qmint}. 

We will describe a procedure that finds a nice expression for $w \in \W$ with $\p_n(w) = k \ge 3$ as a product of $k$ words in $E_n$. Specifically we will show it is possible to find an expression $w = \widetilde{w}_1\cdots \widetilde{w}_k$ where $\widetilde{w}_j \in E_n$ for $j=1,\ldots ,k$ and $n - C < \p_n(\widetilde{w}_j) \le n$ for each $j =1,\ldots, k-2$, i.e. the first $k-2$ words in the  $E_n$ expression for $w$ have length approximately $n$.

We will describe the procedure first for a word $w \in \W$ with $\p_n(w) =3$. Suppose $\p_n(w) = 3$ and that $w = w_1 w_2 w_3$ where $w_1, w_2, w_3 \in E_n$. Now, if $n -C < \p_n(w_1)\le n$ we are done. Suppose this is not the case. Now, if $w_2 \in \W_{\ge \kappa}$ then write $w_2 = w_2^1 w_2^2$ where $w_2^1 \in \W_\kappa$. By Lemma \ref{lem.k} $\p_n(w_2^2) < \p_n(w_2)$ and $\p_n(w_1) < \p_n(w_1w_2^1) \le n$ by Corollary \ref{lem.qmint}. In particular $w = (w_1w_2^1) w_2^2 w_3$ is a new length $3$ expression for $w$ with words in $E_n$ and $\p_n(w_1) < \p_n(w_1w_2^1) \le n$. If $w_2$ does not belong to $\W_{\ge \kappa}$ then $w_3$ must belong to $\W_{\ge \kappa}$: indeed, if it did not, then $w_2w_3 \in \W_{\le 2 \kappa}$ but by assumption on $n$ this would mean $w_2w_3 \in E_n$ which contradicts $\p_n(w) =3$. Hence $w_3= w_3^1w_3^2$ and rewriting $w = w_1 (w_2w_3^1) w_3^2$ we are back to the case where the second word in our expression for $w$ has length at least $\kappa$. We can the extend $w_1$ as before to increase its $\p_n$ length. We have shown that, as long as $w = w_1'w_2'w_3'$ where each of the $w_j' \in E_n$ and $\p_n(w_1') < n -C$, we can find a new $3$ product expression for $w$ where the first word has strictly larger $\p_n$ length.
Hence  continuing this process inductively must eventually result with an expression $w = \widetilde{w}_1 \widetilde{w}_2 \widetilde{w}_3$ where $\widetilde{w}_1, \widetilde{w}_2, \widetilde{w}_3 \in E_n$ and $n-C <\p_n(\widetilde{w}_1)$.

The above procedure shows how to express a word with $\p_n$ length $3$ in the desired form. When $\p_n(w) \ge 4$ we can apply the same process inductively starting with the first word in a minimal expression for $w$. That is, if $\p_n(w) = k \ge 4$ and $w = w_1 \cdots w_k$ with  the $w_j \in E_n$, then we can apply the above argument to write $w = w_1' \cdots \, w_k'$ where $w_j' \in E_n$ and $n - C \le \p_n(w_1')$. We then apply the procedure to $w'_2 \dots w_k'$ and so on, until we have obtained an expression $w = \widetilde{w}_1 \cdots \widetilde{w}_k$ where $\widetilde{w}_j \in E_n$ for $j=1,\ldots ,k$ and $n - C < \p_n(\widetilde{w}_j) \le n$ for each $j =1,\ldots, k-2$. \\

Now, given $w \in \W$ with $\p_n(w) = k$ we can write $w =w_1 \cdots w_k$ where $w_j \in E_n$ for $j=1,\ldots ,k$ and $n - C < \p_n(w_j) \le n$ for each $j =1,\ldots, k-2$.
We then have that 
\begin{align*}
\p(w) = \p(w_1\cdots w_k) &\ge \sum_{j=1}^{k-2}\p(w_j) + \p(w_{k-1}) + \p(w_k) - kD_\p\\
&\ge (n-C)(k-2) - kD_\p \\
&= k(n - D_\p - C) - (n-C) = (n - D_\p -C) \p_n(w) - 2(n-C).
\end{align*}
Combining the above inequalities implies that $[\p_n] \to [\p]$ in the $\D$ metric as $n\to\infty$.
\end{proof}

\subsection{Completeness with bounded defect}

We now turn our attention to understanding the collection of quasimorphisms with bounded defect.

\begin{lemma}\label{lem.bound}
Suppose that $\p_E$ is a word quasimorphism and $B \subset  \Qm^{++}$ a bounded subset (with respect to $\Qq$). Then for any $[\p] \in B$ with $h(\p) =1$ and $D_\p \le D$, and $w \in \W$ we have that
\[
C^{-1} \p_E(w) - C \le \p(w) \le C \, \p_E(w) + C
\]
where $C > 1$ is a constant depending only on $B$ and $E$.
\end{lemma}
\begin{proof}
By Lemma \ref{lem.comp} and Remark \ref{rem.indep} we just need to show that $\Dil(\p,\phi)$ and $\Dil(\phi,\p)$ are uniformly bounded above in terms of $E$ and $B$. To that end note that since $h(\p) =1$ we have that
\[
\Dil(\p_E,\p) \ge \frac{h(\p)}{h(\p_E)} = \frac{1}{h(\p_E)}.
\]
Hence
\[
e^{\Qq([\p], [\p_E])}  = \Dil(\p_E, \p) \Dil(\p,\p_E) \ge \frac{\Dil(\p, \p_E)}{h(\p_E)}
\]
and  $\Dil(\p,\p_E)$ is bounded above by a quantity depending only on $B$ and $E$.

Similarly
\[
\Dil(\p_E,\p) \ge \frac{h(\p)}{h(\p_E)} = \frac{1}{h(\p_E)}
\]
and so
\[
\Dil(\p_E,\p)^{-1} \le h(\p_E).
\]
This concludes the proof.
\end{proof}

Recall the following definition from the introduction.

\begin{definition}
Let $\Qm_D^{++} \subset \Qm^{++}$ be the collection of classes that contain a quasimorphism $\p$ with $h(\p) = 1$ and $D_\p \le D$.
\end{definition}

\begin{theorem}
We have that $(\Qm^{++}_D,\Qq)$ is a complete subset of $(\Qm^{++},\Qq)$.
\end{theorem}
\begin{proof}
Take a Cauchy sequence $[\p_n]$ in $\Qm^{++}_D$ where we assume that $h(\p_n) =1$ and $D_{\p_n} \le D$ for all $n \ge 1$. 

By Lemma \ref{lem.bound} we can extract a subsequence $n_k$ such that $\p_{n_k}(w)$ converges as $k\to\infty$ for each $w\in\W$. Write $\p_\infty(w)$ for the limits. We first want to show that $\p_\infty \in \Qm^+$.
This is easy to see however as
\[
|\p_\infty(w_1w_2) - \p_\infty(w_1) - \p_\infty(w_2)|  = \lim_{k\to\infty} | \p_{n_k}(w_1w_2) - \p_{n_k}(w_1) - \p_{n_k}(w_2)| \le D
\]
and
\[
\p_\infty(w) = \lim_{k\to\infty} \p_{n_k}(w) \ge C^{-1} \p_E(w) - C 
\]
for all $w\in\W$.

We now just need to show that $[\p_n]$ converges to $[\p_\infty]$ and that $[\p_\infty] \in \Qm^{++}_D$. To that end, note that for $w \in \Fix$
\[
\ov{\p}_\infty\{w\} = \lim_{n\to\infty} \frac{\lim_{k\to\infty} \p_{n_k}(w^n) + D}{n} \ge \limsup_{k\to\infty} \ov{\p}_{n_k}\{w\}
\]
where we have used that
\[
\ov{\p}\{w\} = \inf_{n\ge 1} \frac{\p(w^n) + D}{n}
\]
for $\p \in \Qm^+$ and any $w \in \Fix$. Similarly  $\ov{\p}_\infty\{w\} \le \liminf_{k \to\infty} \ov{\p}_{n_k}\{w\}$ and so we deduce that $\ov{\p}_{n_k}$ converges to $\ov{\p}_\infty$ pointwise as $k\to\infty$.

Now note that this convergence implies that
\begin{equation}\label{eq.conv}
\left( \limsup_{k\to\infty} \Dil(\p_n, \p_{n_k})\right)^{-1} \ov{\p}_n\{w\} \le \ov{\p}_\infty\{w\} \le \limsup_{k\to\infty} \Dil(\p_{n_k}, \p_n) \, \ov{\p}_n\{w\}
\end{equation}
for each $n\ge 1$. We also have that
\begin{align*}
1\le \min\{ \Dil(\p_{n_k}, \p_n), \Dil(\p_n,\p_{n_k})\}  &\le \max\{ \Dil(\p_{n_k}, \p_n), \Dil(\p_n,\p_{n_k})\}\\
 &\le \Dil(\p_{n_k},\p_n) \Dil(\p_n, \p_{n_k}) = e^{\Qq([\p_{n_k}], [\p_n])}
\end{align*}
for all $n,k \ge 1$ since $h(\p_n) = 1$ for all $n\ge 1$. It follows, since $\p_n$ is Cauchy, that 
\[
\lim_{n\to\infty} \limsup_{k\to\infty} \Dil(\p_{n_k},\p_n) \Dil(\p_n, \p_{n_k}) = 1.
\]
Along with equation $(\ref{eq.conv})$ this shows that $[\p_n] \to [\p_\infty]$ as $n\to\infty$.

Finally we need to show that $[\p_\infty]$ belongs to $\Qm^{++}_D$. To do so it suffices to show that $h(\p_\infty) =1$ since we saw that $D_{\p_\infty} \le D$ above.
To that end we write $\tau_n = \limsup_{k\to\infty} \Dil(\p_{n_k},\p_n) \Dil(\p_n, \p_{n_k})$ and note that, by the above discussion, $\tau_n \to 1$ as $n\to\infty$ and
\[
\tau_n^{-1} \ov{\p}_n\{w\} \le \ov{\p}_\infty\{w\} \le \tau_n \ov{\p}_n\{w\}
\]
for all $n \ge1$. This equation implies that
\[
\tau_n = h(\tau_n^{-1} \p_n(w)) \ge h( \ov{\p}_\infty\{w\} ) \ge h(\tau_n \p_n(w)) = \tau_n^{-1}
\]
since $h(\p_n) = 1$ for all $n \ge 1$. Letting $n \to\infty$ gives $h(\p_\infty) = 1$ as required. 
\end{proof}

\subsection{Density of locally constant functions}
In this section we prove Theorem \ref{thm.lc}.
\begin{definition}
Given a continuous function $F: \Sigma \to \R$ we say that $F$ is twin to $\p \in \qm^{++}$ if for any $\{w\} \in \Per_n$ and $x_w = (w,w,\ldots) \in  \Sigma$
\[
S_nF(x_w) = \ov\p\{w\}.
\]
We say that $F$ is twin to $[\p'] \in \Qm^{++}$ if $F$ is a twin to some $\p \in [\p']$.
\end{definition}

\begin{proposition}\label{prop.rep}
Each quasimorphism $\p \in \qm^{++}$ is twin to a continuous function.
\end{proposition}

\begin{proof}
Define the sequence of functions  $(f_n)_{n\ge 1}$, $f_n :\Sigma \to \R$, by $f_n(x) = \p(x_0,\ldots,x_{n-1})$. Then $(f_n)_{n\ge 1}$ is an almost additive function in the sense of \cite{Cuneo}. In particular by \cite[Theorem 1.2]{Cuneo} there exists a continuous function $f: \Sigma \to \R$ such that
\begin{equation}\label{eq.cont}
\frac{1}{n} \| f_n - S_nf\|_\infty \to 0
\end{equation}
as $n\to\infty$. 

We will verify that $f$ is twin to $\p$. To that end, take $\{w\} \in \Per_k$ and note that
\[
\frac{\ov \p\{w\}}{k} = \lim_{n\to\infty} \frac{\p(w^n)}{kn} = \lim_{n\to\infty} \frac{f_{nk}(x_w)}{nk} =\lim_{n\to\infty} \frac{S_{nk}f(x_w)}{nk} = \frac{S_k f(x_w)}{k}
\]
where the first and second and fourth equalities follow from the definitions and the third equality follows from  $(\ref{eq.cont})$ above. Hence $\ov \p\{w\} = S_k f(x_w)$ for all $\{w\} \in \Per$ and we are done.
\end{proof}

We also need the following.

\begin{lemma}\label{lem.pos}
Suppose that $\p \in \qm$ then $\p$ belongs to $\qm^{++}$ if and only if given any $\phi \in \qm^{++}$ there is $C >0$ such that
\[
\ov{\p}\{w\} \ge C \, \ov{\phi}\{w\}
\]
for all $\{w\} \in \Per$.
\end{lemma}
\begin{proof}
The forward direction follows from Lemma \ref{lem.comp}. We suppose that we have $\phi \in \qm$ and $C >0$ such that $\ov{\p}\{w\} \ge C \, \ov{\phi}\{w\}$ for all $w \in \Per$. Now, given $w \in \W$ there exists $w' \in \W_M$ where $M \ge 1$ is the mixing constant such that $ww'$ is periodic. Then
\begin{align*}
\p(w) \ge \p(ww')  - \p(w') - D_\p &\ge \ov{\p}\{ww'\} - D_\p - \p(w') - D_\p \\
&\ge C \ov{\phi}\{ww'\} - \p(w') - 2D_\p\\
&\ge C\phi(ww') - D_\phi - \p(w') - 2D_\p\\
&\ge C\phi(w)- C\phi(w') - \p(w') -(1 + C) D_\phi - 2D_\p.
\end{align*}
This implies the lemma as $M$ is independent of $w$.
\end{proof}

We are now ready to prove Theorem \ref{thm.lc}.

\begin{proof}[Proof of Theorem \ref{thm.lc}]
Take $[\phi] \in \Qm^{++}$. It suffices to show that for any $\epsilon >0$ there is $[\p] \in \Qm^{++}$ that is twin to a rational valued, locally constant function and  $\Qq([\phi],[\p]) < \epsilon$.

Let $F \colon \Sigma \to \R$ be a continuous function that represents $\phi$ coming from Proposition \ref{prop.rep}. Note that by Lemma \ref{lem.pos} there exists $\tau > 0$ such that $\ov\phi\{w\} \ge \tau |w|$ for any periodic $w$.

 Then, since rational valued, locally constant functions are dense in the supremum norm in $C(\Sigma,\R)$ \cite{PP}, there exists a locally constant function $r : \Sigma \to \R$ such that $\|r - F\|_\infty < \delta$ where $\delta >0$ is to be chosen later. Define $\p_r \in \qm$ to be the quasimorphism twin to $r$. 

We then have that
\[
|\ov\phi\{w\} - \ov\p_t\{w\}| = |S_{|w|}r(x_w) - S_{|w|}F(x_w)| < \delta |w|
\]
for all $\{w\} \in \Per$. Assuming that $\delta < \tau/2 $ this implies that $\ov\p_r\{w\} \ge \tau|w|/2$ for all periodic $w$. Thus by Lemma \ref{lem.pos} $\p_r \in \qm^{++}$ when $\delta$ is sufficiently small.

We now note that
\[
\left| 1 - \frac{\ov\p_r\{w\}}{\ov\phi\{w\}} \right| = \frac{|\ov\phi\{w\} - \ov\p_t\{w\}|}{\ov\phi\{w\}} \le \frac{\delta |w|}{\tau |w|} = \frac{\delta}{\tau}
\]
or in other words
\[
(1 - \eta)\ov\phi\{w\} \le \ov\p_r\{w\} \le (1 + \eta)\ov\phi\{w\} 
\]
for each $\{w\} \in \Per$ where $\eta = \delta/\tau$.
This implies that
\[
\Qq([\p_r],[\phi]) \le \log \left( \frac{1 + \eta}{1-\eta}\right).
\]
Therefore, for $\delta >0$ sufficiently small we see that $\p_r \in \qm^{++}$ and $\Qq([\p_r],[\phi]) \le \epsilon$ as required.
\end{proof}

\begin{remark}
Note that separability of $\Qm^{++}$ follows from Theorem \ref{thm.lc} since the set of rational valued, local constant functions on a subshift of finite type are countable.
\end{remark}

\subsection{Pressure geodesics}
In this section we show that the pressure curves form geodesics in $(\Qm^{++},\Qq)$.
\begin{theorem}
Fix $[\p], [\phi] \in \Qm^{++}$ and set $\p_t = \Pr(\p,\phi)(t) \phi - t\p$  for each $t \in \R$. Then the $\p_t \in \qm^{++}$ for all $t \in \R$ and
\[
t \mapsto [ \Pr(\p,\phi)(t) \phi - t\p]  = [\p_t]\in \Qm^{++}
\]
is a bi-infinite geodesic for the $\D$ metric: that is, for any $r, s,t \in \R$ with $s < r<t$ we have that
\[
\Qq([\p_r],[\p_t]) = \Qq([\p_r],[\p_s]) + \Qq([\p_s], [\p_t])
\]
and $\lim_{t\to\pm \infty} \Qq([\p_0],[\p_t]) = \infty$.
\end{theorem}

We will use the following lemma multiple times throughout this section.

\begin{lemma}\label{lem.post}
For any $t \ge 0$ we have that
\[
\Pr(\p,\phi)(t) - t \Dil(\p,\phi) > 0.
\]
Likewise for any any $t \le 0$ we have that
\[
\Pr(\p,\phi)(t) - t\Dil(\p,\phi)^{-1} > 0.
\]
\end{lemma}
\begin{proof}
This is true for $t = 0$ since $\Pr(\p,\phi)(0) = h(\p) > 0$.
Suppose for contradiction that there exists $t > 0$ such that $\Pr(\p,\phi)(t) - t \Dil(\p,\phi) = 0$. This implies that $I(\Dil(\p,\phi)) < 0$ where $I$ is the Fenchel-Legendre transform of $\Pr(\p,\phi)(t)$. In particular, since $I$ is $C^1$, there exists $t' >0$ such that $I(t') < 0$. By the large deviation principle this implies that there exists an open set $U \subset \R$ such that 
\[
\liminf_{n\to\infty} \frac{1}{n} \log \left( \frac{1}{\#\{w \in \Fix : \ov{\p}(w) < T \}} \#\left\{ w \in \Fix: \ov{\p}(w) < T ,  \frac{\ov{\phi}(w)}{T} \in U \right\}\right) > 0 
\]
which is absurd. This shows that no such $t > 0$ can exist. A similar proof shows the same for $t < 0$.
\end{proof}

The following lemma shows that the curve $\Pr(\p,\phi)(t) \phi - t\p $ stays inside $\qm^{++}$.

\begin{lemma}\label{lem.++}
Fix $\p,\phi \in \qm^{++}$. For each $t \in \R$ define $\p_t =  \Pr(\p,\phi)(t) \phi - t\p  \in \qm$. Then $\p_t$ belongs to $\qm^{++}$ for each $t \in \R$.
\end{lemma}

\begin{proof}
Suppose that $t >0$. Then we have that for any $\{w\} \in \Per$,
\begin{align*}
\ov{\p}_t \{w\} &=   \Pr(\p,\phi)(t) \ov{\phi}\{w\}- t\ov{\p}\{w\}  \\
&\ge  \Pr(\p,\phi)(t) \ov{\phi}\{w\} - t \Dil(\p,\phi)\ov{\phi}\{w\} \\
&= \ov{\phi}\{w\}( \Pr(\p,\phi)(t) - t \Dil(\p,\phi)).
\end{align*}
When $t \le 0$ we similarly have that for any $w \in \Fix$,
\begin{align*}
\ov{\p}_t \{w\} &=   \Pr(\p,\phi)(t) \ov{\phi}\{w\} - t\ov{\p}\{w\}  \\
&\ge  \Pr(\p,\phi)(t) \ov{\phi}\{w\} - t \Dil(\p,\phi)^{-1}\ov{\phi}\{w\} \\
&= \ov{\phi}\{w\}( \Pr(\p,\phi)(t) - t \Dil(\p,\phi)^{-1}).
\end{align*}
We now apply Lemma \ref{lem.pos} to deduce the conclusion.
\end{proof}

Combined wit Lemma \ref{lem.comp} we can show a stronger, quantitative version of this result.
\begin{corollary}\label{cor.lamb}
Fix $\p,\phi \in \qm^{++}$ and for each $t \in \R$ define $\p_t =  \Pr(\p,\phi)(t) \phi - t\p  \in \qm$.
Then there exists a constant $\lambda > 0$ such that, for all $t \in \R$, $\p_t + (\lambda|t| + \lambda)$ is a strictly positive quasimorphism belonging to $\qm^{++}$.
\end{corollary}
\begin{proof}
Consider $t \ge 0$. Then,  in the proof of the previous lemma we saw that 
\[
\Dil(\phi, \p_t)^{-1} \ge \Pr(\p,\phi)(t) - t \Dil(\p,\phi).
\]
Let $\tau_\phi, C_\phi >0$ be the constants such that $\phi(w) \ge \tau_\phi |w| -C_\phi$ for all $w \in \W$.
Now, by Lemma \ref{lem.comp} there exists $C_t > 0$ such that
\begin{align*}
\p_t(w) &\ge \phi(w)(\Pr(\p,\phi)(t) - t \Dil(\p,\phi)) - C_t \\
&\ge (\tau_\phi |w| - C_\phi)(\Pr(\p,\phi)(t) - t \Dil(\p,\phi)) - C_t\\
& \ge \tau_\phi(\Pr(\p,\phi)(t) - t \Dil(\p,\phi)|w| - C_\phi(\Pr(\p,\phi)(t) - t \Dil(\p,\phi)) - C_t.
\end{align*}
Now, since $\tau_\phi(\Pr(\p,\phi)(t) - t \Dil(\p,\phi)) > 0$ by Lemma \ref{lem.post} it suffices to show that there exists $M >0$ such that
\begin{equation}\label{eq.Mbound}
|C_\phi(\Pr(\p,\phi)(t) - t \Dil(\p,\phi)) + C_t| \le M  |t| + M.
\end{equation}
for all $ t\in \R$.

First note that, by Remark \ref{rem.indep}, the constant $C_t$ can be obtained from Remark \ref{rem.indep} which can easily be checked to be controlled by $c' (|\Pr(\p,\phi)(t)| + |t|) + c'$ for some suitably large $c' >0$ independent of $t \in \R$. We leave this calculation to the reader. Hence the left-hand side of $(\ref{eq.Mbound})$ is bounded above by
\begin{equation}\label{eq.Lbound}
L(|\Pr(\p,\phi)(t)| + |t|) + L
\end{equation}
for some $L > 0$.

Now note that, if $t \ge 0$,
\[
0 < \Pr(\p,\phi)(t) \le \Dil(\p,\phi) t + \Pr(\phi,\p)
\]
by Remark \ref{rem.asympupper} and if $t <0$ 
\[
0 > \Pr(\p,\phi)(t) \le t\Dil(\p,\phi)^{-1}
\]
by Lemma \ref{lem.post}. Hence there is $M' > 0$ such that $|\Pr(\p,\phi)(t)| \le M' |t| + M'$ for all $t \in \R$.
Subbing this into $(\ref{eq.Lbound})$ implies that there is $\lambda > 0$ such that
\[
|C_\phi(\Pr(\p,\phi)(t) - t \Dil(\p,\phi)) + C_t| \le \lambda  |t| + \lambda
\]
as required.
\end{proof}

\subsubsection{Computations}

For $s,t \in \R$ let $D_{s,t} = \Dil(\p_s,\p_t)$ also, assume that we have fixed $\phi,\p \in \Qm^+$ and let $\Pr$ denote the pressure $\Pr(\p,\phi)$.

\begin{lemma}\label{lem.compu}
When $t > -h(\p)$ we have that
\begin{enumerate}
\item $D_{t, -h(\p)} = h(\p)^{-1}(\Pr(t) D_{\phi,\p} - t)$; and,
\item $D_{-h(\p),t} = h(\p)(\Pr(t)D_{\p,\phi}^{-1} - t)^{-1}$.
\end{enumerate}
When $t<0$ we have that 
\begin{enumerate}\setcounter{enumi}{2}
\item $D_{t, 0} = h(\phi)^{-1}( \Pr(t) - D_{\p,\phi}t)$; and,
\item $D_{0,t} = h(\phi)(\Pr(t) - D_{\phi,\p}^{-1} t)^{-1}$.
\end{enumerate}
\end{lemma}

In the following proof, given $\p,\phi \in \qm$ we write $\p \lesssim \phi$ if there exists $C>0$ independent of $w$ such that $\p(w) \le \phi(w) + C$ for all $w\in\W$.
\begin{proof}
(1) Note that when $t > -h(\p)$, $\Pr(t) > 0$ and so
\[
\p_t = \Pr(t)\phi - t\p \lesssim \Pr(t)\p D_{\phi,\p} - t \p = h(\p) \p (\Pr(t) D_{\phi,\p} - t) h(\p)^{-1}.
\]
Hence $D_{t,-h(\p)} \le h(\p)^{-1}(\Pr(t) D_{\phi,\p} - t).$
For the other inequality note that
\[
\Pr(t) \phi - t \p = \p_t \lesssim D_{t,-h(\p)} h(\p) \p
\]
rearranges to
\[
\Pr(t)\phi \lesssim(D_{t,-h(\p)}h(\p) + t) \p.
\]
This implies that $\Pr(t) D_{\phi,\p} \le D_{t,-h(\p)} h(\p) + t$
or equivalently, $D_{t,-h(\p)} \ge h(\p)^{-1} (\Pr(t) D_{\phi,\p} - t)$ as required.  Statement (2) is proved similarly.

The statements $(3)$ and $(4)$ can either be deduced similarly or by using the fact that the pressure curves $\Pr(\p,\phi)$ and $\Pr(\phi,\p)$ are related in the following way
\[
(t,\Pr(\phi,\p)(t)) = (-\Pr(\p,\phi)(t), -t) 
\]
i.e. $\Pr(\phi,\p)$ is obtained from reflecting $\Pr(\p,\phi)$ in the line $y = -x$.
\end{proof}

\begin{corollary}\label{cor.geo}
Suppose that $r,s,t \in \R$ with $r<s<t$ then
\[
D_{r,t} = D_{r,s} \, D_{s,t} \ \ \ \text{ and } \ \ \  D_{t,r} = D_{t,s} \,  D_{s,r}.
\]
\end{corollary}
\begin{proof}
We can easily verify these identities for $r = -h(\p)$ and $t=0$:
\begin{align*}
D_{-h(\p),s} D_{s,0} &= h(\p)(\Pr(s)D_{\p,\phi}^{-1} - s)^{-1} h(\phi)^{-1}(\Pr(s) - D_{\p,\phi}^{-1} s)\\
&= h(\phi)^{-1}h(\p) \left( \frac{\Pr(s) - D_{\p,\phi}s}{\Pr(s) D_{\p,\phi}^{-1} -s}\right)\\
&= h(\phi)^{-1}h(\p) D_{\p,\phi}  = D_{-h(\p),0}.
\end{align*}
and
\begin{align*}
D_{0,s} D_{s,-h(\p)} &= h(\phi)(\Pr(s) - D_{\phi,\p}^{-1} s)^{-1} h(\p)^{-1}(\Pr(s) D_{\phi,\p} - s)\\
&= h(\p)^{-1} h(\phi)  \left( \frac{\Pr(s)D_{\phi,\p} -s}{\Pr(s) - D_{\phi,\p}^{-1} s}\right)\\
&= h(\p)^{-1} h(\phi) D_{\phi,\p} = D_{0,-h(\p)}.
\end{align*}
For general $r,s,t \in \R$ we can reparameterise the curve $\Pr(\phi,\p)$.
\end{proof}
We are now ready to prove the main theorem of this section.

\begin{proof}[proof of Theorem \ref{thm.geo}]
Corollary \ref{cor.geo} and Lemma \ref{lem.++} immediately imply that $[\p_t]$ determines a geodesic in $\Qm^{++}$. To see that this is a bi-infinite geodesic, we calculate, using Lemma \ref{lem.compu},
\[
\lim_{t\to-\infty} e^{\Qq([\p_0],[\p_t])} = \lim_{t\to-\infty} \frac{(\Pr(t) - D_{\p,\phi}t)}{(\Pr(t) - D_{\phi,\p}^{-1}t)} = \lim_{t\to-\infty}  \frac{\left(\frac{\Pr(t)}{t} - D_{\p,\phi}\right)}{\left(\frac{\Pr(t)}{t} - D_{\phi,\p}^{-1}\right)} = \infty
\]
where we have used that $\frac{\Pr(t)}{t} \to D_{\phi,\p}^{-1}$ as $t\to-\infty$ by Lemma \ref{lem.asympgrad}. The limit $t\to\infty$ can be shown to be $\infty$ similarly.
\end{proof}

\subsection{The boundary}\label{sec.boundary}
\begin{lemma}
Let $\p_t = \Pr(\p,\phi)(t) \phi - t\p$ where $\p,\phi \in \qm^{++}$. Then, for any $w\in\W$, we have that
\[
\lim_{t\to\infty} \frac{\p_t(w)}{t}  = \Dil(\p,\phi)\phi(w) - \p(w) \ \ \text{ and } \ \ \lim_{t\to - \infty} \frac{\p_t(w)}{t}  = \frac{\Dil(\phi,\p) \p(w) - \phi(w)}{\Dil(\phi,\p)}.
\]
\end{lemma}
\begin{proof}
This follows from Lemma \ref{lem.asympgrad}.
\end{proof}
This lemma shows that, given $\p,\phi \in \qm^{++}$, there are natural endpoints of the corresponding Pressure geodesic $[\p_t]$ at $\pm \infty$. More precisely, if we let
\[
\p_\infty =  \Dil(\p,\phi)\phi(w) - \p(w) \ \ \text{ and } \p_{-\infty} = \Dil(\phi,\p) \p(w) - \phi(w)
\]
then we can think of $[\p_\infty]$ and $[\p_{-\infty}]$ as the endpoints $[\p_t]$.
\begin{lemma}
The quasimorphisms $\p_{\pm\infty}$ belong to $\qm^+\backslash \qm^{++}$.
\end{lemma}
\begin{proof}
It follows easily from Lemma \ref{lem.hombound} that the quasimorphisms belong to $\qm^{+}$. Now, suppose that $\p_\infty$ belongs to $\qm^{++}$. If this is the case, then for sufficiently small $\epsilon >0$ we have that
\[
\ov{\p_{\infty}}\{w\} \ge \epsilon \ov{\p}\{w\}
\]
for any $\{w\}\in\Per$ by Lemma \ref{lem.pos}. 
This rearranges to
\[
\frac{\Dil(\p,\phi)}{(1+\epsilon)} \ge \frac{\ov\p\{w\}}{\ov\phi\{w\}}
\]
which contradicts the definition of $\Dil(\p,\phi)$.
\end{proof}

\begin{proposition}
Every quasimorphism in $\qm^{+}\backslash \qm^{++}$ can be realised as the endpoint of a pressure geodesic for a pair of quasimorphisms in $\qm^{++}$. 
\end{proposition}

\begin{proof}
Fix $\p_\infty \in \qm^+$. We want to construct a pressure geodesic for which the endpoint is in the class $[\p_\infty]$. To that end, take $\phi \in \qm^{++}$ and define $\rho = \phi + \p_\infty$.

Since $\ov{\rho} \ge \ov{\phi}$ we have that $\Dil(\phi,\rho) \le 1$. Conversely, there exists, by Lemma \ref{lem.comp}, $C>0$ such that 
\[
(1 - \Dil(\phi,\rho))\rho(w) \le \rho(w) - \phi(w) + C \le \p_\infty(w) + C
\]
for all $w \in \W$. It follows from Lemma \ref{lem.comp}, since $\p_\infty \nin \qm^{++}$ that $\Dil(\phi,\rho) \ge 1$. 

We have show that $\Dil(\phi,\rho) =1$ and so
\[
\p_\infty = \Dil(\phi,\rho)\rho - \phi
\]
is the endpoint of the pressure geodesic corresponding to $\rho, \phi \in \qm^{++}$.
\end{proof}

We can then characterise when a pair of quasimorphisms can be realised as the pair of endpoints of a pressure geodesic.
\begin{definition}
We say that two quasimorphisms $\p_1,\p_2 \in \qm^+\backslash \qm^{++}$ are compatible if $\p_1 + \p_2 \in \qm^{++}$.
\end{definition}

\begin{proposition}
The pair $\p, \phi \in \qm^+\backslash \qm^{++}$ are compatible if and only if there exists a pressure geodesic $[\rho_t]$ (associated to some pair in $\qm^{++}$) such that $[\rho_\infty] = [\p]$ and $[\rho_{-\infty}] = [\phi]$.
\end{proposition}
\begin{proof}
We start with the reverse implication. Suppose that $[\rho_t]$ is a pressure geodesic for a pair of quasimorphisms $\rho_1, \rho_2$ both with exponential growth rate $1$. Then, up to scaling and uniformly bounded additive error we may assume that
\[
\rho_\infty = \Dil(\rho_1,\rho_2) \rho_2 - \rho_1 \ \ \text{ and } \ \ \rho_{-\infty} = \Dil(\rho_2,\rho_1)\rho_1 - \rho_2.
\]
We then have that 
\[
\rho_\infty  + \rho_{-\infty} = (\Dil(\rho_2,\rho_1)-1) \rho_1 + (\Dil(\rho_1,\rho_2) -1)\rho_2.
\]
Since $[\rho_1] \neq [\rho_2]$ we have that $\Dil(\rho_1,\rho_2) \Dil(\rho_2,\rho_1) \ge 1$. By Lemma \ref{lem.hinequ} we also see that both $\Dil(\rho_1,\rho_2) \ge 1$ and $\Dil(\rho_2,\rho_1) \ge 1$. Hence at least one of these is strictly bigger than $1$ and  $\rho_\infty  + \rho_{-\infty} + \qm^{++}$.\\

Conversely, suppose that $\p$ and $\phi$ are transverse. Then define
\[
\rho_1 = 2 \p + \phi \ \ \text{ and } \ \ \rho_2 = \p + 2\phi
\]
and note that these belong to $\qm^{++}$.

We note that
\[
\rho_1 = 2\rho_2 - 3\phi
\]
and so $\Dil(\rho_1,\rho_2) \le 2$. In fact, we must have $\Dil(\rho_1,\rho_2) = 2$ since if $2 - \Dil(\rho_1,\rho_2)  = a > 0$ then
\[
3\phi = 2\rho_2 - \rho_1 = (\Dil(\rho_1,\rho_2)  + a)\rho_2 - \rho_1 \ge a\rho_2
\]
contradicting $\phi \in \qm^+ \backslash \qm^{++}$. Likewise $\Dil(\rho_2,\rho_1)  = 2$.

To conclude the proof we now simply note that
\[
\Dil(\rho_1,\rho_2) \rho_2 - \rho_1 = 2\rho_2 - \rho_1 = 3\phi \ \ \text{ and } \ \ \Dil(\rho_2,\rho_1) \rho_1 - \rho_2 = 3\p.
\]
\end{proof}

\begin{example}
Suppose that $\Sigma$ is a full shift over the symbols $x_1,\ldots,x_k$ and suppose that $n < k$ and $y_1,\ldots, y_n \in \{x_1,...,x_k\}$ are $n$ fixed disjoint symbols. Let $y_{n+1},\ldots, y_k$ denote the remaining symbols so that $\{ y_1,\ldots, y_k\} = \{x_1,...,x_k\}$. Let, for each $j=1,\ldots,k$, $\p_j$ denote the counting quasimorphism that counts the number occurrences of $y_j$ in the input word. Then set
\[
\rho_1 = \sum_{j=1}^n \p_j \ \ \text{ and } \ \ \rho_2 = \sum_{j=n+1}^k \p_j.
\]
Note that both of these quasimorphisms belong to $\qm^+$. Furthermore, they are compatible since $\rho_1 + \rho_2$ outputs the number of symbols in the input word.
\end{example}

\subsection{Non-completeness}
In this section we show that the space $(\Qm^{++}, \Qq)$ is not complete. We use the notation $\MC(\Qm^{++})$ to denote the metric completion of $(\Qm^{++}, \Qq)$.

\begin{lemma}\label{lem.unifconv}
Suppose that $F_m$ is a sequence of H\"older functions twin to $[\p_m] \in \Qm^{++}$ and that $F_m$ converges uniformly to $F: \Sigma \to \R_{>0}$. Then $[\p_m]$ converges to some $[\p] \in \MC(\Qm^{++})$ where $F$ is dual to $\p$.
\end{lemma}

\begin{proof}
Given $w = (x_0, \ldots, x_{n-1}) \in\Fix_n$ let $x_w = (x_0,\ldots,x_{n-1}, x_0, \ldots) \in \Sigma$ be the infinite sequence obtained from repeating $w$. Define $\mu_{w}$ denote the uniform counting measure
\[
\frac{1}{k} \sum_{j=0}^{k-1} \delta_{\sigma^j(x)}
\]
over the orbit of $x$.
Note that, if $F_m \to F$ uniformly, then, since $F$ is strictly positive, there exists $\epsilon > 0$ such that $F_m > \epsilon$ for all $m$ sufficiently large. Therefore
\begin{align*}
\left| \log F_m^n(x_w)- \log F_k^n(x_w)\right|  &= \left| \log \left(\int F_m \ d\mu_w\right) - \log \left(\int F_k \ d\mu_w\right) \right|\\
&\le \epsilon^{-1} \left| \int F_m \ d\mu_w- \int F_k \ d\mu_w \right|  \le \epsilon^{-1} \| F_m - F_k\|_\infty
\end{align*}
where we have used the mean value theorem in the first inequality.
It follows that $[\p_m]$ is Cauchy and so converges to a limit in $\MC(\Qm^{++})$. It is easy to check that $F$ is twin to the limit.
\end{proof}

\begin{lemma} \label{lem.uniquegibb}
Suppose that $F$ represents $[\p] \in \Qm^{++}$. Then, $F$ has a unique Gibb's state.
\end{lemma}
\begin{proof}
The Gibbs state for $F$ is $\mu_\p$ and so is unique by ergodicity.
\end{proof}

\begin{theorem}
The space $(\Qm^{++}, \Qq)$ is not complete.
\end{theorem}

\begin{proof}
Take a continuous function $F : \Sigma \to \R$ that does not have a unique Gibb's state. Then, since H\"older functions are dense in the uniform topology in the space of continuous functions, there is a sequence of H\"older functions $F_m$ that converge to $F$ uniformly. Let $[\p_m] \in \Qm^{++}$ be the corresponding quasimorphisms that are represented by the $F_m$. By Lemma \ref{lem.unifconv} the sequence $[\p_m]$ converges in $\MC(\Qm^{++})$ to some $[\p]$ that is represented by $F$. If $[\p]$ belonged to $\Qm^{++}$ then by Lemma \ref{lem.uniquegibb} $F$ would have a single Gibb's state yet this is not the case. We conclude that $[\p]$ does not belong to $\Qm^{++}$.
\end{proof}


\section{Applications}\label{sec.apps}

\subsection{RAAGs and RACGs} \label{sec.rags}
Suppose that $\Gamma$ is a right angled Artin group or right angled Coxeter group that is not virtually cyclic and does not split as a direct product. Let $d$ denote the word metric on $\G$ associated to the standard vertex generating set.

Given a (possibly asymmetric) $\G$ invariant metric  $d_\ast : \G \times \G \to \R$, let $(\cdot|\cdot)_{d_\ast}$ denote (the possibly asymmetric) Gromov product
\[
(x|y)_{d_\ast} = \frac{1}{2} \left( d_\ast(x,o) + d_\ast(o,y) - d_\ast(x,y) \right)
\]
where $x,y \in \G$ and $o \in \G$ is the identity. The Gromov product for an arbitrary function $\p : \G \times \G \to \R$ is defined analogously. 
\begin{definition}\label{def.gp}
Let $\mathcal{H}_\G^d$ denote the collection of (possibly asymmetric) functions $\p : \G \times \G \to \R$ that are $\G$-invariant and satisfy the following: there exists $c_1, c_2 >0$ such that
\[
|(x|y)_{\p}|\le c_1(x|y)_d + c_2
\]
for all $x,y \in \G$.  We write $\mathcal{D}_\G^d \subset \mathcal{H}_\G^d$ for the collection of (possibly asymmetric) metrics in $\mathcal{H}_\G^d$.
\end{definition}

We say that two functions $\p_1, \p_2 \in \mathcal{H}_\G^d$ are roughly similar if there exists $\tau, C>0$ such that
\[
|\p_1(x,y) - \tau \p_2(x,y)|< C
\]
for all $x,y \in \G$. Let $\mathscr{D}_\G^d$ and $\mathscr{H}_\G^d$ denote the rough similarity classes of $\mathcal{D}_\G^d$ and $\mathcal{H}_\G^d$ respectively. We can think of $\mathscr{D}_\G^d$ as the collection of genuinely different metrics that fellow travel $d$. The collection $\mathscr{D}_\G^d$ (hence $\mathscr{H}_\G^d$) is uncountable.

\begin{lemma}\label{lem.uncount}
The space $\mathscr{D}_\G^d$ is uncountable.
\end{lemma}

\begin{proof}
We recall that $\G$ admits a non-trivial morphism to $\R$ (and in fact admits many quasimorphisms).
Now suppose $q : \G \to \R$ is a a non trivial homomorphism.
Now simply note that for any $\epsilon > 0$ sufficiently small the function 
\[   
d_{\epsilon,q}(x,y) :=   d(x,y) + \epsilon q(x^{-1}y)
\]
is a metric that fellow travels $d$ and is quasi-isometric to $d$. It is easy to check that $d_{\epsilon,q}$ and $d_{\epsilon',q}$ are not roughly similar unless $\epsilon = \epsilon'$. 
\end{proof}

\begin{remark}
Note that the metrics constructed in the above proof are not symmetric. The collection of equivalence classes represented by genuine, symmetric metrics in $\mathscr{D}_\G^d$ is also uncountable: the same proof as above can be applied by taking $d + \epsilon d_\Cc$ where $d_\Cc$ is the action of $\G$ on the contact graph, see Section \ref{sec.con}.
\end{remark}

Given $d_\ast \in \Dc_\G^d$ its stable translation length is
\[
\ell_{d_\ast}[x] = \lim_{n\to\infty} \frac{d_\ast(o,x^n)}{n}
\]
for each conjugacy class $[x]$. We then define the dilation of a pair of metrics $d_1,d_2 \in \Dc_\G^d$ to be
\[
\Dil(d_1,d_2) = \sup_{[x] : \ell_d[x] > 0} \frac{\ell_{d_1}[x]}{\ell_{d_2}[x]}
\]
where the supremum is over all conjugacy classes with strictly  positive $d$ (and hence $d_1,d_2$) stable translation length.
Using the dilations we set
\[
\D([d_1],[d_2]) = \log ( \Dil(d_1,d_2) \, \Dil(d_2,d_1))
\]
for $[d_1],[d_2] \in \scrD_\G^d$. It is easy to check that this is well-defined and that $\D([d_1],[d_2]) \ge 0$ for all $[d_1], [d_2] \in \scrD_\G^d$.

\subsubsection{Proofs of results for RAAGs/RACGs}
We now move on to prove our results for RAAGs and RACGs from the introduction.  We require the following result that associates a subshift of finite type to a RAAG or RACG.

\begin{theorem} [\cite{HM}, \cite{gttgraph}]\label{thm.graph}
Suppose $\G$ is a RAAG or RACG that does not split as a direct product and that $S$ is the vertex generating set for $\G$.
Then there exists a pair $(\mathcal{G}, \lambda)$ where 
\begin{enumerate}
\item $\Gc=(V, E, \{\ast_1, \ldots, \ast_k\})$ is a finite directed graph with a collection of vertices $\ast_1, \ldots, \ast_k$ called the \textit{initial states}; and,
\item $\lambda: V \to S$ is a labelling such that
for a directed edge path $(x_0, x_1, \dots, x_{n-1})$ (where the $x_j$ are vertices)
there is an associated path in the Cayley graph $\Cay(\G, S)$ beginning at the identity: the path corresponds to
\[
(o, \lambda(x_0), \lambda(x_0)\lambda(x_1), \dots, \lambda(x_0)\cdots \lambda( x_{n-1}))
\]
\end{enumerate}
that satisfies the following:
Let $\ev$ denote the map that sends a finite path in $\mathcal{G}$ starting in $\{\ast_1,\ldots, x_k\}$ to the endpoint of the corresponding paths in $\Cay(\G,S)$, i.e.
$\ev(x_0,\ldots, x_{n-1}) = \lambda(x_0)\cdots \lambda(x_{n-1})$. Then
\begin{enumerate}
\item for each vertex $v\in V$ there exists a directed path from $\ast_j$ to $v$ for each $j=1,\ldots,k$;
\item for each directed path in $\Gc$ the associated path in $\Cay(\G, S)$ is a geodesic; and,
\item the map $\ev$ defines a bijection between the set of directed paths starting in $\{ \ast_1,\ldots, \ast_k\}$ in $\Gc$ and $\G$.
\end{enumerate}
Furthermore, once the states $\ast_1, \ldots, \ast_k$ are removed from $\mathcal{G}$, the resulting graph is a single component and the subshift of finite type defined over $\mathcal{G}$ is mixing.
\end{theorem}

\begin{remark}
The above theorem is a slight restatement of the cited works \cite{HM}, \cite{gttgraph}. Indeed, in  these works the graph structure has the labelling $\lambda$ defined over directed edges in $\Gc$ opposed to vertices. We have chosen to use a labelling over vertices however as it makes the above result more easily applicable in our geometric setting. 
\end{remark}

Now, let $(\Gc, \lambda)$ be the graph and labelling associated to $\G$ by Theorem \ref{thm.graph}. The graph gives rise to a subshift of finite type $\Sigma_A$ is the obvious way: the vertices $v$ in $\Gc$ are labelled $1,\ldots, n$ and the $n \times n$ transition matrix $A$ has $(i,j)$th entry $1$ if there is an edge in $\Gc$ from vertex $i$ to $j$ and is $0$ otherwise. Let $A'$ be the matrix $A$ but with the rows and columns corresponding to the states $\{\ast_1,\ldots, \ast_k\}$ removed and write $\Gc'$ for the graph $\Gc$ with the states $\ast_1,\ldots, \ast_k$ removed. By Theorem \ref{thm.graph} $A'$ is aperiodic and $\Sigma_{A'}$ is mixing. We want to construct quasimorphisms on $\Sigma_{A'}$ corresponding to $\p \in \Hc_\G^d$.

To achieve this, define, for each finite word $(x_0,\ldots, x_{n-1})$,
\[
\varphi(x_0,\ldots,x_{n-1}) = \p(o, \ev(x_0, \ldots, x_{n-1})).
\]
 The Gromov product condition on $\p$ guarantees precisely that $\varphi$ belong to $\qm(\Sigma_A)$. Indeed, let $(x_0,\ldots, x_{n+m-1})$ be a path in $\Gc$. Then $(x_0,\ldots, x_{n+m-1})$ correspond to a geodesic for $d$. In particular, if we set
\[
x^{-1} = \ev(x_0,\ldots,x_{n-1}) \in \G \ \ \text{ and } y = \ev(x_n,\ldots, x_{n+m-1}) \in \G,
\]
 then $d(o,x^{-1}) = n, d(o,y) = m, d(x,y) = d(o,x^{-1}y) = n+ m$ and 
 \[
 (x|y)_d = n + m - (n+m) = 0.
 \]
Therefore 
\[
|\varphi(x_0,\ldots,x_{n+m}) - \varphi(x_0,\ldots,x_n) - \varphi(x_{n+1}, \ldots, x_{n+m})| = |(x|y)_{\p_j}| \le c_1(x|y)_d + c_2 = c_2
\]
where $c_1, c_2$ is the constant coming from Definition \ref{def.gp}. We deduce that both $\varphi$ is a quasimorphism $\qm(\Sigma_{A'})$.

\begin{definition}\label{def.F}
Let $F: \Hc_\G^d \to \qm(\Sigma_{A'})$ denote the map that sends $\p \in \Hc_\G^d$ to $\varphi \in \qm(\Sigma_{A'})$. 
\end{definition}
\noindent Note that, if $d_\ast \in \Dc_\G^d$ then $F(d_\ast) \in \qm^{++}(\Sigma_{A'})$ since $d_\ast$ is quasi-isometric to $d$.

We are now ready to establish some of the main theorems from the introduction.

\begin{proof}[Proof of Theorem \ref{thm.rags1}]
Let $\theta_{\p_2/\p_1}$ be the  Manhattan curve for $\p_1, \p_2 \in \mathcal{H}_\G^d$. 
Let $\varphi_1 = F(\p_1)$ and $\varphi_2 = F(\p_2)$ where $F$ is the function defined above.
By construction, for each $s,t \in \R$,
\[
\sum_{x \in \G} e^{-s\p_1(o,x) - t\p_2(o,x)} = \sum_{w \in \Pi} e^{-s\varphi_1(w) - t\varphi_2(w)} =: Q(s,t)
\]
where $\Pi = \{ w = (x_0, \ldots, x_{n-1}) \in \W(A) : x_0 \in \{\ast_1,\ldots, \ast_k\} \}$. It is easy to see that, since $\p_1,\p_2$ are quasimorphisms and $A'$ is connected, the two variable series $Q(s,t)$ converges at $(s,t) \in \R^2$ if and only if the series
\[
\sum_{w \in \W(A')} e^{-s\varphi_1(w) - t\varphi_2(w)}
\] 
converges at $(s,t)$. This shows that the Manhattan curve $\theta_{\p_2/\p_1}(t)$ is precisely $\Pr(\varphi_2,\varphi_1)(-t)$ for each $t\in\R$ over the subshift $\Sigma_{A'}$. The $-t$ comes from the slight difference in parameterizations in the definitions for the pressure curve and Manhattan curves (we have followed the original definitions for both the Manhattan curve and pressure curves). In particular, by Theorem \ref{thm.rigidity} we deduce that $\theta_{\p_2/\p_1}$ is $C^1$ and is a straight line if and only if $\varphi_1$ and $\varphi_2$, considered as quasimorphisms in $\qm(\Sigma_{A'})$, are roughly similar. To conclude the proof we need only show that $\varphi_1 \sim \varphi_2$ if and only if $\p_1$ and $\p_2$ are roughly similar. This follows easily however from the quasimorphism property and the fact that there is a bijection between the paths in $\Gc$ starting in $\{\ast_1,\ldots, \ast_k\}$ and the group elements in $\G$.
\end{proof}

In the above proof we established the following result.

\begin{corollary}\label{cor.man}
Suppose that two metrics $\p_1, \p_2 \in \mathcal{H}_\G^d$ have non-empty Manhattan curve $\theta_{\p_1/\p_2}$. Then the for each $t \in \R$ in the domain of $\theta_{\p_1/\p_2}$ we have that $\theta_{\p_1/\p_2}(t) = \Pr(F(\p_1),F(\p_2))(-t)$.
\end{corollary}

We are now ready to prove Theorem \ref{thm.rags2}.

\begin{proof}[Proof of Theorem \ref{thm.rags2}]
Fix $R >0$ and consider the sequence
\[
\sum_{x \in \G: |\p_1(o,x) - n| \le R}  e^{t\p_2(o,x)}
\]
for $n\ge 1$. Note that, for any fixed symbol $x_0$, we have that 
\[
\sum_{x \in \G: |\p_1(o,x) - n| \le R}  e^{t\p_2(o,x)} \ge \sum_{w \in \W(A') : w \in [x_0], |\varphi(w) - n|< R} e^{t\varphi_2(w)}
\]
for each $n\ge1$ since each path starting at $x_0$ corresponds to a unique group element in $\G$.

Further, since there is a bijection between elements in $\G$ and paths starting in $\{\ast_1,\ldots, \ast_k\}$ and by the quasimorphism property for $\varphi_1, \varphi_2$, there exists $C_t>0$ depending only on $t$ and $R' >R$ such that
\[
\sum_{x \in \G: |\p_1(o,x) - n| \le R}  e^{t\p_2(o,x)} \le C_t  \sum_{w \in \W(A') : |\varphi(w) - n|< R'} e^{t\varphi_2(w)}
\]
for each $n\ge1$. Combining these two inequalities and applying Lemma \ref{lem.expgrow} and Remark \ref{rem.expgrow} implies that
for any sufficiently large $R >0$,
\[
\sum_{x \in \G: |\p_1(o,x) - n| \le R}  e^{t\p_2(o,x)} \asymp e^{n\Pr(\varphi_2,\varphi_1)(t)}
\]
as $n\to\infty$. We can then apply the G\"artner-Ellis Theorem to deduce that there exists $\tau(\p_2/\p_1)$ which is equal to the derivative of $\Pr(\varphi_2,\varphi_1)$ at $0$ (or minus the derivative of $\theta_{d_2/d_1}$ at $0$) such that
\[
 \left| \left\{ x \in \G: |\p_1(o,x) - n| \le R \ \text{ and } \, \left| \frac{\p_2(o,x)}{n} - \tau(\p_2/\p_1)\right| > \epsilon \right\}\right|
\]
grows exponentially slower than $\left|\left\{x \in \G : |\p_1(o,x) - n| \le R \right\}\right|$ as $n\to\infty$. It is easy to check that this implies that 
\[
\frac{1}{\#\{ x \in \G: \p_1(o,x) < T\}}  \sum_{\p_1(o,x) < T} \frac{\p_2(o,x)}{T} \to \tau(\p_2/\p_1)
\]
as $T\to\infty$. The inequality $\tau(\p_2/\p_1) \ge h(\p_1)/h(\p_2)$ and the equivalence statements now follow from Theorem \ref{thm.rigidity} since the Manhattan curve $\theta_{\p_2/\p_1}(t)$ can be realised as $\Pr(\varphi_2,\varphi_1)(-t)$ for $t \in \R$.
\end{proof}

\subsection{Space of metric structures}
We now prove our main results concerning the space of metric structures for RAAGs and RACGs.

\begin{lemma}\label{lem.comprag}
Suppose that $d_1, d_2 \in \Dc_\G^d$. Then
\[
\Dil(d_1,d_2) = \Dil(F(d_1), F(d_2)).
\]
Furthermore, this fact implies that there exists $C >0$ such that
\[
\Dil(d_1,d_2)^{-1}d_1(x,y) -C \le d_2(x,y) \le \Dil(d_2,d_1)d_1(x,y) + C
\]
for all $x,y \in \G$.
\end{lemma}
\begin{proof}
Let $\varphi_1 = F(d_1)$ and $\varphi_2 = F(d_2)$.
First note that, by Lemma  \ref{lem.comp} there exists $C>0$ such that
\[
\varphi_1(w) \le \Dil(\varphi_1, \varphi_2) \varphi_2(w) + C
\]
for all $w \in \W(A')$. Since every element in $\G$ can be realised as a finite path in $\Gc$ and, up to removing an initial state $\ast_1,\ldots, \ast_k$, as a finite path in $\Gc'$, we deduce that there exists $C' > 0$ such that
\begin{equation}\label{eq.comp}
d_1(o,x) \le \Dil(\varphi_1, \varphi_2) d_2(o,x) + C'
\end{equation}
for all $x \in \G$. In particular we deduce that 
\[
\Dil(\varphi_1, \varphi_2) \ge \Dil(d_1,d_2).
\]

For the opposite inequality note that, by the definition of $\varphi_1$, for any $\{w\} \in \Per(A')$ with $x_w = \ev(w) \in \G$ we have that
\[
\ov\varphi_1\{w\} = \lim_{n\to\infty} \frac{\varphi_1(w^n)}{n} = \lim_{n\to\infty} \frac{d_1(o,x_w^n)}{n} = \ell_{d_1}[x_w].
\]
The same is true for $\varphi_2$. It follows that $\Dil(\varphi_1, \varphi_2) \le \Dil(d_1,d_2)$ and the dilations agree.

For the furthermore statement, we simply set $\Dil(\varphi_1, \varphi_2) = \Dil(d_1,d_2)$ in $(\ref{eq.comp})$ above.
\end{proof}

\begin{corollary}
The distance $\D$ on $\Dc_\G^d$ determines a metric on $\Dc_\G^d$.
\end{corollary}
\begin{proof}
It is easy to see that $\D$ is symmetric, positive and satisfies the triangle inequality and so we are left to check that $\D$ distinguishes point. However, this follows easily form Lemma \ref{lem.comprag} which implies $\D([d_1],[d_2]) = 0$ if and only if $\Dil(d_1,d_2)\Dil(d_2,d_1) = 1$ if and only if $d_1$ and $d_2$ are roughly similar, i.e. $[d_1] = [d_2]$.
\end{proof}

We now want to establish that the Manhattan curve $\theta_{d_2/d_1}$ for metrics in $\Dc_\G^d$ determines a curve in $\Dc_\G^d$.

\begin{proposition}\label{prop.curve}
Take $d_1,d_2 \in \Dc_\G^d$. Then, for each $t \in \R$ there exists a metric $d_t \in \Dc_\G^d$  and $C_t >0$ such that 
\[
|d_t(x,y) - (\theta_{d_2/d_1}(t)  d_1(x,y) +  td_2(x,y))| < C_t
\]
for all $x,y \in \G$. Furthermore, the $C_t$ can be chosen so that 
\[
C_t \le \lambda \, |t| + \lambda
\]
for some constants $\lambda >0$ and for all $t\in\R$.
\end{proposition}

\begin{proof}
Let $(\Gc,\lambda)$ be the pair as above and let $\varphi_1,\varphi_2 \in \qm^{++}(\Sigma_{A'})$ be the quasimorphisms associated to $d_1,d_2$ as above. Now consider the quasi-morphism
\[
\varphi_t = t \varphi_2 + \Pr(\varphi_2,\varphi_1)(-t)\varphi_1 = t \varphi_2 + \theta_{d_2/d_1}(t) \varphi_1
\]
which belongs to $\qm^{++}(\Sigma_{A'})$ by Lemma \ref{lem.++}. 

By Corollary \ref{cor.lamb} we have that $\varphi_t \in \qm^{++}(\Sigma_{A'})$ and there exists $R_t > 0$ such that $\varphi_t(w) \ge -R_t$ for all $w\in\W(A')$. Furthermore, we may assume that there is $\lambda > 0$ such that $R_t < \lambda|t| + \lambda$.

Now define, for $x\in \G$,
\[
\p_t(o,x) = \varphi_t(w_x) +R_t
\]
where $w_x = (x_0, \ldots, x_{n-1}) \in \W(A')$ is obtained from unique word $(\ast_j, x_0, \ldots, x_{n-1}) \in \W(A)$ such that $\ev(\ast_j, x_0, \ldots, x_{n-1}) = x$ by deleting the initial state. Then define $\p_t(x,y) = \varphi_t(w_{x^{-1}y}) + R_t$ for all $x,y \in \G$. By construction $\varphi_t(x,y) \ge 0$ for all $x,y\in\G$ and 
\begin{equation}\label{eq.boo}
|\p_t(x,y)  - (\theta_{d_2/d_1}(t)  d_1(x,y) +  td_2(x,y))|  \le  R_t + C(|\Pr(\p_2,\p_1)(t)|  + |t|) 
\end{equation}
for some $C >0$ independent of $x,y,t$.

Note that
\[
(x|y)_{\p_t} = \varphi_t(w_{x^{-1}}) + \varphi_t(w_y) - \varphi_t(w_{x^{-1}y})
\]
for each $x,y \in \G$. Now, since $\varphi_t$ belongs to $\qm^{++}(\Sigma_{A'})$, there exists $M \ge 1$ such that, for any $x,y \in \G$ there is $w_{x^{-1},y} \in \W(A')_{\le M}$ such that $\varphi_t(w_{x^{-1}y}) < \varphi_t(w_{x^{-1}} w_{x^{-1},y} w_y)$. In particular we deduce that 
\[
(x|y)_{\p_t}  \ge \varphi_t(w_{x^{-1}}) + \varphi_t(w_y) - \varphi_t(w_{x^{-1}} w_{x^{-1},y} w_y) \ge -\varphi_t(w_{x^{-1},y}) - 2D_{\varphi_t}
\]
by the quasimorphism property.

In particular if we let $C_t = 2D_{\varphi_t} + \sup_{w \in \W(A')_{\le M}} \varphi_t(w)$ then
\[   
d_t(x,y) := 
     \begin{cases}
      \p_t(x,y) + C_t &\quad\text{if $x \neq y$}\\
      0 &\quad\text{if $x =y$}
     \end{cases} 
\]
determines a positive function that satisfies the triangle inequality (i.e. is a possibly asymmetric metric) and that is quasi-isometric to $d$ since $\varphi_t \in \qm^{++}(\Sigma_{A'})$.

Now, from the definition of $d_t$ and $(\ref{eq.boo})$ we see that 
\begin{equation}\label{eq.end}
|d_t(x,y)  - (\theta_{d_2/d_1}(t)  d_1(x,y) +  td_2(x,y))|  \le  R_t + C(|\Pr(\p_2,\p_1)(t)|  + |t|) + C_t.
\end{equation}
As in the proof of Corollary \ref{cor.lamb} it is easy to check that there exists $\lambda' > 0$ such that
\[
|C(|\Pr(\p_2,\p_1)(t)|  + |t|) + C_t| \le \lambda' |t| + \lambda'
\]
for all $t\in\R$. Since $|R_t| < \lambda |t| + \lambda$, the furthermore statement follows from $(\ref{eq.end})$.

Lastly, since $d_t(x,y)$ is within uniformly bounded distance of $\theta_{d_2/d_1}(t)  d_1(x,y) +  td_2(x,y)$ for each $x,y \in \G$ from which we deduce that there is $L_t, M_t >0$ such that 
\[
(x|y)_{d_t} \le |\theta_{d_2/d_1}(t)|(x|y)_{d_1} + |t|(x|y)_{d_2} + L_t \le M_t (x|y)_d + M_t
\]
for all $x,y \in \G$, where in the last inequality we have used the upper Gromov product bounds for $d_1,d_2$ for the Gromov product for $d$. That is, $d_t$ belongs $\Dc_\G^d$ as required.
\end{proof}

We now prove the following result, from which we can deduce our final results for RRAGs.

\begin{theorem}\label{thm.RAAGembed}
Let $(\mathscr{D}_\G^d, \D)$ be the space of metrics on $\G$ with subshift $\Sigma_{A'}$  and map $F : \Dc_\G^d \to \qm^{++}(\Sigma)$ as above. Then $F$  descends (via quotienting by the equivalence relations) to a well-defined, isometric embedding $\ov{F} : (\mathscr{D}_\G^d, \D) \to (\Qm^{++}(\Sigma), \Qq)$, i.e. $\ov{F}$ is injective and preserves distances. Furthermore $F$ satisfies 
\[
\theta_{d_2/d_1}(t) = \Pr(F(d_2), F(d_1))(-t)
\]
for each $t \in \R$  (where $\theta_{d_2/d_1}$ is the Manhattan curve)
and so
\[
\ov{F}[td_2 + \theta_{d_2/d_1}(t)d_1] = [t F(d_2) + \Pr(F(d_2), F(d_1))(-t) F(d_1)]
\]
for each $t \in \R$.
\end{theorem}

\begin{proof}
To show that $\ov F$ is an isometric embedding it suffices to show that, for any $[d_1], [d_2] \in \qm^{++}(\Sigma_{A'})$, 
\[
\D([d_1],[d_2]) = \Qq([\varphi_1], [\varphi_2]).
\]
where $F(d_1) = \varphi_1$, $F(d_2) = \varphi_2$.  However, this follows immediately from Lemma \ref{lem.comprag}.

The furthermore statement was established in Corollary \ref{cor.man} given that, by Proposition \ref{prop.curve} the combination $td_2 + \theta_{d_2/d_1}(t)d_1$ determines a metric in $\Dc_\G^d$ for each $t \in \R$.
\end{proof}

\begin{proof}[Proof of Theorem \ref{thm.ragsspace}]
This now follows immediately from Proposition \ref{prop.curve} and  Theorem \ref{thm.RAAGembed}.
\end{proof}

We now study the Manhattan boundary $\partial_M\scrD_\G^d$. Given $d_1,d_2 \in \Dc_\G$ we note that
\begin{align*}
\lim_{t\to\infty} \frac{1}{t}( td_2(x,y) + \theta_{d_2/d_1}(t) d_1(x,y)) &= d_2(x,y) + \lim_{t\to\infty} \frac{\Pr(d_2,d_1)(-t)}{t} d_1(x,y)\\
&= d_2(x,y) - \Dil(d_1,d_2)^{-1} d_1(x,y) \\
&= \frac{1}{\Dil(d_1,d_2)}  \left( \Dil(d_1,d_2) d_2(x,y) - d_1(x,y) \right)
\end{align*}
by Lemma \ref{lem.asympgrad}. Similarly
\[
 - \lim_{t\to -\infty} \frac{1}{t}( td_2(x,y) + \theta_{d_2/d_1}(t) d_1(x,y)) = \Dil(d_2,d_1) d_1(x,y) - d_2(x,y).
\]
\begin{proposition}\label{prop.inftymetric}
There exist constants $C_1, C_2 > 0$ such that the functions 
\[
d_\infty(x,y) = \begin{cases}
       \Dil(d_1,d_2) d_2(x,y) - d_1(x,y) + C_1 &\quad\text{if $x \neq y$}\\
      0 &\quad\text{if $x =y$}
     \end{cases} 
     \]
     and
     \[
      d_{-\infty}(x,y) = \begin{cases}
       \Dil(d_2,d_1) d_1(x,y) - d_2(x,y) + C_2 &\quad\text{if $x \neq y$}\\
      0 &\quad\text{if $x =y$}
     \end{cases} 
\]
determine metrics in $\Hc_\G^d$ that are not quasi-isometric to $d$.
\end{proposition}

\begin{proof}
We prove the assertion from $d_\infty$ and the same proof works for $d_{-\infty}$. By Proposition \ref{prop.curve} there is $\lambda > 0$ and a metric $d_t \in \Dc_\G^d$ such that 
\[
|d_t(x,y) - (\theta_{d_2/d_1}(t)  d_1(x,y) +  td_2(x,y))| < \lambda |t| +\lambda.
\]
In particular, 
\[
\lim_{t\to\infty} \frac{1}{t}( td_2(x,y) + \theta_{d_2/d_1}(t) d_1(x,y)) \ge - \lambda
\]
since $d_t(x,y) \ge 0$ for all $t \in\R$, $x,y \in \G$. It follows that 
\[
 \Dil(d_1,d_2) d_2(x,y) - d_1(x,y)  \ge - \lambda \Dil(d_1,d_2).
\]
We set $C_1 = \lambda \Dil(d_1,d_2)$ in the definition of $d_\infty$. This then gives that $d_\infty$ is a positive function. Note that $d_\infty$ satisfies the triangle inequality since it is obtained (up to a constant) as a limit of the metrics $d_t$ which each satisfy the triangle inequality. Hence $d_\infty$ is a metric and it is easy to check that there are $c_1, c_2 >0$ such that 
\[
(x|y)_{d_\infty} \le c_1(x|y)_d + c_2
\]
since both $d_1, d_2$ satisfy such bounds. This concludes the proof.
\end{proof}

\begin{definition}
Define $\partial_M \scrD_\G^d \subset \mathscr{H}_\G^d$ to be the collection of rough similarity classes contain a metric of the form $d_{\infty}$ or $d_{-\infty}$ as constructed above. We call $\partial_M \scrD_\G^d$ the Manhattan boundary.
\end{definition} 

We now prove Theorem \ref{thm.raagboundary} that precisely describes when a metric is in the Manhattan boundary
\begin{proof}[Proof of Theorem \ref{thm.raagboundary}]
We need to show that any metric $d' \in \mathcal{H}_\G^d$ that satisfies $(1), (2)$ and $(3)$ from the theorem can be realised as the endpoint of a Manhattan geodesic, i.e. there exists $[d_1], [d_2] \in \Dc_\G^d$ such that, up to a uniformly bounded constant, $d'$ is equal to $ \Dil(d_1,d_2) d_2(x,y) - d_1(x,y)$.

To that end, consider the metrics $d_1 = d + d'$ and $d_2 = d$ which it is easy to verify both live in $\Dc_\G^d$. Now, since $\ov{d_1} \ge \ov{d_2}$ we see that $\Dil(d_2,d_1) \le 1$. Then by Lemma  \ref{lem.comprag} there is $C>0$ such that
\[
(1- \Dil(d_2,d_1))d_1(x,y) \le d_1(x,y) - d_2(x,y) + C \le d'(x,y) + C
\]
for all $x,y \in \G$. Since $d'$ is not quasi-isometric to $d$ we must have $\Dil(d_2,d_1) \ge 1$. 

Therefore $\Dil(d_2,d_1) = 1$ and $d' = \Dil(d_2,d_1) d_1 - d_1$ implying that $[d']$ is the endopoint of the Manhattan geodesic connecting $[d_1]$ and $[d_2]$.
\end{proof}

\begin{example}
The function $\#_v : \G \to \R$ that counts the number of occurrences of a fixed vertex $v^{\pm1}$ in the shortest word representation of a group element  (i.e. $\p(x,y) = \#_v(x^{-1}y)$) belongs to $\partial_M \Dc_\G^d$.
\end{example}


To conclude this section we prove a precise large deviation theorem.

\begin{theorem}\label{thm.ldpraags}
Let $\G$ be a RAAG or RACG that does not split as a direct product and write $d$ for the standard vertex generating set. Let $\p \in \Hc_\G^d$, $d_\ast \in \Dc_\G^d$ and for each $T >0$ let $\mathcal{B}_{d_\ast}(T) = \{x \in \G : d_\ast(o,x) <T\}$. Then there exists a function $I : \R \to \R$ that is finite on a closed interval $[a,b] \subset \R$, infinite outside of this interval  and $C^1$ on $(a,b)$ such that the following holds. For any non-empty open set $U \subset \R$ and closed set $V \subset \R$ with $U \subset V$ we have that
    \begin{align*}
        -\inf_{s \in U} I(s) &\le \liminf_{T\to\infty} \frac{1}{T} \log\left( \frac{1}{\left|\mathcal{B}_{d_\ast}(T)\right|} \left|\left\{x \in \mathcal{B}_{d_\ast}(T): \frac{\p(x)}{ T}\in U \right\}\right|\right)\\
        &\le \limsup_{T\to\infty} \frac{1}{T} \log\left( \frac{1}{\left|\mathcal{B}_{d_\ast}(T)\right|} \left|\left\{x \in \mathcal{B}_{d_\ast}(T): \frac{\p(x)}{ T}\in V \right\}\right|\right) \le -\inf_{s \in V} I(s).
    \end{align*}
    \end{theorem}
    
    \begin{proof}
In the proof of Theorem \ref{thm.rags2} we saw that for any sufficiently large $R >0$,
\[
\sum_{x \in \G: |d_\ast(o,x) - n| \le R}  e^{t\p(o,x)} \asymp e^{n\Pr(F(\p),F(d_\ast))(t)}
\]
as $n\to\infty$ where $F$ is as in Definition \ref{def.F}. We can then apply the G\"artner-Ellis Theorem as in the proof of Theorem \ref{thm.ldp}.
    \end{proof}
    
\subsection{The action on the contact graph}\label{sec.con}
Let $\G$ be a right-angled Artin group that is not virtually cyclic and does not split as a direct product. Its associated contact graph
$\Cc X$ is a hyperbolic graph encoding the large-scale geometry of the universal cover of the Salvetti complex. The group $\G$ acts naturally on $\Cc X$ by graph automorphisms,
providing a useful perspective on the geometry and dynamics of $\G$ see \cite{hagen} for a comprehensive account of the contact graph.
Fixing a base vertex $v_0\in \Cc X$, this action induces an
orbit metric on $\G$ by
\[
d_{\mathcal{C}}(g,h)
=
d_{\mathcal{C}X}(g \cdot v_0,h \cdot v_0),
\ \text{ for } \  g,h\in \Gamma.
\]
\begin{proposition}\label{prop.conbound}
The metric $d_\Cc$ determines a function in $\Hc_\G^d$, i.e. the map $(g,h) \mapsto d_\Cc(g,h)$ is in $\Hc_\G^d$. In fact, $[d_\Cc]$ belongs to $\partial_M \scrD_\G^d$.
\end{proposition}

We thank Mark Hagen for suggesting the following proof.

\begin{proof}
To avoid introducing too much introductory material, we follow the notation established in \cite{hagen} and refer to this reference for definitions and details.

Let $(X,d)$ be the universal cover of the Salvetti complex and $d$ is the piecewise $l_1$ metric. Alternatively, $X$ could be the $0$-skeleton, where $d$ restricts to the combinatorial metric.  This is a median metric space and we write $m : X^3 \to X$ for the median operator.  For any hyperplane $h$ of $X$, $\mathcal{N}(h)$ is the (closed) carrier of $h$ which consists of the closed cubes intersecting $h$.

Let $\Cc X$ be the contact graph as above. We let $\pi$ be the projection that sends each $x \in X$ to the clique in $\Cc X$ corresponding to the set of hyperplanes whose carriers contain $x$.  Let $m'$ be the coarse median operator on $\Cc X$.  Then $m'$ has the property that if $a_1,a_2,a_3$ are in $\Cc X$, then for all distinct pairs $a_i,a_j$, any geodesic from $a_i$ to $a_j$ in $\Cc X$ comes $D$-close to $m'(a_1,a_2,a_3)$. Further, $m'$ is coarsely characterised by this property.\\

Claim: There is $K >0$ such that 
\[
d_{\Cc X}(\pi(m(x,y,z)), m'(\pi(x),\pi(y),\pi(z))) \leq K \ \ \text{ for all $x,y,z$ corresponding to vertices in  $X$}.
\]

Proof:  Let $h_x$ and $h_y$ be hyperplanes in $\pi(x)$ and $\pi(y)$ respectively, so by definition, $x$ is in $\mathcal{N}(h_x)$ and $y$ is in $\mathcal{N}(h_y)$. 
Points $(1)$ and $(2)$ in Lemma 4.11 from \cite{hagen} imply there exists a sequence of hyperplanes $h_x = h_1,...,h_k = h_y$ such that $h_i$ and $h_{i+1}$ are $\Cc X$--adjacent for all $i$, and the resulting edge path in $\Cc X$ is a geodesic of $\Cc X$.  Call this $\Cc X$-geodesic $A$.  Points $(3)$ and $(4)$ from \cite[Lemma 4.11]{hagen} imply that these $h_i$ can be chosen with an additional property, namely that there are paths $g_i$ in $\mathcal{N}(h_i)$ so that the concatenation $g=g_1g_2 \cdots g_k$ is defined and is a path from $x$ to $y$ in $X$, and, moreover, $g$ is a geodesic of $x$.

Now, if $t$ is a point in $g$, then $t$ is in some $\mathcal{N}(h_i)$, so $\pi(t)$ is a clique in $\Cc X$ containing the vertex corresponding to $h_i$.  Hence, in $\Cc X$, the set $\pi(g)$ is contained in the $1$-neighbourhood of $A$.  Now let $q$ be any other $d$-geodesic in $X$ from $x$ to $y$.  Then every hyperplane crossing $q$ crosses $g$, and hence crosses some $h_i$.  Moreover, each point in $q$ is in $\mathcal{N}(h)$ for some $h$ crossing $q$.  So $\pi(q)$ is in, say, the $L$-neighbourhood of $A$ in $\Cc X$.  We have just shown that if $B$ is any $\Cc X$-geodesic from $h_x$ to $h_y$, and $q$ is any $X$ geodesic from $x$ to $y$, then $\pi(q)$ is in the $K_0$-neighbourhood of $B$, where $K_0$ just depends on the hyperbolicity constant for $\Cc X$.

Now let $A'$ be a $\Cc X$ geodesic from $h_y$ to some $h_z$ in $\pi(z)$, and let $A''$ be a $\Cc X$ geodesic from $h_z$ to $h_x$.  Let $q,q',q''$ be geodesics in $\Cc X$ with $q$ joining $x,y$, and $q'$ joining $y,z$, and $q''$ joining $z,x$.  Choose these geodesics so that they all pass through $m=m(x,y,z)$ .  Let $A, A', A''$ be $\Cc X$ geodesics joining $h_x,h_y$, and $h_y,h_z$, and $h_z,h_x$ respectively.  Then $\pi(m)$ is $K_0$-close to $A$ in $\Cc X$, since $m$ is in $q$.  Similarly, $\pi(m)$ is $K_0$ close to $A'$ and $A''$.  Hence the point $\pi(m)$ is uniformly close to $m'(h_x,h_y,h_z)$, as required.\\

Now let $d'$ be the metric on $X$ obtained by coning off the hyperplane carriers; that is, we join two points by a new edge whenever they are contained in a common hyperplane carrier.  Then $X'=(X,d')$ is quasi-isometric to $\Cc X$, via a quasi-isometry that sends each $x$ into $\pi(x)$ \cite[Section 5.2]{hagen.2}.  In particular, it is hyperbolic and by abusing notation, we can also let $m'$ denote the coarse median on $X'$. Note that by the claim (up to enlarging $K$ if necessary), $d'(m(x,y,z), m'(x,y,z)) \leq K$. 
Also, by construction, the identity map $X \to X'$ is $1$-Lipschitz, i.e. $d(x,y) \geq d'(x,y)$ for all $x,y \in X$.

To conclude, suppose that $o,x,y$ are vertices of $X$.  Then the Gromov product in $X$ is $d(o,m(o,x,y))$.  In the hyperbolic space $X'$, the Gromov product is, up to uniform additive error, $d'(o,m',o,x,y))$.  And $d'(o,m'(o,x,y)) \leq d'(o,m(o,x,y))+K \leq d(o,m(o,x,y)) + K$. Since quasi-isometries of hyperbolic spaces passes to the Gromov product, there exists $A,B > 0$ such that
\[
(x|y)_{\Cc X} \le A(x|y)_{X'} + B \le A(x|y)_d + AK + B.
\]
This shows that $d_\Cc$ determines a function in $\Hc_\G^d$ as required.

Since $d_\Cc$ is $\G$-invariant and is not quasi-isometric to $d$, we deduce that, by Theorem \ref{thm.raagboundary}, $d_\Cc$ determines a point in $\partial_M \scrD_\G^d$.
\end{proof}

\begin{corollary}
The Manhattan curve for the pair $d,d_\Cc$ is $C^1$ and globally strictly convex.
\end{corollary}

\begin{remark}
Since $d_\Cc$ is not quasi-isometric to $d$ \cite{hagen}, $F(d_\Cc)$ is a quasimorphism in $\qm^{+}\backslash \qm^{++}$.
\end{remark}

Using our results for quasimorphisms we obtain the following statistical result comparing $d$ and $d_\Cc$. 
\begin{theorem}
Suppose that $\G$ is a right angled Artin group that is not virtually cyclic and does not split as a direct product. Let $d$ be the word metric on $\G$ associated to the canonical vertex generating set. Write $d_{\Cc}$ for the orbit metric on $\G$ coming from the action of $\G$ in the contact graph. Then, there exists $\tau > 0$ such that
\[
\limsup_{T\to\infty} \frac{1}{T} \log \left( \frac{1}{\left|\{ g \in \G: d(o,g) < T \right|} \biggl|\left\{ g \in \G : d(o,g) < T,  \left|\frac{d_\Cc(o,g)}{T} - \tau \right| \geq \epsilon \right\}\biggr|\right) < 0.
\]
\end{theorem}

This theorem shows that in $\G$ a typical element of $d$ length at most $T$ causes a displacement of approximately $\tau T$ through the action on $\Cc(\G)$. Note that $\tau$ is strictly positive.

\subsection{The space \texorpdfstring{$\Dc_\G$}{DcG} for hyperbolic groups}\label{sec.d_G}
In this section we show how to use our theory to recover the results of \cite{cantrell.tanaka.1}, \cite{cantrell.reyes.man}. We follow the notation used in \cite{cantrell.reyes.man} and direct the reader there for more details on the background material.

Let $\G$ be a non-elementary group. We write $\Dc_\G$ for the collection of metrics on $\G$ that are
\begin{enumerate}
\item quasi-isometric to a word metric;
\item Gromov hyperbolic; and, 
\item left $\G$-invariant.
\end{enumerate}
We write $\conj(\G)$ for the collection of non-torsion conjugacy classes in $\G$. Given $d \in \Dc_\G$ let $\ell_d : \conj(\G) \to \R$ denote the translation length function
\[
\ell_d[x]  = \lim_{n\to\infty} \frac{d(o,x^n)}{n}
\]
where $o \in \G$ denotes the identity element.
As above, we say that two metrics $d,d_\ast \in \Dc_\G$ are roughly similar if there exists $\tau, C >0$ such that
\[
\sup_{x,y \in \G} |d(x,y) - \tau d_\ast(x,y) | < C.
\]
Two metrics $d,d_\ast \in \Dc_\G$ are roughly similar if and only $\ell_d$ and $\ell_{d_\ast}$ are proportional.

Let $\mathscr{D}_\G$ be the quotient of $\Dc_\G$ by the equivalence relation of rough similarity. We write $[d] \in \scrD_\G$ for the equivalence class containing $d \in \Dc_\G$.

We can equip $\mathscr{D}_\G$ with a symmetrised version of the Thurston metric $\D$: given $d,d_\ast \in \Dc_\G$, let
\[
\Dil(d,d_\ast) = \sup_{[x] \in \conj(\G)} \frac{\ell_d[x]}{\ell_{d_\ast}[x]}
\]
and define
\[
\D([d],[d_\ast]) = \log \, (\Dil(d,d_\ast) \, \Dil(d_\ast,d)).
\]
We call $(\mathscr{D}_\G, \D)$ the space of hyperbolic metric structures on $\G$. We will show how to recover the results of \cite{cantrell.tanaka.1} and \cite{cantrell.reyes.man} using our machinery of quasimorphisms. All the results will be a corollary of the following theorem which is an analogue of Theorem \ref{thm.RAAGembed} above.

\begin{theorem}\label{thm.D_Gembed}
Let $(\mathscr{D}_\G, \D)$ be the space of hyperbolic metric structures on a non-elementary hyperbolic group $\G$ as above. Then there exists a subshift of finite type $\Sigma$  and a map $F : \Dc_\G \to \qm^{++}(\Sigma)$ that descends (via quotienting by the equivalence relations) to a well-defined, isometric embedding $\ov{F} : (\mathscr{D}_\G, \D) \to (\Qm^{++}(\Sigma), \Qq)$, i.e. $\ov{F}$ is injective and preserves distances. Furthermore $F$ satisfies 
\[
\theta_{d_\ast/d}(t) = \Pr(F(d_\ast), F(d))(-t)
\]
for each $t \in \R$  (where $\theta_{d_\ast/d}$ is the Manhattan curve)
and so
\[
\ov{F}[td_\ast + \theta_{d_\ast/d}(t)d] = [-t F(d_\ast) + \Pr(F(d_\ast), F(d))(-t) F(d)]
\]
for each $t \in \R$.
\end{theorem}

\begin{proof}
Given $\G$ we need to exhibit a subshift of finite type $\Sigma$. This subshift will come from the Cannon coding: a finite graph that codes $\G$ equipped with a generating set. To avoid having to introduce too much introductory material here, we refer the reader to Section 2 of~\cite{cantrell2025mixing}.

Fix a finite generating set $S$ for $\G$ and a Cannon coding for $\G,S$. Let $\Cc$ be a maximal component in this coding and write $\Sigma$ for the subshift defined over this component. Let $\W(\Cc)$ be the collection of words defined over the component $\Cc$. As for RAAGs and RACGs, the Cannon coding comes with a labelling map $\ev: \W(\Cc) \to \G$. \\

\textit{Definition:}
We define $F : \Dc_\G \to \Qm^{++}(\Sigma)$ in the following way. Given $d \in \Dc_\G$, let $F(d) = \p$ where given $w \in \W(\Cc)$, 
\[
\p(w) = d(o,g_w) \ \ \text{ where $g_w \in \G$ is the group element corresponding to $g$: $g_w = \ev(w)$.}
\]

By the Morse Lemma, $d$ is roughly additive along $d_S$ geodesic rays. Hence $\p$ is a quasi-morphism and, since $d$ is quasi-isometric to $d_S$, $d \in \qm^{++}(\Sigma)$. The map $\ov{F}$ is then obtained by taking the quotients by the equivalence relation of quasimorphism equivalence and rough similarity.\\

\textit{Injective:}
Suppose that $d,d_\ast \in \Dc_\G$ satisfy $[\p_d] := \ov{F}[d] = \ov{F}[d_\ast]=: [\p_{d_\ast}]$ where $\p_d, \p_{d_\ast} \in \qm^{++}(\Sigma)$. Then $\ov{\p_d}$ and $\ov{\p_{d_\ast}}$ are proportional. But note that, each $\{w\} \in \Per$ corresponds to a conjugacy class $[g_w] \in \conj(\G)$ where $g_w = \ev(w)$. Furthermore, as in the proof of Lemma \ref{lem.comprag} we have that
\[
\ov{\p_d}(w) = \ell_d[g_w] \ \ \text{ and } \ov{\p_{d_\ast}}(w) = \ell_{d_\ast}[g_w].
\]
Conversely, by Proposition 3.5 from \cite{cantrell2025mixing} for any non torsion $[g] \in \conj(\G)$ there exists $w_g \in \Fix$ such that $w_g$ corresponds to an element in $[g^{\pm M}]$.
In particular, the equality $[\p_d] = [\p_{d_\ast}]$ implies that $\ell_d$ and $\ell_{d_\ast}$ are proportional. In particular $[d] = [d_\ast]$ by Theorem 1 in \cite{cantrell.reyes.man} and we have shown that $\ov{F}$ is injective.\\

\textit{Isometric:}
Since there exists an integer $M \ge 1$ such that for any (non torsion) conjugacy class $[g] \in \conj(\G)$ there is  $\{w_g\} \in \Per$ corresponding to $[g^{\pm M}]$, we have that
\[
\Dil(d,d_\ast) = \sup_{[g] \in \conj'(\G)} \frac{\ell_d[g]}{\ell_{d_\ast}[g]} = \sup_{w\in\Fix} \frac{\ov{\p_d}\{w\}}{\ov{\p_{d_\ast}}\{w\}} = \Dil(\p_d,\p_{d_\ast}).
\]
It then follows from the definitions that
\[
\D([d],[d_\ast]) = \Qq([\p_d],[\p_{d_\ast}]).
\]
\\
\textit{Preserves geodesics:} 
It suffices to show that the Manhattan curve $\theta_{d_\ast/d}(t)$ for the pair $d,d_\ast$ agrees with the pressure curve $\Pr(\p_{d_\ast}, \p_d)(-t)$. To do so we follow the method presented in~\cite[Section 3]{cantrell2025mixing}. This uses the fact that hyperbolic groups are growth quasi-tight. In particular  if we let $\G_\Cc$ denote the collection of group elements that correspond to a path in $\Cc$, i.e. the collection of elements in $g \in \G$ for which there exists $w \in \W(\Cc)$ with $\ev(w) = g$. Then there exists a finite set $F \subset \G$ such that $F \G_\Cc F = \G$ \cite[Lemma 2.16]{cantrell2025mixing}.

We now note that $\theta_{d_\ast/d}(t)$ is the abscissa of convergence 
\[
\sum_{g \in \G} e^{-td(o,g) - sd_\ast(o,g)} =: Q(s,t)
\]
as $s$ varies and $t$ is fixed. 

Now, fix $R > 1$. By growth quasi-tightness, there is $N \ge 1$ such that for any $g \in \G$ there exists $w_g \in \W(\Cc)$ such that  $\ev(w_g) = a_1ga_2$ for some $a_1, a_2 \in F$. Furthermore, the association $g \mapsto w$ is at most $|F|^2$ to $1$. The quasimorphism property then implies that there exists a positive real number $C_{s,t} > 0$ (depending only on $s,t$) such that
\begin{align*}
Q(s,t) = \sum_{x \in \G} e^{-td(o,g) - sd_\ast(o,g)} &\le |F|^2 C_{s,t} \sum_{w \in \W(\Cc)} e^{-t\p_d(w) - s\p_{d_\ast}(w)}\\
 &\le |F|^2 C_{s,t}  \sum_{n=1}^\infty \sum_{w \in \W(\Cc): |\p_d(w) - n| < R} e^{-t\p_d(w) - s\p_{d_\ast}(w)} =: P_1(s,t)
\end{align*}
when $P_1$ converges at $(s,t)$.

We now note that, since paths in the $\Cc$ are geodesics for $S$, for any fixed cylinder $[x_0] \subset \Sigma_\Cc$, there exists $C_R >0$ depending only on $R$ such that
\[
P_2(s,t):= \sum_{n=1}^\infty \sum_{w \in [x_0], |\p_d(w) - n| < R} e^{-t\p_d(w) - s\p_{d_\ast}(w)}  \le C_R \, Q(s,t)
\]
when $Q(s,t)$ converges.

Now, by Lemma \ref{lem.expgrow}, if $R$ is sufficiently large, the abscissa of convergence of both $P_1(s,t)$ and $P_2(s,t)$ in $s$ for fixed $t$ is equal to $\Pr(\p_{d_\ast}, \p_d)(-t)$. It follows that the abscissa of convergence of $Q(s,t)$ in $s$ for fixed $t$ is precisely $\Pr(\p_{d_\ast}, \p_d)(-t)$. That is,  $= \Pr(\p_{d_\ast}, \p_d)(-t) = \theta_{d_\ast/d}(t)$ as required. 
\end{proof}

\begin{remark}\label{rem.out}
When $\G$ is a surface group. respectively free group, then Teichm\"uller space, respectively Outer space, is a natural subset of $\Dc_\G$. Then Theorem \ref{thm.D_Gembed} justifies our comments made in the introduction: Techm\"uller space and Outer space naturally embed into a quasimorphism moduli space (for an appropriate choice of subshift).
\end{remark}


\printbibliography

@article{Cuneo,
  title={Additive, Almost Additive and Asymptotically Additive Potential Sequences Are Equivalent: N. Cuneo},
  author={Cuneo, No{\'e}},
  journal={Communications in Mathematical Physics},
  volume={377},
  number={3},
  pages={2579--2595},
  year={2020},
  publisher={Springer}
}

@book {Heinonen,
    AUTHOR = {Heinonen, Juha},
     TITLE = {Lectures on {A}nalysis on {M}etric {S}paces},
    SERIES = {Universitext},
 PUBLISHER = {Springer-Verlag, New York},
      YEAR = {2001},
     PAGES = {x+140},
}

@book {DZ,
    AUTHOR = {Dembo, Amir and Zeitouni, Ofer},
     TITLE = {Large {D}eviations {T}echniques and {A}pplications},
    SERIES = {Applications of Mathematics (New York)},
    VOLUME = {38},
   EDITION = {Second},
 PUBLISHER = {Springer-Verlag, New York},
      YEAR = {1998},
     PAGES = {xvi+396},
}

@article{HM,
title = {Algorithms and Geometry for Graph Products of Groups},
journal = {Journal of Algebra},
volume = {171},
number = {1},
pages = {230-257},
year = {1995},
author = {S. Hermiller and J. Meier}
}

@article{gttgraph,
  title={Counting problems in graph products and relatively hyperbolic groups},
  author={Gekhtman, Ilya and Taylor, Samuel J and Tiozzo, Giulio},
  journal={Israel Journal of Mathematics},
  volume={237},
  number={1},
  pages={311--371},
  year={2020},
  publisher={Springer}
}

@article{hagen,
  title={Large facing tuples and a strengthened sector lemma},
  author={Hagen, Mark},
  journal={Tunisian Journal of Mathematics},
  volume={4},
  number={1},
  pages={55--86},
  year={2022},
  publisher={Mathematical Sciences Publishers}
}

@article{hagen.2,
  title={Weak hyperbolicity of cube complexes and quasi-arboreal groups},
  author={Hagen, Mark},
  journal={Journal of Topology},
  volume={7},
  number={2},
  pages={385--418},
  year={2014},
  publisher={Oxford University Press}
}

@Article{Barreira_nonadditive,
  author  = {L. Barreira},
  title   = {A non-additive thermodynamic formalism and applications to dimension theory of hyperbolic dynamical systems},
  journal = {Ergodic Theory Dynam. Systems},
  year    = {1996},
  volume  = {16},
  number  = {5},
  pages   = {871--927},
}

@article{bray2022counting,
  title={Counting, equidistribution and entropy gaps at infinity with applications to cusped Hitchin representations},
  author={H. Bray and R. Canary and L.-Y. Kao and G. Martone},
  journal={Journal f{\"u}r die reine und angewandte Mathematik (Crelles Journal)},
  volume={2022},
  number={791},
  pages={1--51},
  year={2022},
  publisher={De Gruyter}
}

@Article{Barreira_2009,
  author  = {L. Barreira and P. Doutor},
  title   = {Almost additive multifractal analysis},
  journal = {Journal de Mathéma\-tiques Pures et Appliquées},
  year    = {2009},
  volume  = {92},
  number  = {1},
  pages   = {1--17},
  langid  = {french},
}

@Book{Barreira_book,
  author    = {L. Barreira},
  title     = {Thermodynamic Formalism and Applications to Dimension Theory},
  publisher = {Birkh{\"a}user/Springer Basel AG},
  address   = {Basel},
  series    = {Progress in Mathematics},
  number    = {294},
  year      = {2011},
  note      = {xii+297 pp.},
}

@Online{Bianchi-He,
  author  = {F. Bianchi and Y.M. He},
  journal = {Inst. Hautes {\'E}tudes Sci. Publ. Math.},
  title   = {Manhattan curves in complex dynamics and asymptotic correlation of multiplier spectra},
  year    = {2025},
  eprint      = {2508.12490},
  eprinttype  = {arXiv},
  eprintclass = {math.DS},
}

@Book{EquSta,
  author    = {R. Bowen},
  title     = {Equilibrium States and the Ergodic Theory of Anosov Diffeomorphisms},
  publisher = {Springer-Verlag},
  year      = {2008},
  series    = {Lecture Notes in Mathematics},
  volume    = {470},
  pages     = {1--73},
}

@Article{Bowen-quasicircles,
  author  = {R. Bowen},
  journal = {Inst. Hautes {\'E}tudes Sci. Publ. Math.},
  title   = {Hausdorff dimension of quasicircles},
  year    = {1979},
  number  = {50},
  pages   = {11--25},
}

@Book{Brin_2002,
  author    = {M. Brin and G. Stuck},
  title     = {Introduction to Dynamical Systems},
  publisher = {Cambridge University Press},
  year      = {2002},
}

@article{burger,
	author = {M. Burger},
	fjournal = {International Mathematics Research Notices},
	issn = {1073-7928},
	journal = {Internat. Math. Res. Notices},
	number = {7},
	pages = {217--225},
	title = {Intersection, the {M}anhattan curve, and {P}atterson-{S}ullivan theory in rank {$2$}},
	year = {1993},
	}

@Article{Burger-Ozawa-Thom,
  author  = {M. Burger and N. Ozawa and A. Thom},
  journal = {Israel J. Math.},
  title   = {On {U}lam stability},
  year    = {2013},
  volume  = {193},
  number  = {1},
  pages   = {109--129},
}

@Online{C.D-M.R,
  author      = {S. Cantrell and D. Martínez-Granado and E. Reyes},
  title       = {A geometric correspondence for reparameterizations of geodesic flows},
  year        = {2026},
  eprint      = {2605.02585},
  eprinttype  = {arXiv},
  eprintclass = {math.DS},
}

@article{cantrell.reyes.man,
	author = {S. Cantrell and  E. Reyes},
	journal = {Comment. Math. Helv.},
	title = {Manhattan geodesics and the boundary of the space of metric structures on hyperbolic groups},
	volume={100},
	numer={1},
	pages={11-59},
	year = {2025}}

@article{cantrell.tanaka.2,
author ={S. Cantrell and R. Tanaka},
title = {Invariant measures of the topological flow and measures at infinity on hyperbolic groups},
journal = {Journal of Modern Dynamics},
volume = {20},
number = {0},
pages = {215-274},
year = {2024}
}

@article{cantrell.tanaka.1,
	author = {S. Cantrell and R. Tanaka},
	journal = {Geometry and Topology},
	title = {The Manhattan curve, ergodic theory of topological flows and rigidity},
	volume={29},
	issue={4},
	pages = {1851-1288},
	year = {2025}}

@article{cantrell2025mixing,
  title={Mixing of the Mineyev flow, orbital counting and Poincar{\'e} series for strongly hyperbolic metrics: S. Cantrell},
  author={Cantrell, Stephen},
  journal={Mathematische Annalen},
  volume={392},
  number={1},
  pages={1253--1288},
  year={2025},
  publisher={Springer}
}

@Article{Cao_2008,
  author  = {Y. Cao and D.-J. Feng and W. Huang},
  title   = {The thermodynamic formalism for sub-additive potentials},
  journal = {Discrete and Continuous Dynamical Systems - A},
  year    = {2008},
  volume  = {20},
  number  = {3},
  pages   = {639--657},
}

@Online{TermoEqSt,
  author      = {P.D. Carrasco and F. Rodriguez-Hertz},
  title       = {Thermodynamic formalism for quasimorphisms:
                 bounded cohomology and statistics},
  year        = {2025},
  eprint      = {2504.06054},
  eprinttype  = {arXiv},
  eprintclass = {math.DS},
}

@Online{thermodynamiclattice,
  author      = {P.D. Carrasco and F. Rodriguez-Hertz},
  title       = {Thermodynamic formalism for quasimorphisms:
                 lattices on semisimple groups},
  year        = {2026},
  eprint      = {2603.28535},
  eprinttype  = {arXiv},
  eprintclass = {math.DS},
}

@Article{Chapman-Lubotzky,
  author  = {M. Chapman and A. Lubotzky},
  journal = {Eur. J. Math.},
  title   = {Stability of homomorphisms, coverings and cocycles {I}: equivalence},
  year    = {2025},
  volume  = {11},
  number  = {1},
  note    = {38 pp.},
}

@Article{Culler-Vogtmann,
  author  = {M. Culler and K. Vogtmann},
  title   = {Moduli of graphs and automorphisms of free groups},
  journal = {Invent. Math.},
  year    = {1986},
  volume  = {84},
  number  = {1},
  pages   = {91--119},
}

@Article{Falconer,
  author  = {K. Falconer},
  title   = {A subadditive thermodynamic formalism for mixing repellers},
  journal = {J. Phys. A},
  year    = {1988},
  volume  = {21},
  number  = {14},
  pages   = {L737--L742},
}

@Book{Farb-Margalit,
  author    = {B. Farb and D. Margalit},
  title     = {A Primer on Mapping Class Groups},
  publisher = {Princeton University Press},
  address   = {Princeton, NJ},
  series    = {Princeton Mathematical Series},
  number    = {49},
  year      = {2012},
  note      = {xiv+472 pp.},
}

@Article{Feng_2011,
  author  = {D-J. Feng},
  title   = {Equilibrium states for factor maps between subshifts},
  journal = {Advances in Mathematics},
  year    = {2011},
  volume  = {226},
  number  = {3},
  pages   = {2470--2502},
}

@Book{Frigerio,
  author    = {R. Frigerio},
  title     = {Bounded Cohomology of Discrete Groups},
  publisher = {American Mathematical Society},
  address   = {Providence, RI},
  series    = {Mathematical Surveys and Monographs},
  number    = {227},
  year      = {2017},
  note      = {xvi+193 pp.},
}

@Article{Fujiwara-Kapovich,
  author  = {K. Fujiwara and M. Kapovich},
  journal = {Geom. Funct. Anal.},
  title   = {On quasihomomorphisms with noncommutative targets},
  year    = {2016},
  volume  = {26},
  number  = {2},
  pages   = {478--519},
}

@inproceedings{furman,
	address = {Berlin},
	author = {A. Furman},
	booktitle = {Rigidity in dynamics and geometry (Cambridge, 2000)},
	pages = {149-166},
	publisher = {Springer},
	title = {Coarse-geometric perspective on negatively curved manifolds and groups},
	year = {2002}}

@article{Kao,
	author = {N. Kao},
	fjournal = {Ergodic Theory and Dynamical Systems},
	number = {7},
	pages = {1843--1874},
	title = {Manhattan curves for hyperbolic surfaces with cusps},
	volume = {40},
	year = {2020},
}

@Book{Lind-Marcus,
  author    = {D. Lind and B. Marcus},
  title     = {An Introduction to Symbolic Dynamics and Coding},
  edition   = {2},
  publisher = {Cambridge University Press},
  address   = {Cambridge},
  series    = {Cambridge Mathematical Library},
  year      = {2021},
  note      = {xix+550 pp.},
}

@Article{Livsic,
  author  = {A.N. Liv{\v{s}}ic},
  title   = {Certain properties of the homology of {$Y$}-systems},
  journal = {Mat. Zametki},
  year    = {1971},
  volume  = {10},
  pages   = {555--564},
  note    = {English translation: \emph{Math. Notes Acad. Sci. USSR} 10 (1971), 758--763},
}

@Book{PP,
  author    = {W. Parry and M. Pollicott},
  publisher = {Soci{\'e}t{\'e} Math{\'e}matique de France},
  title     = {Zeta Functions and the Periodic Orbit Structure of Hyperbolic Dynamics},
  year      = {1990},
  series    = {Ast{\'e}risque},
  number    = {187--188},
}

@article{Oregon-Reyes,
	author = {E. Reyes},
	journal = {To appear in Journal of the LMS},
	note = {arXiv:2204.12545},
	title = {The space of metric structures on hyperbolic groups},
	year = {2022}}

@Online{Reyes-Wang,
  author      = {E. Reyes and T. Wang},
  title       = {Regularity of Manhattan manifolds and exact dimensionality for relatively Anosov groups},
  year        = {2026},
  eprint      = {2607.13730},
  eprinttype  = {arXiv},
  eprintclass = {math.GR},
}

@Book{Rockafellar,
  author    = {R.T. Rockafellar},
  publisher = {Princeton University Press},
  title     = {Convex analysis. Reprint of the 1970 original},
  year      = {1997},
  address   = {Princeton, NJ},
  series    = {Princeton Landmarks Math. Princeton Paperbacks.},
  note      = {xviii+451 pp.},
}

@Book{Ruelle-book,
  author    = {D. Ruelle},
  publisher = {Cambridge University Press},
  title     = {Thermodynamic Formalism: The Mathematical Structures of Equilibrium Statistical Mechanics},
  edition   = {2},
  year      = {2004},
  address   = {Cambridge},
  series    = {Cambridge Mathematical Library},
  note      = {xx+174 pp.},
}

@article{SharpManhattan,
	author = {R. Sharp},
	fjournal = {Mathematische Zeitschrift},
	journal = {Math. Z.},
	number = {4},
	pages = {745--750},
	title = {The {M}anhattan curve and the correlation of length spectra on hyperbolic surfaces},
	volume = {228},
	year = {1998},
}

@Article{Sigmund1974,
  author  = {K. Sigmund},
  journal = {Transactions of the American Mathematical Society},
  title   = {On dynamical systems with the specification property},
  year    = {1974},
  number  = {0},
  pages   = {285--299},
  volume  = {190},
}

@Book{Ulam,
  author    = {S.M. Ulam},
  publisher = {Interscience Publishers},
  title     = {A Collection of Mathematical Problems},
  year      = {1960},
  address   = {New York--London},
  number    = {8},
  series    = {Interscience Tracts in Pure and Applied Mathematics},
  note      = {xiii+150 pp.},
}
\end{document}